\RequirePackage{snapshot}

\newif\ifsubsections
\subsectionstrue

\documentclass[]{pcmi}

\usepackage[sc]{mathpazo}          
\usepackage{eulervm}               
\usepackage[scaled=0.86]{berasans} 
\usepackage[scaled=1]{inconsolata} 
\usepackage[T1]{fontenc}

\usepackage[%
	protrusion=true,
	expansion=false,
	auto=false
	]{microtype}

\usepackage{xcolor}
\usepackage{graphicx}
\graphicspath{{./figures/}}

\ifdraft
	\definecolor{linkred}{rgb}{0.7,0.2,0.2}
	\definecolor{linkblue}{rgb}{0,0.2,0.6}
\else
	\definecolor{linkred}{rgb}{0.0,0.0,0.0}
	\definecolor{linkblue}{rgb}{0,0.0,0.0}
\fi

\usepackage{amssymb}
\usepackage{amsbsy}
\usepackage{amscd}
\usepackage{mathtools}

\usepackage[mathscr]{eucal}

\usepackage{tikz}
\usepackage{verbatim}
\usepackage{version}

\PassOptionsToPackage{hyphens}{url} 
\usepackage[
    setpagesize=false,
    pagebackref,
	pdfpagelabels=false,
    pdfstartview={FitH 1000},
    bookmarksnumbered=false,
    linktoc=all,
    colorlinks=true,
    anchorcolor=black,
    menucolor=black,
    runcolor=black,
    filecolor=black,
    linkcolor=linkblue,
	citecolor=linkblue,
	urlcolor=linkred,
]{hyperref}
\usepackage[backrefs,msc-links,nobysame]{amsrefs}

\customizeamsrefs 

\newtheorem{Theorem}[equation]{Theorem}
\newtheorem{Corollary}[equation]{Corollary}
\newtheorem{Lemma}[equation]{Lemma}
\newtheorem{Proposition}[equation]{Proposition}

\theoremstyle{definition}
\newtheorem{Definition}[equation]{Definition}
\newtheorem{Example}[equation]{Example}

\newtheorem{History}[equation]{Historical Comment}
\newtheorem{Exercise}[equation]{Exercise}

\theoremstyle{remark}
\newtheorem{Remark}[equation]{Remark}

\newenvironment{aenume}{%
  \begin{enumerate}%
  }{\end{enumerate}}

\newenvironment{NB}{
\color{red}{\bf NB}. \footnotesize
}{}

\excludeversion{NB}
\excludeversion{NB2}
\newcommand{\thmref}[1]{Theorem~\ref{#1}}
\newcommand{\secref}[1]{\S\ref{#1}}
\newcommand{\lemref}[1]{Lemma~\ref{#1}}
\newcommand{\propref}[1]{Proposition~\ref{#1}}
\newcommand{\corref}[1]{Corollary~\ref{#1}}
\newcommand{\subsecref}[1]{\S\ref{#1}}

\newcommand{\defref}[1]{Definition~\ref{#1}}
\newcommand{\remref}[1]{Remark~\ref{#1}}

\newcommand{\defeq}{\overset{\operatorname{\scriptstyle def.}}{=}}
\newcommand{\CC}{{\mathbb C}}

\newcommand{\RR}{{\mathbb R}}
\newcommand{\proj}{{\mathbb P}}
\newcommand{\CP}{\proj}

\newcommand{\SL}{\operatorname{\rm SL}}
\newcommand{\SU}{\operatorname{\rm SU}}
\newcommand{\GL}{\operatorname{GL}}

\newcommand{\grpSp}{\operatorname{\rm Sp}}
\newcommand{\algsl}{\operatorname{\mathfrak{sl}}} 

\newcommand{\gl}{\operatorname{\mathfrak{gl}}}

\newcommand{\g}{{\mathfrak g}}
\newcommand{\h}{{\mathfrak h}}
\newcommand{\Spec}{\operatorname{Spec}\nolimits}
\newcommand{\Proj}{\operatorname{Proj}\nolimits}
\newcommand{\End}{\operatorname{End}}
\newcommand{\Hom}{\operatorname{Hom}}
\newcommand{\Ext}{\operatorname{Ext}}
\newcommand{\Ker}{\operatorname{Ker}}

\newcommand{\Ima}{\operatorname{Im}}

\newcommand{\codim}{\mathop{\text{\rm codim}}\nolimits}
\newcommand{\rank}{\operatorname{rank}}

\newcommand{\tr}{\operatorname{tr}}

\newcommand{\id}{\operatorname{id}}
\newcommand{\ve}{\varepsilon}
\newcommand{\shfO}{\mathcal O}
\newcommand{\bA}{\mathbf A}
\renewcommand{\MR}[1]{}
\newcommand{\linf}{\ell_\infty}

\renewcommand{\AA}{{\mathbb A}}

\newcommand{\cL}{\mathcal L}
\newcommand{\Uh}[2][G]{{\mathcal U}_{#1}^{#2}}
\newcommand{\Bun}[2][G]{\operatorname{Bun}_{#1}^{#2}}
\newcommand{\UhL}[1]{\Uh[L]{#1}}

\newcommand{\UhP}[1]{\Uh[P]{#1}}

\newcommand{\Gi}[2][r]{{\widetilde{\mathcal U}_{#1}^{#2}}}

\newcommand\ZZ{\mathbb Z}

\newcommand{\cC}{\mathcal C}

\newcommand{\Perv}{\operatorname{Perv}}

\newcommand{\TT}{\mathbb T}

\newcommand{\bigzerou}{\lower1.7ex\hbox{\Huge 0}}
\newcommand{\normal}[1]{\mbox{\large\bf:}#1\mbox{\large\bf:}}

\newcommand{\bF}{\mathbf F}

\newcommand{\Heis}{\mathfrak{Heis}}
\newcommand{\Vir}{\mathfrak{Vir}}
\newcommand{\IC}{\operatorname{IC}}

\newcommand{\Stab}{\operatorname{Stab}}

\newcommand{\scF}{\mathscr F}

\newcommand{\cA}{\mathcal A}
\newcommand{\cR}{\mathcal R}
\newcommand{\cV}{\mathcal V}

\newcommand{\cF}{\mathcal F}

\newcommand{\scW}{\mathscr W}
\newcommand{\Leaf}{\operatorname{Leaf}}
\newcommand{\dslash}{/\!\!/}
\newcommand{\Gr}{\operatorname{Gr}}
\newcommand{\loc}{\mathrm{loc}}

\begin{document}

%
%
%
%
%
%

\title{Lectures on
perverse sheaves on instanton moduli spaces
}

%
%
\author{
Hiraku Nakajima
}
\address{
Research Institute for Mathematical Sciences,
Kyoto University, Kyoto 606-8502,
Japan
}
\email{
nakajima@kurims.kyoto-u.ac.jp
}
%
%
\subjclass[2010]{Primary 14D21; Secondary 14J60, 17B69}
\keywords{Park City Mathematics Institute}

%
%
\maketitle

%
%


\begin{NB}
\begin{verbatim}
\newcommand{\Uh}[2][G]{{\mathcal U}_{#1}^{#2}}
\newcommand{\Gi}[2][r]{\widetilde{\mathcal U}_{#1}^{#2}}
\end{verbatim}

\verb+\Uh{d}+ yields $\Uh{d}$.
\verb+\Gi{d}+ yields $\Gi{d}$.
\end{NB}

\setcounter{tocdepth}{1}
\tableofcontents

\section{Introduction}

Let $G$ be an almost simple simply-connected algebraic group over
$\CC$ with the Lie algebra $\g$. Let $\h$ be a Cartan subalgebra of
$\g$. We assume $G$ is of type $ADE$, as there arise technical issues
for type $BCFG$. (We will remark them at relevant places. See
footnotes~\ref{fnt:instanton_number}, \ref{fnt:restriction}.) At some
points, particularly in this introduction, we want to include the case
$G=\GL(r)$.
We will not make clear distinction between the case $G =
\SL(r)$ and $\GL(r)$ in the main text.

Let $G_c$ denote a maximal compact subgroup of $G$.
Our main player is
\begin{equation*}
    \Uh{d} = \text{the Uhlenbeck partial compactification}
\end{equation*}
of the moduli spaces of framed $G_c$-instantons on $S^4$ with
instanton number $d$. The framing means a trivialization of the fiber
of the $G_c$-bundle at $\infty\in S^4$. Framed instantons on $S^4$ are
also called instantons on $\RR^4$, as they extend across $\infty$ if
their curvature is in $L^2(\RR^4)$. We follow this convention. The
Uhlenbeck compactification were first considered in a differential
geometric context by Uhlenbeck, Donaldson and others, for more general
$4$-manifolds and usually without framing. Since we are considering
\emph{framed} instantons, we only get \emph{partial} compactification.

The Uhlenbeck compactification has been used to define differential
topological invariants of $4$-manifolds, i.e., Donaldson invariants,
as integral of cohomology classes over moduli spaces of instantons:
Moduli spaces are noncompact, therefore the integral may diverge. Thus
compactification is necessary to make the integral well-defined. (See
e.g., \cite{MR1079726}.)

Our point of view here is different. We consider Uhlenbeck partial
compactification of instanton moduli spaces on $\RR^4$ as objects in
\emph{geometric representation theory}. We study their intersection
cohomology groups and perverse sheaves in view of representation
theory of the affine Lie algebra of $\g$ or the closely related $\scr
W$-algebra.\footnote{If $G$ is not of type $ADE$, we need to replace
  the affine Lie algebra of $\g$ by its Langlands dual
  $\g_{\text{aff}}^\vee$. It is a twisted affine Lie algebra, and
  should not be confused with the untwisted affine Lie algebra of the
  Langlands dual of $\g$.} We will be concerned only with a very
special $4$-manifold, i.e., $\RR^4$ (or $\CC^2$ as we will use an
algebro-geometric framework). On the other hand, we will study
instantons for any group $G$, while $G_c=\SU(2)$ is usually enough for
topological applications.

We will drop `Uhlenbeck partial compactification' hereafter unless it
is really necessary, and simply say instanton moduli spaces or moduli
spaces of framed instantons, as we will keep `U' in the notation.

We will study \emph{equivariant} intersection cohomology groups of
instanton moduli spaces
\begin{equation*}
    IH^*_{T\times\CC^\times\times\CC^\times}(\Uh{d}),
\end{equation*}
where $T$ is a maximal torus of $G$ acting by the change of the
framing, and $\CC^\times\times\CC^\times$ is a maximal torus of
$\GL(2)$ acting on $\RR^4=\CC^2$. The $T$-action has been studied in
the original context: it is important to understand singularities of
instanton moduli spaces around \emph{reducible} instantons. The
$\CC^\times\times\CC^\times$-action is specific for $\RR^4$,
nevertheless it makes the tangent bundle of $\RR^4$ \emph{nontrivial},
and yields a meaningful counter part of Donaldson invariants, as
Nekrasov partition functions. (See below.)

More specifically, we will explain the author's joint work with
Braverman and Finkelberg \cite{2014arXiv1406.2381B} with an emphasis
on its geometric part in these lectures. The stable envelope introduced
by Maulik-Okounkov \cite{MO} and its reformulation in \cite{tensor2}
via Braden's hyperbolic restriction functors are key technical
tools. They also appear in other situations in geometric
representation theory. Therefore those will be explained in a general
framework. In a sense, a purpose of lectures is to explain these
important techniques and their applications.

It is not our intention to reproduce the proof of
\cite{2014arXiv1406.2381B} in full here. As lectures move on, we will
often leave details of arguments to \cite{2014arXiv1406.2381B}. The
author hopes that a reader is comfortable to read
\cite{2014arXiv1406.2381B} at that stage after learning preliminary
materials assumed in \cite{2014arXiv1406.2381B}.

\subsection*{Prerequisite}

\begin{itemize}
      \item The theory of perverse sheaves is introduced in
    de~Cataldo's lectures in the same volume. We will also use
    materials in \cite[\S8.6]{CG}, in particular the isomorphism
    between convolution algebras and Ext algebras.

      \item We assume readers are familiar with \cite{Lecture}, at
    least for Chapters 2, 3, 6. (Chapter 6 presents Hilbert-Chow
    morphisms as examples of semi-small morphisms. They also appear in
    de~Cataldo's lectures.) Results explained there will be briefly
    recalled, but proofs are omitted.

      \item We will use equivariant cohomology and Borel-Moore
    homology groups. A brief introduction can be found in
    \cite{more}. We also use derived categories of equivariant
    sheaves. See \cite{BL} (and/or \cite[\S1]{Lu-cus2} for summary).

      \item We will not review the theory of $\scr W$-algebras, such
    as \cite[Ch.~15]{F-BZ}. It is not strictly necessary, but better
    to have some basic knowledge in order to appreciate the final
    result.
\end{itemize}

\subsection*{History}

Let us explain history of study of instanton moduli spaces in
geometric representation theory.

Historically the relation between instanton moduli spaces and
representation theory of affine Lie algebras was first found by the
author in the context of quiver varieties \cite{Na-quiver}. The
relation is different from one we shall study in this paper. Quiver
varieties are (partial compactifications of) instanton moduli spaces
on $\RR^4/\Gamma$ with gauge group $G = \GL(r)$, where $\Gamma$ is a
finite subgroup of $\SL(2)$. The corresponding affine Lie algebra is
one attached to $\Gamma$ via the McKay correspondence, not the affine
Lie algebra of $\g$ as in this paper.
Moreover the argument in \cite{Na-quiver} works only for $\Gamma\neq
\{1\}$.
The case $\Gamma=\{1\}$ corresponds to the Heisenberg algebra, that is
the affine Lie algebra for $\gl(1)$. The result for $\Gamma=\{1\}$ was
obtained later by Grojnowski and the author independently
\cite{MR1386846,MR1441880} for $r=1$, Baranovsky \cite{Baranovsky} for
general $r$.
This result will be recalled in \secref{sec:Heis}, basically for the
purpose to explain why they were {\it not enough\/} to draw a full
picture.

In the context of quiver varieties, a $\CC^\times$-action naturally
appears from an action on $\RR^4/\Gamma$. (If $\Gamma$ is of type $A$,
we have $\CC^\times\times\CC^\times$-action as in these lectures.) 
The equivariant $K$-theory of quiver varieties are related to
representation theory of quantum toroidal algebras, where $\CC^\times$
appears as a quantum parameter $q$. More precisely the representation
ring of $\CC^\times$ is the Laurent polynomial ring $\ZZ[q,q^{-1}]$.
See \cite{Na-qaff}. The corresponding result for equivariant
homology/affine Yangian version, which is closer to results explained
in this lecture series was obtained by Varagnolo \cite{Varagnolo}. But
these works covered only the case $\Gamma\neq\{ 1\}$.
It is basically because the construction relies on a particular
presentation of quantum toroidal algebras and affine Yangian, which
was available only for $\Gamma\neq\{1\}$ when the paper \cite{Na-qaff}
was written.

In physics side, Nekrasov \cite{Nekrasov} introduced `partition
functions' which are roughly considered as generating functions of
equivariant Donaldson invariants on $\RR^4$ with respect to the
$\CC^\times\times\CC^\times$-action. The ordinary Donaldson invariants
do not make sense for $\RR^4$ (or $S^4$) as there is no interesting
topology on $\RR^4$. But equivariant Donaldson invariants are
nontrivial, and contain interesting information. Nekrasov partition
functions have applications to ordinary Donaldson invariants, see
e.g., \cite{GNY1,GNY3}. Roughly equivariant variables for
$\CC^\times\times\CC^\times$ play the role of Chern roots of the
tangent bundle of a complex surface. Important parts of Donaldson
invariants (namely wall-crossing terms and coefficients relating
Donaldson and Seiberg-Witten invariants) are universal in Chern
classes, hence it is enough to compute for $\RR^4$.

Another conjectural connection between affine Lie algebras and
cohomology groups of instanton moduli spaces on $\RR^4/\ZZ_\ell$,
called the geometric Satake correspondence for affine Kac-Moody
groups, was found in \cite{braverman-2007}.\footnote{The first version
  of the preprint was posted to arXiv in Nov.\ 2007. It was two years
  before \cite{AGT} was posted.} Recall the affine Grassmannian $\Gr_G
= G((z))/G[[z]]$ for a finite dimensional complex reductive group has
a $G[[z]]$-action whose orbits are parametrized by dominant coweights
$\lambda$. The geometric Satake correspondence roughly says $\IC$ of
the $G[[z]]$-orbit through $\lambda$ \emph{knows} the irreducible
representation $V(\lambda)$ with highest weight $\lambda$ of the
Langlands dual group $G^\vee$.
The affine Grassmannian of an affine Kac-Moody group is difficult to
make sense as a geometric object, but instanton moduli spaces on
$\RR^4/\ZZ_\ell$ play the role of a slice to an orbit in the closure
of another larger orbit.
In this connection, the $\CC^\times\times\CC^\times$-action is
discarded, but the intersection cohomology is considered. The main
conjecture says the graded dimension of the intersection cohomology
groups can be computed in terms of the $q$-analog of weight
multiplicities of integrable representations of the affine Lie algebra
of $\g$ with level $\ell$. (See \cite{MR2855083} also for an important
correction of the conjectural definition of the filtration.) We will
visit a simpler variant of the conjecture in
\subsecref{subsec:affine-grassm-an-affine}.

When an affine Kac-Moody group is of affine type $A$, the instanton
moduli space is a quiver variety, and has the Gieseker partial
compactification, as a symplectic resolution of singularities. Then
one can work on ordinary homology/K-theory of the resolution. There is
a close relation between ordinary homology of the symplectic
resolution and intersection cohomology. (See \eqref{eq:11}.)
In particular, the above conjectural dimension formula was checked for
type $A$ (after discarding the grading) in \cite{Na-branching}. A key is
that the quiver variety is already linked with an affine Lie algebra,
as we have recalled above. More precisely, for $G = \SL(r)$, the
quiver variety construction relates the $IH$ of an instanton moduli
space to $(\algsl_{\ell})_{\text{aff}}$ of level $r$, while the
conjectural geometric Satake correspondence relates it to
$(\algsl_{r})_{\text{aff}}$ of level $\ell$. They are related by
I.Frenkel's level-rank duality, and the weight multiplicities are
replaced by tensor product multiplicities.

We still lack a precise (even conjectural) understanding of $IH$ of
instanton moduli spaces for general $G$ and general $\RR^4/\Gamma$,
not necessarily of type $A$, say both $G$ and $\Gamma$ are of type
$E_8$. It is a good direction to pursue in future.

Despite its relevance for the study of Nekrasov partition function,
the equivariant $K$-theory and homology group for the case $\Gamma =
\{1\}$ were not understood for several years. They were more difficult
than the case $\Gamma\neq\{1\}$, because analogs of quantum
toroidal algebras and affine Yangian for $\Gamma=\{1\}$ were not known
as mentioned above. In the context of the geometric Satake
correspondence, $\Gamma = \{1\}$ corresponds to level $1$
representations, which have explicit constructions from the Fock space
(Frenkel Kac construction), and the conjectural dimension formula was
checked \cite{braverman-2007,MR2855083}.  But the role of
$\CC^\times\times\CC^\times$-action was not clear.

In 2009, Alday-Gaiotto-Tachikawa \cite{AGT} has connected Nekrasov
partition for $G=\SL(2)$ with the representation theory of the
Virasoro algebra via a hypothetical $6$-dimensional quantum field
theory. This AGT correspondence is hard to justify in a mathematically
rigorous way, but yet gives a very good view point. In particular, it
predicts that the equivariant intersection cohomology of instanton
moduli space is a representation of the Virasoro algebra for
$G=\SL(2)$, and of the $\scr W$-algebra associated with $\g$ in
general. See \subsecref{subsec:AGT} for a short summary of the AGT
correspondence.

On the other hand, in mathematics side, equivariant $K$ and homology
groups for the case $\Gamma=\{1\}$ had been understood gradually
around the same time.
Before AGT,\footnote{Preprints of those papers were posted to arXiv,
  slightly before \cite{AGT} was appeared on arXiv.}
Feigin-Tsymbaliuk and Schiffmann-Vasserot \cite{MR2854154,MR3018956}
studied the equivariant $K$-theory for the $G=\GL(1)$-case. Then
research was continued under the influence of the AGT correspondence,
and the equivariant homology for the $G=\GL(r)$-case was studied by
Schiffmann-Vasserot, Maulik-Okounkov \cite{MR3150250,MO}.

In particular, the approach taken in \cite{MO} is considerably
different from previous ones. It does not use a particular
presentation of an algebra. Rather it constructs the algebra action
from the $R$-matrix, which naturally arises on equivariant homology
group. This sounds close to a familiar $RTT$ construction of Yangians
and quantum groups, but it is more general: the $R$-matrix is
constructed in a purely geometric way, and has infinite size in
general. Also for $\Gamma\neq\{1\}$, it defines a coproduct on the
affine Yangian. It was explicitly constructed for the usual Yangian
for a finite dimensional complex simple Lie algebra long time ago by
Drinfeld \cite{Drinfeld}, but the case of the affine Yangian was new.
\footnote{It motivated the author to define the coproduct in terms of
  standard generators in his joint work in progress with Guay.}

In \cite{tensor2}, the author reformulated the stable envelope, a
geometric device to produce the $R$-matrix in \cite{MO}, in a sheaf
theoretic language, in particular using Braden's hyperbolic
restriction functors \cite{Braden}. This reformulation is necessary in
order to generalize the construction of \cite{MO} from $G=\GL(r)$ to
other $G$. It is because the original formulation of the stable
envelope required a symplectic resolution. We have a symplectic
resolution of $\Uh{d}$ for $G=\GL(r)$, as a quiver variety, but not
for general $G$.

Then the author together with Braverman, Finkelberg
\cite{2014arXiv1406.2381B} studies the equivariant intersection
cohomology of the instanton moduli space, and constructs the $\scr
W$-algebra representation on it. Here the geometric Satake
correspondence for the affine Lie algebra of $\g$ gives a
philosophical background: the reformulated stable envelope is used to
realize the restriction to the affine Lie algebra of a Levi subalgebra
of $\g$. It nicely fits with Feigin-Frenkel description of the $\scr
W$-algebra \cite[Ch.~15]{F-BZ}.

\subsection*{Convention}\label{convention}

\begin{enumerate}
      \item A partition $\lambda$ is a nonincreasing sequence
    $\lambda_1\ge\lambda_2\ge \cdots$ of nonnegative integers with
    $\lambda_N = 0$ for sufficiently large $N$. We set $|\lambda| =
    \sum \lambda_i$, $l(\lambda) = \# \{ i \mid \lambda_i \neq
    0\}$. We also write $\lambda = (1^{\alpha_1} 2^{\alpha_2}\cdots)$
    with $\alpha_k = \# \{ i\mid \lambda_i = k\}$.

      \item For a variety $X$, let $D^b(X)$ denote the
    bounded derived category of complexes of constructible
    $\CC$-sheaves on $X$.
    Let $\IC(X,\mathcal L)$ denote the intersection cohomology complex
    associated with a local system $\mathcal L$ over a Zariski open dense
    subvariety $X_0$ in the smooth locus of $X$. We denote it also by
    $\IC(X)$ if $\mathcal L$ is trivial.
  When $X$ is smooth and irreducible, $\cC_X$ denotes the constant
  sheaf on $X$ shifted by $\dim X$. If $X$ is a disjoint union of
  irreducible smooth varieties $X_\alpha$, we understand $\cC_X$ as
  the direct sum of $\cC_{X_\alpha}$.

    \item We make a preferred degree shift for cohomology and
  Borel-Moore homology groups, and denote them by $H^{[*]}(X)$ and
  $H_{[*]}(X)$, where $H^{[*]}(X) = H^{*+\dim X}(X)$, $H_{[*]}(X) =
  H_{*+\dim X}(X)$ for a smooth variety $X$. The same convention is
  applied for cohomology groups with compact support. More generally,
  if $L$ is a closed subvariety in a smooth variety $X$, we consider
  $H_{[*]}(L) = H_{*+\dim X}(L)$.
\end{enumerate}


\subsection*{Coulomb branches}

About a few month before the author delivered lectures, he found a
mathematical approach to the so-called \emph{Coulomb branches} of
$3$-dimensional SUSY gauge theories \cite{2015arXiv150303676N}. Such a
gauge theory is associated with a pair $(H_c,\mathbf M)$ of a compact
Lie group $H_c$ and its quaternionic representation $\mathbf M$.
Coulomb branches are hyper-K\"ahler manifolds with $\SU(2)$-action
possibly with singularities. They have been studied in physics for
many years. But physical definition contains \emph{quantum
  corrections}, which were difficult to justify in a mathematically
rigorous way.
On the other hand, the so-called \emph{Higgs branches} of gauge
theories are hyper-K\"ahler quotients of $\mathbf M$ by $H_c$. This is
a mathematically rigorous definition, and quiver varieties mentioned
above are examples of Higgs branches of particular $(H_c,\mathbf M)$,
called (framed) quiver gauge theories.

At the time these notes are written, the author together with
Braverman and Finkelberg have established a rigorous definition of
Coulomb branches as affine algebraic varieties with holomorphic
symplectic structures under the condition
$\mathbf M = \mathbf N\oplus\mathbf N^*$ for some $\mathbf N$
\cite{main}. The spaces $\Uh{d}$ discussed in these notes, and even
more generally moduli spaces of framed $G_c$-instantons on
$\RR^4/\ZZ_\ell$ conjecturally appear as Coulomb branches of
$(H_c,\mathbf M)$.
The gauge theories are framed quiver gauge theories of affine type
$ADE$, whose Higgs branches are quiver varieties of affine types,
i.e., moduli spaces of $\GL(r)$-instantons on $\RR^4/\Gamma$. A proof
of this conjecture for type $A$ is given in \cite{bowCoulomb}.

Thus Coulomb branches open a new way to approach to instanton moduli
spaces. For examples, their quantization, i.e., noncommutative
deformation of the coordinate rings naturally arise from the
construction. We also obtain variants of $\Uh{d}$, arising from the
same type of framed quiver gauge theory, for each weight
$\mu\le\Lambda_0$ where $\Lambda_0$ is the $0$th fundamental weight of
the affine Lie algebra for $G$. Here $\Uh{d}$ corresponds to $\mu =
\Lambda_0 - d\delta$, where $\delta$ is the primitive positive
imaginary root. In fact, the Coulomb branch of this category is
conjecturally an instanton moduli space of Taub-NUT space. The
Taub-NUT space is a hyper-K\"ahler manifold which is isomorphic to
$\CC^2$ as a holomorphic symplectic manifold. The Riemannian metric is
different from the Euclidean metric: The size of fibers of the Hopf
fibration at infinity remains bounded. There could be nontrivial
monodromy on fibers at infinity, which is the extra data cannot be
seen in $\Uh{d}$.

Moreover Higgs and Coulomb branches are tightly connected in the
author's approach: The construction of Coulomb branches uses moduli
spaces of twisted maps from $\CP^1$ to Higgs branches. These links
should fit with \emph{symplectic duality} proposed in
\cite{2014arXiv1407.0964B}. The intersection cohomology groups of
$\Uh{d}$ studied in this paper should be looked at from this view
point. See \remref{rem:duality} for an example.

Thus it is clear that the whole stories talked here must be redone for
instanton moduli spaces on Taub-NUT space. The author would
optimistically hope that they are similar, and our lectures still
serve as basics.

\subsection*{Acknowledgments}

The notes are based on the lectures I gave at the Park City
Mathematical Institute in Summer of 2015. I thank the PCMI for its
warm hospitality.

I am grateful to A.~Braverman and M.~Finkelberg for the collaboration
\cite{2014arXiv1406.2381B}, on which the lecture series is based.
I still vividly remember when Sasha said that symplectic resolution
exists only in rare cases, hence we need to study singular spaces. It
was one of motivations for me to start instanton moduli spaces for
general groups, eventually it led me to the collaboration.

Prior to the PCMI, lectures were delivered at University of Tokyo,
September 2014. I thank A.~Tsuchiya for invitation.

I would also thank the referee for her/his careful reading and various
suggestions.

This research is supported by JSPS Kakenhi Grant Numbers 
23224002, 
23340005, 
24224001. 

\section{Uhlenbeck partial compactification -- in brief}

We do not give a precise definition of Uhlenbeck partial
compactification of an instanton moduli space on $\RR^4$, which was
introduced in \cite{BFG}. For type $A$, there is an earlier
alternative construction as an example of a quiver variety, as
reviewed in \subsecref{subsec:Gieseker}.

For our purpose, we will only use properties explained in this
section. These are easy to check in the quiver variety construction,
hence the reader should accept them at the first reading.

\subsection{Moduli space of bundles}

First of all, $\Uh{d}$ is a partial compactification of $\Bun{d}$, the
moduli space of framed holomorphic (or equivalently algebraic)
$G$-bundles $\mathcal F$ over $\proj^2$ of instanton number $d$,
where the framing is the trivialization $\varphi$ of the restriction
of $\mathcal F$ at the line $\linf$ at infinity. It will be denoted by
$\Uh{d}$ throughout these lectures. It is an affine algebraic variety.

Here the instanton number is the integration over $\proj^2$ of the
characteristic class of $G$-bundles corresponding to the invariant
inner product on $\g$, normalized as $(\theta,\theta)=2$ for the
highest root $\theta$.\footnote{\label{fnt:instanton_number}When an
  embedding $\SL(2)\to G$ corresponding to a coroot $\alpha^\vee$ is
  gives, we can induction a $G$-bundle $\cF$ from a $\SL(2)$-bundle
  $\cF_{\SL(2)}$ Then we have $d(\cF) = d(\cF_{\SL(2)})\times
  (\alpha^\vee,\alpha^\vee)/2$. Instanton numbers are preserved if $G$
  is type $ADE$, but not in general.}

\begin{Remark}
By an analytic result due to Bando \cite{Bando}, we can replace
instanton moduli spaces on $\RR^4$ by $\Bun{d}$. We use
algebro-geometric approaches to instanton moduli spaces hereafter.
\end{Remark}

We also use that $\Bun{d}$ is a smooth locus of $\Uh{d}$, and its
tangent space at $(\cF,\varphi)$ is
\begin{equation*}
    H^1(\proj^2, \g_\cF(-\linf)),
\end{equation*}
where $\g_\cF$ is the associated vector bundle $\cF\times_G \g$.

In the quiver variety construction, $\Bun{d}$ is defined as the space
of linear maps satisfying a quadratic equation modulo the action of
$\GL(d)$. The tangent space is the quotient of the derivative the
defining equation modulo the differential of the action. Properties of
tangent spaces which we will use can be equally checked by using such
a description.

\subsection{Stratification}

We will use its stratification:
\begin{equation}\label{eq:6}
   \Uh{d} = \bigsqcup_{0\le d'\le d} \Bun{d'} \times S^{d-d'}(\CC^2),
\end{equation}
where $\Bun{d'}$ is the moduli space 
with smaller instanton number $d'$.

Let us refine the stratification in the symmetric product $S^{d-d'}(\CC^2)$ part:
\begin{equation}\label{eq:7}
   \Uh{d} =  \bigsqcup_{d=|\lambda|+d'} \Bun[G]{d'} \times S_\lambda(\CC^2),
\end{equation}
where $S_\lambda(\CC^2)$ consists of configurations of points in $\CC^2$ whose multiplicities are given by the partition $\lambda$.

\subsection{Factorization morphism}

We will also use the factorization morphism
\begin{equation*}
    \pi^d\colon \Uh{d}\to S^d\CC^1.
\end{equation*}
The definition is given in \cite[\S6.4]{BFG}. It depends on a choice
of the projection $\CC^2\to\CC^1$. We do not recall the definition,
but its crucial properties are
\begin{enumerate}
      \item On the factor $S^{d-d'}(\CC^2)$, it is given by the
    projection $\CC^2\to\CC^1$.
      \item Consider $C_1 + C_2\in S^d\CC^1$ such that $C_1\in
    S^{d_1}\CC^1$, $C_2\in S^{d_2}\CC^1$ are disjoint. Then
    $(\pi^d)^{-1}(C_1+C_2)$ is isomorphic to
    $(\pi^{d_1})^{-1}(C_1)\times (\pi^{d_2})^{-1}(C_2)$.
\end{enumerate}

Intuitively $\pi^d$ is given as follows. By (1), it is enough to
consider the case of {\it genuine\/} framed $G$-bundles
$(\cF,\varphi)$. Let $x\in \CC^1$ and consider the line $\proj^1_x =
\{z_1 = x z_0\}$ in $\proj^2$. If $x=\infty$, we can regard
$\proj^1_x$ as the line $\linf$ at infinity. If we restrict the
$G$-bundle $\cF$ to $\linf$, it is trivial. Since the triviality is an
open condition, the restriction $\cF|_{\proj^1_x}$ is also trivial
except for finitely many $x\in\CC^1$, say $x_1$, $x_2\dots, x_k$. Then
$\pi^d(\cF,\varphi)$ is the sum $x_1+x_2+\dots+x_k$ if we assign the
multiplicity appropriately.

\subsection{Symplectic resolution}

As we will see in the next section, $\Uh{d}$ for $G=\SL(r)$ has a
symplectic resolution. On the other hand, $\Uh{d}$ for $G\neq\SL(r)$
\emph{does not} have a symplectic resolution. This can be checked as
follows. First consider the case $d=1$. It is well known that $\Uh{1}$
is the product of $\CC^2$ and the closure of the minimal nilpotent
orbit in $\g$.\footnote{The author does not know who first noticed
  this fact, and even where a proof is written. In \cite{AHS} it was
  shown that any $G_c$ $1$-instanton is reduced to an $\SU(2)$
  1-instanton, and the instanton moduli space (but not framed one) for
  $\SU(2)$ is determined. But this statement itself is not stated.} Fu
proved that the minimal nilpotent orbit has no symplectic resolution
except $G = \SL(r)$ \cite{MR1943745}. Next consider $d > 1$. Take the
second largest stratum $\Bun{d-1}\times\CC^2$. By the factorization,
$\Uh{d}$ is isomorphic to a product of a smooth variety and $\Uh{1}$
in a neighborhood of a point in the stratum. Fu's argument is local
(the minimal nilpotent orbit has an isolated singular point), hence
$\Uh{d}$ does not have a symplectic resolution except $G=\SL(r)$.

\subsection{Group action}\label{subsec:group-action}

We have an action of a group $G$ on $\Uh{d}$ by the change of
framing. We also have an action of $\GL(2)$ on $\Uh{d}$ induced from
the $\GL(2)$ action on $\CC^2$. Let $T$ be a maximal torus of $G$.
In these notes, we only consider the action of the subgroup $T\times
\CC^\times\times\CC^\times$ in $G\times \GL(2)$. Let us introduce the
following notation: $\TT = T\times\CC^\times\times\CC^\times$.

Our main player will be the equivariant intersection cohomology group
$IH^*_\TT(\Uh{d})$. It is a module over
\begin{equation}\label{eq:30}
    H^*_\TT(\mathrm{pt}) \cong \CC[\operatorname{Lie}\TT]
    \cong \CC[\ve_1,\ve_2,\vec{a}],
\end{equation}
where $\ve_1$, $\ve_2$ (resp.\ $\vec{a}$) are coordinates (called
equivariant variables) on $\operatorname{Lie}(\CC^\times\times
\CC^\times)$ (resp.\ $\operatorname{Lie} T$).

We also use the notation
\begin{equation*}
    \bA_T = \CC[\operatorname{Lie}\TT] = \CC[\vec{a},\ve_1,\ve_2],\qquad
    \bA = \CC[\ve_1,\ve_2].
\end{equation*}
Their quotient fields are denoted by
\begin{equation*}
    \bF_T = \CC(\operatorname{Lie}\TT), \qquad
    \bF = \CC(\ve_1,\ve_2).
\end{equation*}

\section{Heisenberg algebra action on Gieseker partial compactification}\label{sec:Heis}

This lecture is an introduction to the actual content of subsequent
lectures. We consider instanton moduli spaces when the gauge group $G$
is $\SL(r)$. We will explain results about (intersection) cohomology
groups of instanton moduli spaces known before the AGT correspondence
was found. Then it will be clear what was lacking, and readers are
motivated to learn more recent works.

\subsection{Gieseker partial compactification}\label{subsec:Gieseker}

When the gauge group $G$ is $\SL(r)$, we denote the corresponding
Uhlenbeck partial compactification $\Uh{d}$ by $\Uh[r]{d}$.

For $\SL(r)$, we can consider a modification $\Gi{d}$ of $\Uh[r]{d}$,
called the {\it Gieseker partial compactification}. It is a moduli
space of framed torsion free sheaves $(E,\varphi)$ on $\proj^2$, where
the framing $\varphi$ is a trivialization of the restriction of $E$ to
the line at infinity $\linf$. It is known that $\Gi{d}$ is a smooth
(holomorphic) symplectic manifold. It is also known that there is a
projective morphism $\pi\colon\Gi{d}\to\Uh[r]{d}$, which is a
resolution of singularities.

When $r=1$, the group $\SL(1)$ is trivial. But the Giesker space is
nontrivial: $\Gi[1]{d}$ is the Hilbert scheme $(\CC^2)^{[d]}$ of $d$
points on the plane $\CC^2$, and the Uhlenbeck partial
compactification\footnote{Since $\Bun[\GL(1)]{d} = \emptyset$ unless
  $d=0$, this is a confusing name.} $\Uh[1]{d}$ is the
$d^{\mathrm{th}}$ symmetric product $S^d(\CC^2)$ of $\CC^2$. The
former parametrizes ideals $I$ in the polynomial ring $\CC[x,y]$ of
two variables with colength $d$, i.e., $\dim \CC[x,y]/I = d$. The
latter is the quotient of the Cartesian product $(\CC^2)^d$ by the
symmetric group $S_d$ of $d$ letters. It parametrized $d$ unordered
points in $\CC^2$, possibly with multiplicities. We will use the
summation notation like $p_1 + p_2 + \dots + p_d$ or $d\cdot p$ to
express a point in $S^d(\CC^2)$.

\begin{figure}[htbp]
    \centering
\begin{tikzpicture}[
    node distance=40pt,auto]
\node (V) {$\CC^d$};
\node (W) [below of=V] {$\CC^r$};
\path[->] (V) edge [thick,in=345,out=15,loop] node {$B_2$} ()
          (V) edge [thick,in=195,out=165,loop]  node [left] {$B_1$} ();
\path[->] (W) edge [thick,bend left=20] node {$I$} (V)
          (V) edge [thick,bend left=20] node {$J$} (W);
\end{tikzpicture}
    \caption{Quiver varieties of Jordan type}
    \label{fig:Jordan}
\end{figure}

For general $r$, these spaces can be understood as quiver varieties
associated with Jordan quiver. It is not my intention to explain the
theory of quiver varieties, but here is the definition in this case:
Take two complex vector spaces of dimension $d$, $r$ respectively. Consider linear maps $B_1$, $B_2$, $I$, $J$ as in Figure~\ref{fig:Jordan}. We impose the equation
\begin{equation*}
   \mu(B_1,B_2,I,J) \defeq [B_1,B_2]+IJ = 0
\end{equation*}
Then we take two types of quotient of $\mu^{-1}(0)$ by $\GL(d)$. The
first one corresponds to $\Uh[r]{d}$, and is the affine
algebro-geometric quotient $\mu^{-1}(0)\dslash\!\GL(d)$. It is defined
as the spectrum of $\CC[\mu^{-1}(0)]^{\GL(d)}$, the ring of
$\GL(d)$-invariant polynomials on $\mu^{-1}(0)$. Set-theoretically it
is the space of closed $\GL(d)$-orbits in $\mu^{-1}(0)$.
The second quotient corresponds to $\Gi{d}$, and is the geometric invariant theory quotient with respect to the polarization given by the determinant of $\GL(d)$. It is $\Proj$ of $\bigoplus_{n\ge 0}
\CC[\mu^{-1}(0)]^{\GL(d),\det^n}$, the ring of $\GL(d)$-semi-invariant polynomials. Set-theoretically it is the quotient of {\it stable\/} points in $\mu^{-1}(0)$ by $\GL(d)$. Here $(B_1,B_2,I,J)$ is stable if there is no proper subspace $T$ of $\CC^d$ which is invariant under $B_1$, $B_2$ and is containing the image of $I$.

\begin{Theorem}
    We have isomorphisms
    \begin{equation*}
        \Uh[r]{d} \cong \mu^{-1}(0)\dslash\!\GL(d), \qquad
        \Gi[r]{d} \cong \Proj\left(\bigoplus_{n\ge 0}
        \CC[\mu^{-1}(0)]^{\GL(d),\det^n}\right).
    \end{equation*}
\end{Theorem}

Let us briefly recall how those linear maps determine points in $\Gi{d}$ and $\Uh[r]{d}$. The detail was given in \cite[Ch.~2]{Lecture}. Given $(B_1,B_2,I,J)$, we consider the following complex
\begin{equation*}
    \shfO^{\oplus d}(-1) \xrightarrow{a}
    \begin{matrix}
        \shfO^{\oplus d} \\ \oplus \\ \shfO^{\oplus d} \\ \oplus
        \\ \shfO^{\oplus r}
    \end{matrix}
    \xrightarrow{b} \shfO^{\oplus d}(1),
\end{equation*}
where
\begin{equation*}
    a =
    \begin{pmatrix}
        z_0 B_1 - z_1 \\ z_0 B_2 - z_2 \\ z_0 J
    \end{pmatrix},
    \qquad
    b =
    \begin{pmatrix}
        - (z_0 B_2 - z_2) & z_0 B_1 - z_1 & z_0 I
    \end{pmatrix}.
\end{equation*}
Here $[z_0:z_1:z_2]$ is the homogeneous coordinate system of $\proj^2$
such that $\linf = \{ z_0 = 0\}$. The equation $\mu=0$ guarantees that
this is a complex, i.e., $ba = 0$. One sees easily that $a$ is
injective on each fiber over $[z_0:z_1:z_2]$ except for finitely many.
The stability condition implies that $b$ is surjective on each
fiber. It implies that $E \defeq \Ker b/\Ima a$ is a torsion free
sheaf of rank $r$ with $c_2 = d$. Considering the restriction to $z_0
= 0$, one sees that $E$ has a canonical trivialization $\varphi$
there. Thus we obtain a framed sheaf $(E,\varphi)$ on $\proj^2$. One
also see that $a$ is injective on any fiber if and only if $E$ is a
locally free sheaf, i.e., a vector bundle.

From this description, we can check the stratification
\eqref{eq:6}. If $(B_1,B_2,I,J)$ has a closed $\GL(d)$-orbit, it is
{\it semisimple}, i.e., a `submodule' (in appropriate sense) has a
complementary submodule. Thus $(B_1,B_2,I,J)$ decomposes into a direct
sum of {\it simple\/} modules, which do not have nontrivial
submodules. There is exactly one simple summand with nontrivial $I$,
$J$, and all others have $I=J=0$. The former gives a point in
$\Bun{d'}$, as one can check that $a$ is injective in this case. The
latter is a pair of commuting matrices $[B_1, B_2] = 0$, and the
simplicity means that the size of matrices is $1$. Therefore the
simultaneous eigenvalues give a point in $\CC^2$.

\begin{Exercise}
    (a) We define the factorization morphism $\pi^d$ for $G=\SL(r)$ in
    terms of $(B_1,B_2,I,J)$. Let $\pi^d([B_1,B_2,I,J])\in S^d\CC$ be
    the spectrum of $B_1$ counted with multiplicities. Check that
    $\pi^d$ satisfies the properties (1),(2) above.

    (b) 
    Check that $\cF|_{\proj^1_x}$ is trivial if
    $B_1 - x$ is invertible.

    More generally one can define the projection as the spectrum of
    $a_1 B_1 + a_2 B_2$ for $(a_1,a_2)\in\CC^2\setminus\{0\}$, but it
    is enough to check this case after a rotation by the $\GL(2)$-action.
\end{Exercise}

\subsection{Tautological bundle}\label{subsec:taut}

Let $\cV$ be the tautological bundle over $\Gi{d}$. It is a rank $d$
vector bundle whose fiber at a framed torsion free sheaf $(E,\varphi)$
is $H^1(\CP^2,E(-\linf))$. In the quiver variety description, it is
the vector bundle associated with the principal $\GL(d)$-bundle
$\mu^{-1}(0)^{\text{stable}}\to \Gi{d}$. For Hilbert schemes of
points, parametrizing ideals $I$ in $\CC[x,y]$, the fiber at $I$ is
$\CC[x,y]/I$.

\subsection{Gieseker-Uhlenbeck morphism}


Recall that a (surjective) projective morphism $\pi\colon M\to X$ from
a nonsingular variety $M$ is {\it semi-small\/} if $X$ has a
stratification $X = \bigsqcup X_\alpha$ such that
$\pi|_{\pi^{-1}(X_\alpha)}$ is a topological fibration, and $\dim
\pi^{-1}(x_\alpha) \le \frac12 \codim X_\alpha$ for $x_\alpha\in
X_\alpha$.

\begin{Proposition}[\protect{\cite{Baranovsky},
      \cite[Exercise~5.15]{Lecture}}]
\label{prop:semismall}
    $\pi\colon \Gi{d}\to \Uh[r]{d}$ is semi-small with respect to the
    stratification \eqref{eq:7}. Moreover the fiber $\pi^{-1}(x)$ is
    irreducible\footnote{The (solution of) exercise only shows there
      is only one irreducible component of $\pi^{-1}(x_\alpha)$ with
      dimension $\frac12\codim X_\alpha$. The irreducibility was
      proved by Baranovsky and Ellingsrud-Lehn.} at any point
    $x\in\Uh[r]{d}$.
\end{Proposition}

This semi-smallness result is proved for general symplectic
resolutions by Kaledin \cite{MR2283801}.

\cite[Exercise~5.15]{Lecture} asks the dimension of the central fiber
$\pi^{-1}(d\cdot 0)$. Let us explain why the estimate for the central
fiber is enough. Let us take $x\in\Uh[r]{d}$ and write it as
$(F,\varphi,\sum \lambda_i x_i)$, where $(E,\varphi)\in\Bun[\SL(r)]{d'}$,
$x_i\neq x_j$. The morphism $\pi$ is assigning $(E^{\vee\vee},\varphi,\operatorname{Supp}(E^{\vee\vee}/E))$ to a framed torsion free sheaf $(E,\varphi)\in\Gi{d}$. See \cite[Exercise~3.53]{Lecture}. Then
$\pi^{-1}(x)$ parametrizes quotients of $E^{\vee\vee}$ with given multiplicities $\lambda_i$ at $x_i$. Then it is clear that $\pi^{-1}(x)$ is isomorphic to the product of quotients of $\shfO^{\oplus r}$ with multiplicities $\lambda_i$ at $0$, i.e., $\prod_i \pi^{-1}(\lambda_i\cdot 0)$. If one knows each $\pi^{-1}(\lambda_i \cdot 0)$ has dimension $r\lambda_i - 1$, we have
$\dim\pi^{-1}(x) = \sum (r\lambda_i - 1) = \frac12 \codim \Bun[\SL(r)]{d'}\times S_\lambda(\CC^2)$. Thus it is enough to check that $\pi^{-1}(d\cdot 0) = rd-1$.

\subsection{Equivariant cohomology groups of Giesker partial compactification}

We have a group action of $G\times \GL(2)$ on $\Gi{d}$, $\Uh[r]{d}$ as in 
\subsecref{subsec:group-action}. The Gieseker-Uhlenbeck morphism
$\pi\colon\Gi{d}\to\Uh[r]{d}$ is equivariant. Let $\TT = T\times\CC^\times\times\CC^\times$ as in \subsecref{subsec:group-action}.
We will study the equivariant cohomology groups of Gieseker spaces
\begin{equation*}
    H^{[*]}_{\TT}(\Gi{d}), \quad H^{[*]}_{\TT,c}(\Gi{d})
\end{equation*}
with arbitrary and compact support respectively. We use the degree convention, i.e., the degree $0$ is $2dr$.

We denote the equivariant variables by $\vec{a} = (a_1,\dots,a_r)$ with $a_1+\dots+a_r = 0$ for $T$, and $\ve_1$, $\ve_2$ for $\CC^\times\times\CC^\times$. (See \eqref{eq:30}.)

We have the intersection pairing
\begin{equation*}
    H^{[*]}_\TT(\Gi{d})\otimes H^{[*]}_{\TT,c}(\Gi{d})\to
    H^*_\TT(\mathrm{pt});
    c\otimes c'\mapsto (-1)^{dr} \int_{\Gi{d}} c\cup c'.
\end{equation*}
This is of degree $0$. The sign $(-1)^{dr}$ is introduced to save
$(-1)^\star$ in the later formula. The factor $dr$ should be
understood as the half of the dimension of $\Gi{d}$. Similarly the
intersection form on $H^{[*]}_{\CC^\times\times\CC^\times}(\CC^2)$ has the
sign factor $(-1)^1 = (-1)^{\dim \CC^2/2}$.

\subsection{Heisenberg algebra via correspondences}\label{subsec:heis-algebra-via}

Let $n > 0$.
Let us consider a correspondence in the triple product
$\Gi{d+n}\times\Gi{d}\times\CC^2$:
\begin{equation*}
    P_n \defeq \left\{ (E_1,\varphi_1, E_2,\varphi_2, x)
      \in \Gi{d+n}\times\Gi{d}\times\CC^2 \middle|
      E_1\subset E_2, \text{Supp}(E_2/E_1) = \{x\}
      \right\}.
\end{equation*}
Here the condition $\text{Supp}(E_2/E_1) = \{x\}$ means the quotient
sheaf $E_2/E_1$ is $0$ outside $x$. In the left hand side the index
$d$ is omitted: we understand either $P_n$ is the disjoint union 
for various $d$, or $d$ is implicit from the situation. It is known that
$P_n$ is a lagrangian subvariety in $\Gi{d+n}\times\Gi{d}\times\CC^2$.

We have two projections $q_1\colon P_n\to \Gi{d+n}$, $q_2\colon P_n\to
\Gi{d}\times\CC^2$, which are proper. The convolution product gives an
operator
\begin{equation*}
    P^\Delta_{-n}(\alpha)\colon
    H^{[*]}_\TT(\Gi{d})\to H^{[*+\deg\alpha]}_\TT(\Gi{d+n});
    c\mapsto q_{1*}\left( q_2^*(c\otimes \alpha)\cap [P_n]\right)
\end{equation*}
for $\alpha\in H^{[*]}_{\CC^\times\times\CC^\times}(\CC^2)$.
The meaning of the superscript `$\Delta$' will be explained later.
We also consider the adjoint operator
\begin{equation*}
    P^\Delta_n(\alpha) = (P^\Delta_{-n}(\alpha))^*
    \colon H^{[*]}_{\TT,c}(\Gi{d+n})\to H^{[*+\deg\alpha]}_{\TT,c}(\Gi{d}).
\end{equation*}

A class $\beta\in H^{[*]}_{\CC^\times\times\CC^\times,c}(\CC^2)$ with
compact support gives operators
\(
P^\Delta_{-n}(\beta)\colon
    H^{[*]}_{\TT,c}(\Gi{d})\to H^{[*+\deg\beta]}_{\TT,c}(\Gi{d+n}),
\)
and
\(
    P^\Delta_n(\beta) = (P^\Delta_{-n}(\beta))^*
    \colon H^{[*]}_\TT(\Gi{d+n})\to H^{[*+\deg\beta]}_\TT(\Gi{d}).
\)

\begin{Theorem}[\cite{MR1386846,MR1441880} for $r=1$,
    \cite{Baranovsky} for $r\ge 2$]
    As operators on $\bigoplus_d H^{[*]}_\TT(\Gi{d})$ or $\bigoplus_d
    H^{[*]}_{\TT,c}(\Gi{d})$, we have the Heisenberg commutator relations
    \begin{equation}\label{eq:Heis}
        \left[ P_m^\Delta(\alpha), P_n^\Delta(\beta)\right]
        = r m\delta_{m,-n} \langle\alpha,\beta\rangle \id.
    \end{equation}
\end{Theorem}

Here $\alpha$, $\beta$ are equivariant cohomology classes on $\CC^2$
with arbitrary or compact support. When the right hand side is
nonzero, $m$ and $n$ have different sign, hence one of $\alpha$, $\beta$ is compact support and the other is arbitrary support. Then $\langle\alpha,\beta\rangle$ is well-defined.

\begin{History}
    As mentioned in Introduction, the author \cite{Na-quiver} found
    relation between representation theory of affine Lie algebras and
    moduli spaces of instantons on $\CC^2/\Gamma$, where the affine
    Lie algebra is given by $\Gamma$ by the McKay correspondence. It
    was motivated by works by Ringel \cite{Ringel} and Lusztig
    \cite{MR1035415}, constructing upper triangular subalgebras of
    quantum enveloping algebras by representations of quivers.

    The above theorem can be regarded as the case $\Gamma = \{ 1\}$,
    but the Heisenberg algebra is not a Kac-Moody Lie algebra, and
    hence it was not covered in \cite{Na-quiver}, and dealt with later
    \cite{MR1386846,MR1441880,Baranovsky}. Note that a Kac-Moody Lie
    algebra only has finitely many generators and relations, while the
    Heisenberg algebra has infinitely many.

    A particular presentation of an algebra should {\it not\/} be
    fundamental, so it was desirable to have more intrinsic
    construction of those representations. More precisely, a
    definition of an algebra by convolution products is natural, but
    we would like to understand why we get a particular algebra,
    namely the affine Lie algebra in our case. We do not have a
    satisfactory explanation yet. The same applies to Ringel,
    Lusztig's constructions.
\end{History}

\begin{Exercise}\label{ex:Betti}
    
    \cite[Remark~8.19]{Lecture} Define operators $P^\Delta_{\pm
      1}(\alpha)$ acting on $\bigoplus_n H^{*}(S^n X)$ for a (compact)
    manifold $X$ in a similar way, and check the commutation relation
    \eqref{eq:Heis} with $r=1$.
\end{Exercise}

\subsection{Dimensions of cohomology groups}

When $r=1$, it is known that the generating function of dimension of
$H^{[*]}_\TT((\CC^2)^{[d]})$ over $\bA_T = H^{[*]}_\TT(\mathrm{pt})$
for $d\ge 0$ is
\begin{equation*}
    \sum_{d=0}^\infty \dim H^{[*]}_\TT((\CC^2)^{[d]}) q^d = 
    \prod_{d=1}^\infty \frac1{1-q^d}.
\end{equation*}
(See \cite[Chap.~5]{Lecture}.) This is also equal to the character of
the Fock space\footnote{The Fock space is the polynomial ring of infinitely many variables
$x_1$, $x_2$, \dots. The operators $P^\Delta_n(\alpha)$ act by either multiplication of $x_n$ or differentiation with respect to $x_n$ with appropriate constant multiplication. It has the highest weight vector (or the vacuum vector) $1$, which is killed by $P^\Delta_n(\alpha)$ with $n > 0$. The Fock space is spanned by vectors given by operators $P^\Delta_n(\alpha)$ successively to the highest weight vector.}of the Heisenberg algebra. Therefore the Heisenberg
algebra action produces all cohomology classes from the vacuum vector
$|\mathrm{vac}\rangle = 1_{(\CC^2)^{[0]}}\in H^{[*]}_\TT((\CC^2)^{[0]})$.

On the other hand we have
\begin{equation}\label{eq:1}
    \sum_{d=0}^\infty \dim H^{[*]}_\TT(\Gi{d}) q^d = 
    \prod_{d=1}^\infty \frac1{(1-q^d)^r}.
\end{equation}
for general $r$. Therefore the direct sum of these cohomology groups
is not an irreducible representation for the Heisenberg algebra.

Let us explain how to see the formula \eqref{eq:1}. Consider the torus
$T$ action on $\Gi{d}$. A framed torsion free sheaf $(E,\varphi)$ is
fixed by $T$ if and only if it is a direct sum
$(E_1,\varphi_1)\oplus\cdots \oplus (E_r,\varphi_r)$ of rank $1$ framed
torsion free sheaves. Rank $1$ framed torsion free sheaves are nothing
but ideal sheaves on $\CC^2$, hence points in Hilbert schemes. Thus
\begin{equation}\label{eq:2}
    (\Gi{d})^T = \bigsqcup_{d_1+\dots+d_r = d} (\CC^2)^{[d_1]}\times\cdots\times
    (\CC^2)^{[d_r]}.
\end{equation}
Hence
\begin{equation}\label{eq:31}
    \begin{split}
   \bigoplus_d H^{[*]}_\TT((\Gi{d})^T) &=
   \bigoplus_{d_1,\dots,d_r} H^{[*]}_\TT((\CC^2)^{[d_1]}\times\cdots\times
    (\CC^2)^{[d_r]})
\\
   &= \bigotimes_{i=1}^r \bigoplus_{d_i=0}^\infty H^{[*]}_\TT((\CC^2)^{[d_i]}).
    \end{split}
\end{equation}

It is not difficult to show that $H^{[*]}_\TT(\Gi{d})$ is free over
$\bA_T$. (For example, it follows from the vanishing of odd degree
nonequivariant cohomology group of $\Gi{d}$ and the spectral sequence
relating $H^{[*]}_\TT(\Gi{d})$ and $H^{[*]}(\Gi{d})\otimes \bA_T$.)
Therefore in order to compute the dimension of $H^{[*]}_\TT(\Gi{d})$
over $\bA_T$, we restrict equivariant cohomology groups to generic
points, that is to consider tensor products with the fractional field
$\bF_T$ of $\bA_T$. Then the localization theorem for equivariant
cohomology groups gives an isomorphism between \( H^{[*]}_\TT(\Gi{d})
\) and \( H^{[*]}_\TT((\Gi{d})^T) \) over $\bF_T$. Therefore the above
observation gives the formula \eqref{eq:1}.

In view of \eqref{eq:1}, we have the action of $r$ copies of
Heisenberg algebra on $\bigoplus_d H^{[*]}_\TT((\Gi{d})^T)$ and hence
on $\bigoplus_d H^{[*]}_\TT(\Gi{d})\otimes_{\bA_T}\bF_T$ by the
localization theorem.
It is isomorphic to the tensor product of $r$ copies of the Fock
module, so all cohomology classes are produced by the action.
This is a good starting point to understand $\bigoplus_d H^{[*]}_\TT(\Gi{d})$.
However this action cannot be defined over \emph{non-localized}
equivariant cohomology groups. In fact, $P^\Delta_n(\alpha)$ is the
`diagonal' Heisenberg in the product, and other non diagonal
generators have no description like convolution by $P_n$.

The correct algebra acting on $\bigoplus_d H^{[*]}_\TT(\Gi{d})$ is the
$\scW$-algebra $\scW(\mathfrak{gl}(r))$ associated with
$\mathfrak{gl}(r)$. It is the tensor product of the $\scW$-algebra
$\scW(\algsl(r))$ and the Heisenberg algebra (as the vertex
algebra). Its Verma module has the same size as the tensor product of
$r$ copies of the Fock module.

\begin{Theorem}
    Operators given by correspondences $P_n$ and multiplication by
    Chern classes of the tautological bundle $\cV$ make $\bigoplus_d
    H^{[*]}_\TT(\Gi{d})\otimes_{\bA_\TT}\bF_\TT$ as a
    $\scW(\mathfrak{gl}(r))$-module, isomorphic to a Verma module.
\end{Theorem}

This result is due to Schiffmann-Vasserot \cite{MR3150250} and
Maulik-Okounkov \cite{MO} independently. We will prove this for $r=2$
in \secref{sec:r-matrix} and a similar result for general $G$ in
\thmref{thm:main}.

Let us give a comment on the statement, which is a little imprecise.
Operators discussed are defined over non-localized equivariant
cohomology groups $\bigoplus_d H^{[*]}_\TT(\Gi{d})$.
They and conventional generators of $\scW(\mathfrak{gl}(r))$ are
related by an explicit formula, involving elements in
$\CC(\ve_1,\ve_2)$ (see \eqref{eq:18} for Virasoro generators). This
is a reason why the above theorem is formulated over $\bF_\TT$. An
explicit formula of generators is known only for $\mathfrak{sl}(r)$,
and will not be reviewed, hence an interested reader should read the
original papers \cite{MR3150250,MO}.
Our approach is different. We introduce an $\bA$-form of the
$\scW(\g)$ for general $\g$ in \subsecref{subsec:integr-form}, and
show that it acts on the non localized equivariant intersection
cohomology group in \subsecref{subsec:inters-cohom-with}.

\subsection{Intersection cohomology group}\label{sec:inters-cohom-group}


Recall that the decomposition theorem has a nice form for a semi-small resolution $\pi\colon M\to X$:
\begin{equation}\label{eq:11}
   \pi_*(\cC_M) = \bigoplus_{\alpha,\chi} \IC(X_\alpha,\chi)\otimes H_{[0]}(\pi^{-1}(x_\alpha))_\chi,
\end{equation}
where we have used the following notation:
\begin{itemize}
 \item $\cC_M$ denotes the shifted constant sheaf $\CC_M[\dim M]$.
 \item $\IC(X_\alpha,\chi)$ denotes the intersection cohomology
complex associated with a simple local system $\chi$ on a Zariski open
dense subset in the smooth locus of an irreducible closed subvariety
$X_\alpha$ of $X$.  (We will simply write $\IC(X_\alpha)$ when $\chi$
is trivial. We may also write it $\IC(X)$ if $X_\alpha$ is the open
dense in $X$.)

 \item $H_{[0]}(\pi^{-1}(x_\alpha))$ is the homology group of the shifted degree $0$, which is the usual degree $\codim X_\alpha$. When $x_\alpha$ moves in $X_\alpha$, it forms a local system. $H_{[0]}(\pi^{-1}(x_\alpha))_\chi$ denotes its $\chi$-isotropic component.
\end{itemize}

\begin{Exercise}
    Let $\Gr(d,r)$ be the Grassmannian of $d$-dimensional subspaces in
    $\CC^r$, where $0\le d\le r$. Let $M = T^*\Gr(d,r)$. Determine $X
    = \Spec(\CC[M])$. Study fibers of the affinization morphism
    $\pi\colon M\to X$ and show that $\pi$ is semi-small. Compute
    graded dimensions of $IH^*$ of strata, using the well-known
    computation of Betti numbers of $T^*\Gr(d,r)$.
\end{Exercise}

Consider the Gieseker-Uhlenbeck morphism $\pi\colon\Gi{d}\to\Uh[r]{d}$.
By \propref{prop:semismall}, any fiber $\pi^{-1}(x_\alpha)$ is irreducible. Therefore all the local systems are trivial, and
\begin{equation}\label{eq:23}
   \pi_*(\cC_{\Gi{d}}) = \bigoplus_{d=|\lambda|+d'} 
   \IC(\overline{\Bun[\SL(r)]{d'}\times S_\lambda(\CC^2)})
   	\otimes \CC[\pi^{-1}(x^{d'}_\lambda)],
\end{equation}
where $x^{d'}_\lambda$ denotes a point in the stratum $\Bun[\SL(r)]{d'}\times S_\lambda(\CC^2)$,
and $[\pi^{-1}(x^{d'}_\lambda)]$ denotes the fundamental class of $\pi^{-1}(x^{d'}_\lambda)$, regarded as an element of $H_{[0]}(\pi^{-1}(x^{d'}_\lambda))$.

The main summand is $\IC(\overline{\Bun[\SL(r)]{d}}) = \IC(\Uh[r]{d})$, and other smaller summands could be understood recursively as follows.
Let us write the partition $\lambda$ as $(1^{\alpha_1} 2^{\alpha_2} \dots)$, when $1$ appears $\alpha_1$ times, $2$ appears $\alpha_2$ times, and so on. We set $l(\lambda) = \alpha_1 + \alpha_2 + \cdots$ and $\operatorname{Stab}(\lambda) = S_{\alpha_1}\times S_{\alpha_2}\times \cdots$. Note
that we have total $l(\lambda)$ distinct points in $S_\lambda(\CC^2)$. The group $\operatorname{Stab}(\lambda)$ is the group of symmetries of a configuration in $S_\lambda(\CC^2)$.
We have a finite morphism
\[
   \xi\colon \Uh[\SL(r)]{d'}\times (\CC^2)^{l(\lambda)}/\operatorname{Stab}(\lambda)
   \to \overline{\Bun[\SL(r)]{d'}\times S_\lambda(\CC^2)}
\]
extending the identity on $\Bun[\SL(r)]{d'}\times
S_\lambda(\CC^2)$. Then $\IC(\overline{\Bun[\SL(r)]{d'}\times
  S_\lambda(\CC^2)})$ is the direct image of $\IC$ of the domain.  We
have the K\"unneth decomposition for the domain, and the factor
$(\CC^2)^{l(\lambda)}/\operatorname{Stab}(\lambda)$ is a quotient of a
smooth space by a finite group. Therefore the $\IC$ of the second
factor is the (shifted) constant sheaf. We thus have
\[
   \IC(\overline{\Bun[\SL(r)]{d'}\times S_\lambda(\CC^2)}) \cong
   	\xi_*\left(\IC(\Uh[r]{d'})\boxtimes 
	\cC_{(\CC^2)^{l(\lambda)}/\operatorname{Stab}(\lambda)}\right).
\]
Thus
\begin{multline*}
   H^{[*]}_\TT(\Gi{d}) = \bigoplus_{d=|\lambda|+d'} IH^{[*]}_\TT(\Uh[r]{d'})
   \otimes H^{[*]}_\TT((\CC^2)^{l(\lambda)}/\operatorname{Stab}(\lambda))
   	\otimes \CC[\pi^{-1}(x^{d'}_\lambda)]. 
\end{multline*}

This decomposition nicely fits with the Heisenberg algebra
action. Note that the second and third factors are both
$1$-dimensional. Thus we have $1$-dimensional space for each partition
$\lambda$. If we take the sum over $\lambda$, it has the size of the
Fock module, and it is indeed the submodule generated by the vector
$[x^{d'}_\emptyset]\in H^{[*]}(\Gi[r]{d'})$, the fundamental class of
the point $x^{d'}_\emptyset$ corresponding to the empty partition.
(Note $x^{d'}_\emptyset$ is a point in $\Bun[r]{d'}$, regarded as a
point in $\Gi[r]{d'}$ via $\pi$.)
This statement can be proved by the analysis of the convolution algebra in \cite[Chap.~8]{CG}, but it is intuitively clear as the $1$-dimensional space corresponding to $\lambda$ is the span of $P^\Delta_{-1}(1)^{\alpha_1} P^\Delta_{-2}(1)^{\alpha_2}\cdots|\mathrm{vac}\rangle$.

The Heisenberg algebra acts trivially on the remaining factor
\[
    \bigoplus_d IH^{[*]}_\TT(\Uh[r]{d}).
\]
%
The goal of these notes is to see that it is a module of $\scr
W(\algsl(r))$, and the same is true for ADE groups $G$, not only for
$\SL(r)$.

\begin{Exercise}\label{ex:trivial}
    Show the above assertion that the Heisenberg algebra acts
    trivially on the first factor $\bigoplus_{d}
    IH^{[*]}_\TT(\Uh[\SL(r)]{d})$.
\end{Exercise}

\section{Stable envelopes}

The purpose of this lecture is to explain the stable envelope
introduced in \cite{MO}. It will nicely explain a relation between
$H^{[*]}_\TT(\Gi{d})$ and $H^{[*]}_\TT((\Gi{d})^T)$. This is what we
need to clarify, as we have explained in the previous lecture. The
stable envelope also arises in many other situations in geometric
representation theory. Therefore we explain it in a wider context, as
in the original paper \cite{MO}.

\subsection{Setting -- symplectic resolution}\label{subsec:setting}

Let $\pi\colon M\to X$ be a resolution of singularities of an affine
algebraic variety $X$. We assume $M$ is symplectic. We suppose that a
torus $\TT$ acts on both $M$, $X$ so that $\pi$ is
$\TT$-equivariant. We suppose $\TT$-action on $X$ is linear. We also
assume a technical condition that $\TT$ contains $\CC^\times$ such
that $X$ is a cone with respect to $\CC^\times$ and the weight of the
symplectic form is positive except in
\subsecref{subsec:case-when-pi}. (This assumption will be used when we
quote a result of Namikawa later.)
Let $T$ be a subtorus of $\TT$ which preserves the symplectic form of
$M$.

\begin{Example}
Our basic example is $M = \Gi{d}$, $X = \Uh[r]{d}$ and $\pi$ is the
Gieseker-Uhlenbeck morphism with the same $\TT$, $T$ as above. 
In fact, we can also take a larger torus $T\times\CC^\times_{\rm hyp}$
in $\TT$, where $\CC^\times_{\rm hyp}\subset
\CC^\times\times\CC^\times$ is given by $t\mapsto (t,t^{-1})$.
\end{Example}

\begin{Example}
    Another example is $M = T^*(\text{flag variety}) = T^* (G/B)$, $X
    = (\text{nilpotent variety})$ and $\pi$ is the
    Grothendieck-Springer resolution. Here $T$ is a maximal torus of
    $G$ contained in $B$, and $\TT = T\times\CC^\times$, where
    $\CC^\times$ acts on $X$ by scaling on fibers.
\end{Example}

We can also consider the same $\pi\colon M\to X$ as above with smaller
$\TT$, $T$.

Let $M^T$ be the $T$-fixed point locus in $M$. It decomposes $M^T =
\bigsqcup F_\alpha$ to connected components, and each $F_\alpha$ is a
smooth symplectic submanifold of $M$. Let $i\colon M^T\to M$ be the
inclusion. We have the pull-back homomorphism
\begin{equation*}
    i^*\colon H^{[*]}_\TT(M) \to H^{[*+\codim X^T]}_\TT(M^T)
    = \bigoplus_\alpha H^{[*+\codim F_\alpha]}_\TT(F_\alpha).
\end{equation*}
Here we take the degree convention as before. Our degree $0$ is the
usual degree $\dim_\CC M$ for $H^{[*]}_\TT(M)$, and $\dim_\CC F_\alpha$
for $H^{[*]}_\TT(F_\alpha)$. Since $i^*$ preserves the usual degree, it
shifts our degree by $\codim F_\alpha$. Each $F_\alpha$ has its own
codimension, but we denote the direct sum as $H^{[*+\codim
  X^T]}_\TT(M^T)$ for brevity.

The stable envelope we are going to construct goes in the opposite direction
\(
    H^{[*]}_\TT(M^T)\to H^{[*]}_\TT(M)
\)
and preserves (our) degree.

In the above example $M = \Gi{d}$, the $T$ and $T\times\CC^\times_{\rm
  hyp}$-fixed point loci are
\begin{equation*}
    \begin{split}
        & (\Gi{d})^T = \bigsqcup_{d_1+\cdots+d_r=d}
    \Gi[1]{d_1}\times\cdots\times \Gi[1]{d_r},
\\
        & (\Gi{d})^{T\times\CC^\times_{\rm hyp}}
        = \bigsqcup_{|\lambda_1|+\cdots+|\lambda_r|}
        \{ I_{\lambda_1} \oplus\cdots\oplus I_{\lambda_r}\},
    \end{split}
\end{equation*}
where $\lambda_i$ is a partition, and $I_{\lambda_i}$ is the
corresponding monomial ideal sheaf (with the induced framing).

Also
\begin{equation*}
    T^*(G/B)^T = W,
\end{equation*}
where $W$ is the Weyl group.

\subsection{Chamber structure}

Let us consider the space $\Hom_{\mathrm{grp}}(\CC^\times,T)$ of one
parameter subgroups in $T$, and its real form
$\Hom_{\mathrm{grp}}(\CC^\times,T)\otimes_\ZZ \RR$. 
A generic one parameter subgroup $\rho$ satisfies
$M^{\rho(\CC^\times)} = M^T$. But if $\rho$ is special (the most
extreme case is $\rho$ is the trivial), the fixed point set
$M^{\rho(\CC^\times)}$ could be larger.
This gives us a `chamber' structure on
$\Hom_{\mathrm{grp}}(\CC^\times,T)\otimes_\ZZ \RR$, where a chamber is
a connected component of the complement of the union of hyperplanes
given by $\rho$ such that $M^{\rho(\CC^\times)}\neq M^T$.

\begin{Exercise}\label{ex:chamber}
    (1) In terms of $T$-weights on tangent spaces $T_p M$ at various
    fixed points $p\in M^T$, describe the hyperplanes.

    (2) Show that the chamber structure for $M = T^*(\text{flag
      variety})$ is identified with usual Weyl chambers.

    (3) Show that the chamber structure for $M = \Gi{d}$ is identified
    with the usual Weyl chambers for $\SL(r)$.

    (4) Compute the chamber structure for $M = \Gi{d}$, but with the
    larger torus $T\times\CC^\times_{\rm hyp}$.
\end{Exercise}

For a chamber $\cC$, we have the {\it opposite chamber\/}
$-\cC$ consisting of one parameter subgroups $t\mapsto
\rho(t^{-1})$ for $\rho\in\cC$.

The stable envelope depends on a choice of a chamber $\cC$.

\subsection{Attracting set}

Let $\cC$ be a chamber and $\rho\in\cC$. We define the {\it
  attracting set\/} $\cA_{X}$ by
\begin{equation*}
    \cA_{X} = \left\{ x\in X \middle| \exists
      \lim_{t\to 0} \rho(t)x
    \right\}.
\end{equation*}
We similarly define the attracting set $\cA_{M}$ in $M$ in
the same way. As $\pi$ is proper, we have $\cA_M = \pi^{-1}(\cA_{X})$. We put the scheme structure on $\cA_M$ as $\pi^{-1}(\cA_X)$ in these notes.

\begin{Example}
    Let $X = \Uh[2]{d}$.  In the quiver description, $\cA_{X}$
    consists of closed $\GL(d)$-orbits $\GL(d)(B_1,B_2,I,J)$ such that
    $J f(B_1,B_2) I$ is upper triangular for any noncommutative
    monomial $f\in \CC\langle x,y\rangle$. It is the {\it tensor
      product variety\/} introduced in \cite{Na-Tensor}, denoted by
    $\pi(\mathfrak Z)$ therein.

    As framed sheaves, $\cA_M$ consists of $(E,\varphi)$ which are written as
    an extension $0\to E_1 \to E\to E_2\to 0$ (compatible with the framing)
    for some $E_1\in\Gi[1]{d_1}$, $E_2\in\Gi[1]{d_2}$ with $d=d_1+d_2$.
\end{Example}

We have the following diagram
\begin{equation}\label{eq:13}
    X^{\rho(\CC^\times)} = X^T
    \overset{p}{\underset{i}{\leftrightarrows}} \cA_X
    \overset{j}{\rightarrow} X,
\end{equation}
where $i$, $j$ are natural inclusion, and $p$ is given by 
\(
   \cA_X\ni x\mapsto \lim_{t\to 0} \rho(t) x.
\)
\begin{NB}
This is wrong.

We also have the corresponding diagram for $M$, and two diagrams are commutative under
$\pi$.
\end{NB}%
If $X$ is a representation of $T$, $X^{T}$ (resp.\ $\cA_X$) is the
$0$-weight space (resp.\ the direct sum of nonnegative weight spaces
with respect to $\rho(\CC^\times)$). Hence $X^T$ and $\cA_X$ are
closed subvarieties, and $p$ is a morphism. The same is true for
general $X$ as an affine algebraic variety and the $T$-action is
linear.

\subsection{Leaves}\label{subsec:leaves}

Let $p$ denote the map $\cA_M\ni x\mapsto\lim_{t\to 0}\rho(t)x\in
M^T$.  This is formally the same as one for $\cA_X\to X^T$, but it is
not continuous as the limit point may jump at a special point $x$, as
we will see an example below. Nevertheless it is set-theoretically
well-defined.
Since $M^T = \bigsqcup F_\alpha$, we have the corresponding
decomposition of $\cA_M = \bigsqcup p^{-1}(F_\alpha)$. Let
$\Leaf_\alpha = p^{-1}(F_\alpha)$. By the Bialynicki-Birula theorem
(\cite{BB}, see also \cite{MR523824}), $p\colon\Leaf_\alpha\to F_\alpha$
is a $\TT$-equivariant affine bundle.
Similarly $\Leaf^-_\alpha\to F_\alpha$ denote the corresponding affine
bundle for the opposite chamber $-\cC$.
(For quiver varieties, they are in fact isomorphic to vector bundles
$L^\pm_\alpha$ below. See \cite[Prop.~3.14]{Na-Tensor}.)

Let us consider the restriction of the tangent bundle $TM$ to a fixed
point component $F_\alpha$. It decomposes into weight spaces with
respect to $\rho$:
\begin{equation}\label{eq:3}
    TM|_{F_\alpha} = \bigoplus T(m),
    \qquad T(m) = \{ v \mid \rho(t)v = t^m v \}.
\end{equation}
Then $T\Leaf_\alpha = \bigoplus_{m\ge 0} T(m)$. Note also $T(0) =
TF_\alpha$. Since $T$ preserves the symplectic form, $T(m)$ and
$T(-m)$ are dual to each other. From these, one can also check that
$\Leaf_\alpha\xrightarrow{j\times p} M\times F_\alpha$ is a lagrangian
embedding.

For a later purpose let 
\begin{equation}
    \label{eq:19}
    L^\pm_\alpha \defeq \bigoplus_{\pm m > 0} T(m).
\end{equation}
The normal bundle of $F_\alpha$ in $\Leaf_\alpha$ (resp.\
$\Leaf^-_\alpha$) is $L^+_\alpha$ (resp.\ $L^-_\alpha$).

\begin{Example}\label{ex:T*P}
    Let $\pi\colon M = T^*\proj^1\to X = \CC^2/\pm$. Let $T=\CC^\times$
    act on $X$ and $\CC^2/\pm 1$ so that it is given by
    $t(z_1,z_2)\bmod\pm = (tz_1, tz_2^{-1})\bmod\pm$. Then $X^T$
    consists of two points $\{0,\infty\}$ in the zero section
    $\proj^1$ of $T^*\proj^1$. If we take the `standard' chamber
    containing the identity operator, $\Leaf_0$ is the zero section
    $\proj^1$ minus $\infty$. On the other hand $\Leaf_\infty$ is the
    (strict transform of) the axis $z_2 = 0$. See Figure~\ref{fig:leaf}.

    For the opposite chamber, $\Leaf_0$ is the axis $z_1 = 0$, and
    $\Leaf_\infty$ is the zero section minus $0$.

\begin{figure}[htbp]
    \centering
\begin{tikzpicture}[scale=0.3]
\draw (-5,10) parabola bend (5,-5) (15,10);
\draw[thick,fill] (1.5,0) circle (0.3);
\draw[thick] (8.5,0) circle (0.3);
\draw[thick,bend right,distance=40] (1.8,0) node[below left] {$0$}
to node[midway,below] {$\Leaf_0$}
(8.2,0) node[below right] {$\infty$};
\draw[thick,dotted,bend right,distance=25] (8.7,0.3) to
node[midway] {$\Leaf_\infty$}
(12,10);
\end{tikzpicture}
    \caption{Leaves in $T^*\proj^1$}
    \label{fig:leaf}
\end{figure}
\end{Example}

\begin{Definition}\label{def:order}
We define a partial order $\ge$ on the index set $\{\alpha\}$ for the fixed
point components so that
\begin{equation*}
    \overline{\Leaf_\beta}\cap F_\alpha\neq\emptyset
    \Longrightarrow \alpha \le \beta.
\end{equation*}
\end{Definition}
We have $\infty\le 0$ in Example~\ref{ex:T*P}.

Let
\begin{equation*}
    \cA_{M,\le\alpha} = \bigsqcup_{\beta:\beta\le\alpha} \Leaf_\beta. 
\end{equation*}
Then $\cA_{M,\le\alpha}$ is a closed subvariety. We define
$\cA_{M,<\alpha}$ in the same way.

\begin{Proposition}\label{prop:filtration}
    \textup{(1)} $H_{[*]}^\TT(\cA_{M,\le\alpha})$ vanishes in the odd
    degree.

    \textup{(2)} We have an exact sequence
\begin{equation*}
    0\to
    H_{[*]}^\TT(\cA_{M,<\alpha}) \to H_{[*]}^\TT(\cA_{M,\le\alpha})
    \to H_{[*]}^\TT(\Leaf_\alpha) \to 0.
\end{equation*}
\end{Proposition}

\begin{proof}
Consider the usual long exact sequence
\begin{equation*}
    H_{[*]}^\TT(\cA_{M,<\alpha}) \to H_{[*]}^\TT(\cA_{M,\le\alpha})
    \to H_{[*]}^\TT(\Leaf_\alpha) \xrightarrow{\delta}
    H_{[*-1]}^\TT(\cA_{M,<\alpha}).
\end{equation*}
Recall that $\Leaf_\alpha$ is an affine bundle over $F_\alpha$. Hence
the pullback $H_{[*]}^\TT(F_\alpha)\to H_{[*]}^\TT(\Leaf_\alpha)$ is
an isomorphism. It is known that $H^\TT_{[*]}(F_\alpha)$ vanishes in
odd degrees. (It follows from \cite[Exercise~5.15]{Lecture} for
$\Gi{d}$, and is a result of Kaledin \cite{MR2283801} in general.)

Let us show that $H_{[*]}^\TT(\cA_{M,\le\alpha})$ vanishes in odd
degrees by the descending induction on $\alpha$. In particular, the
assertion for $H^\TT_{[*]}(\cA_{M,<\alpha})$ implies $\delta = 0$, i.e., (2).

If $\alpha$ is larger than the maximal element, $\cA_{M,\le\alpha} =
\emptyset$, hence the assertion is true. Suppose that the assertion is true for $H^\TT_{[*]}(\cA_{M,<\alpha})$. Then the above exact sequence and the odd vanishing of $H^\TT_{[*]}(\Leaf_\alpha)$ implies the odd vanishing of
$H^\TT_{[*]}(\cA_{M,\le\alpha})$.
\end{proof}

\begin{Exercise}
    \textup{(1)} Determine $\cA_M$ for $M = T^*(\text{flag variety})$.

    \textup{(2)} Determine $\cA_M$ for $M = \Gi{d}$ with respect to
    $T\times\CC^\times_{\rm hyp}$.
\end{Exercise}

\subsection{Steinberg type variety}

Recall that Steinberg variety is the fiber product of $T^*(\text{flag
  variety})$ with itself over $\text{nilpotent variety}$. Its
equivariant $K$-group realizes the affine Hecke algebra (see
\cite[Ch.~7]{CG}), and it plays an important role in the geometric
representation theory.

Let us recall the definition of the product in our situation. We
define $Z = M\times_{X} M$, the fiber product of $M$ itself over
$X$. Its equivariant Borel-Moore homology group has the convolution
product:
\begin{equation*}
    H^\TT_{[*]}(Z)\otimes H^\TT_{[*]}(Z)\ni c\otimes c'
    \mapsto p_{13*}(p_{12}^*c \cap p_{23}^* c')\in H^\TT_{[*]}(Z),
\end{equation*}
where $p_{ij}\colon M\times M\times M\to M\times M$ is the projection
to the product of $i^{\mathrm{th}}$ and $j^{\mathrm{th}}$
factors.\footnote{We omit explanation of pull-back with supports
  $p_{12}^*$, $p_{23}^*$, etc. See \cite{CG} for more detail.} When
$M\to X$ is semi-small, one can check that the multiplication
preserves the shifted degree $[*]$.

We introduce a variant, mixing the fixed point set $M^T$ and the whole
variety $M$:
Let $Z_\cA$ be the fiber product of $\cA_M$ and $M^T$ over $X^T$, considered as a closed subvariety in $M\times M^T$:
\begin{equation*}
   Z_\cA = \cA_M \times_{X^T} M^T \subset M\times M^T,
\end{equation*}
where $\cA_M\to X^T$ is the composite of $p\colon\cA_M\to M^T$ and
$\pi^T\colon M^T\to M^T_0$, or alternatively $\pi|_{\cA_M}\colon
\cA_M\to \cA_{M_0}$ and $p\colon \cA_{M_0}\to X^T$. Here we denote
projections $\cA_M\to M^T$ and $\cA_{X}\to X^T$ both by $p$ for
brevity.
As a subvariety of $M\times M^T$, $Z_\cA$ consists of pairs $(x,x')$
such that $\lim_{t\to 0} \rho(t)\pi(x) = \pi^T(x')$.

The convolution product as above defines a $(H^\TT_{[*]}(Z),
H^\TT_{[*]}(Z^T))$-bi\-module structure on $H^\TT_{[*]}(Z_\cA)$.
\begin{NB}
    The right $H^\TT_{[*]}(Z^T)$-module structure is clear. For the
    left $H^\TT_{[*]}(Z)$-module structure, note the following:
    Suppose $(x_1,x_2,x_3)\in p_{12}^{-1}(Z)\cap
    p_{23}^{-1}(Z_\cA)$. It means that $\pi(x_1) = \pi(x_2)$,
    $p\pi(x_2) = \pi^T(x_3)$. Therefore
    $p\pi(x_1)=\pi^T(x_3)$. Therefore $(x_1,x_3)\in Z_\cA$. Moreover
    $p_{13}|_{p_{12}^{-1}(Z)\cap p_{23}^{-1}(Z_\cA)}$ is proper.
\end{NB}%
Here $Z^T$ is the $T$-fixed point set in $Z$, or equivalently the
fiber product of $M^T$ with itself over $X^T$.

In our application, we use $H^\TT_{[*]}(Z)$ as follows:
we shall construct an operator $H^{[*]}_{\TT}(M^T)\to H^{[*]}_{\TT}(M)$ by
\begin{equation}\label{eq:5}
    H^{[*]}_{\TT}(M^T)\ni \xi \mapsto 
    \cL\ast \xi \defeq
    p_{1*}(p_2^* \xi\cap \cL) \in H^{[*]}_{\TT}(M)
\end{equation}
for a suitably chosen (degree $0$) equivariant class $\cL$ in
$H_{[0]}^{\TT}(Z_\cA)$. Note that the projection $Z_\cA\to M$ is
proper, hence the operator in this direction is well-defined. On the
other hand, $Z_\cA\to M^T$ is {\it not\/} proper. See
\subsecref{subsub:adjoint} below.

Recall $\cA_M$ and $M^T$ decompose as $\bigsqcup \Leaf_\beta$, $\bigsqcup F_\alpha$ respectively. Therefore
\begin{equation*}
   Z_\cA = \bigsqcup_{\alpha,\beta} \Leaf_\beta\times_{X^T} F_\alpha.
\end{equation*}
We have the projection $p\colon \Leaf_\beta\times_{X^T} F_\alpha \to F_\beta\times_{X^T} F_\alpha$.

\begin{Proposition}\label{prop:irr}
    $Z_\cA$ is a lagrangian subvariety in $M\times M^T$. If $Z_1^F$,
    $Z_2^F$, \dots denote the irreducible components of
    $\bigsqcup_{\alpha,\beta} F_\beta\times_{X^T} F_\alpha$, then the
    closures of their inverse images 
\[
    Z_1 \defeq
    \overline{p^{-1}(Z_1^F)}, Z_2\defeq \overline{p^{-1}(Z_1^F)}, \dots
\]
    are the irreducible components of $Z_\cA$.
\end{Proposition}

\begin{proof}
   Consider an irreducible component $Z_\nu^F$ of $F_\beta\times_{X^T} F_\alpha$. Using the semi-smallness of $\pi^T\colon M^T\to X^T$, one can check that $Z_\nu^F$ is half-dimensional. Then the dimension of $Z_\nu$ is
\begin{equation*}
   \frac12 (\dim F_\beta + \dim F_\alpha) + \frac 12 \codim_X F_\beta
   	= \frac12 (\dim X + \dim F_\alpha),
\end{equation*}
as the dimension of fibers of $\Leaf_\beta$ is the half of
codimension of $F_\beta$. Therefore $Z_\nu$ is half-dimensional in
$F_\beta\times M$.

We omit the proof that $Z_\nu$ is lagrangian.
\end{proof}

\subsection{Polarization}

Consider the normal bundle $N_{F_\alpha/M}$ of $F_\alpha$ in $M$. It
is the direct sum $L^+_\alpha\oplus L^-_\alpha$, where $L^\pm_\alpha$
is as in \eqref{eq:19}. Since $L^+_\alpha$, $L^-_\alpha$ are dual to
each other with respect to the symplectic form, the equivariant Euler
class $e(N_{F_\alpha/M})$ is a square up to sign:
\begin{equation*}
    e(N_{F_\alpha/M}) = e(L_\alpha^+) e(L_\alpha^-)
    = (-1)^{\codim_M F_\alpha/2} e(L^+_\alpha)^2.
\end{equation*}

Let us consider $H^*_T(\mathrm{pt})$-part of the equivariant classes, i.e., weights of fibers of vector bundles as $T$-modules. In many situations, we have another preferred choice of a square root of $(-1)^{\codim_M F_\alpha/2} e(N_{F_\alpha/M})|_{H^*_T(\mathrm{pt})}$. For example, suppose $M = T^*Y$ and the $T$-action on $M$ is coming from a $T$-action on $Y$. Then $F_\alpha = T^* N$ for a submanifold $N\subset Y$, we have $T(T^*Y) = TY\oplus T^* Y$, and
\(
    N_{F_\alpha/M} = N_{N/Y}\oplus N_{N/Y}^*,
\)
hence
\(
   (-1)^{\codim_M F_\alpha/2} e(N_{F_\alpha/M}) = e(N_{N/Y})^2.
\)
This choice $e(N_{N/Y})$ of the square root is more canonical than
$e(L^+_\alpha)$ as it behaves more uniformly in $\alpha$. We call such
a choice of a square root of $(-1)^{\codim_M F_\alpha/2}
e(N_{F_\alpha/M})|_{H^*_T(\mathrm{pt})}$ {\it polarization}.
We will understand a polarization as a choice of $\pm$ for each
$\alpha$, e.g., $e(L^+_\alpha)|_{H^*_T(\mathrm{pt})} = \pm
e(N_{N/Y})|_{H^*_T(\mathrm{pt})}$ for this example.

Our main example $\Gi{d}$ is not a cotangent bundle, but is a symplectic reduction of a symplectic vector space, which is a cotangent bundle. Therefore we also have a natural choice of a polarization for $\Gi{d}$.

\subsection{The case when $\pi$ is an isomorphism}\label{subsec:case-when-pi}

Before giving a definition of the stable envelope, let us consider the
situation when $\pi\colon M\to X$ is an {\it isomorphism}. In this
case we have
\begin{equation*}
    Z_\cA = \cA_M = \cA_X = \bigsqcup_\alpha \Leaf_\alpha,
\end{equation*}
where the decomposition is by connected components. Therefore unlike a
general situation, all $\Leaf_\alpha$ are smooth closed subvarieties
in $M$ as $p$ is a morphism in this case.

As it will be clear below, we do not have an extra $\CC^\times$-action
which scales the symplectic form. Therefore we have $\TT = T$. However
we use the notation $T$ for equivariant cohomology groups, the reason
will be clear soon.

Let
\begin{equation*}
    F_\alpha
    \overset{p_\alpha}{\underset{i_\alpha}{\leftrightarrows}} \Leaf_\alpha
    \overset{j_\alpha}{\rightarrow} M
\end{equation*}
be the restriction of the diagram \eqref{eq:13} to $\Leaf_\alpha$.
Then for $\gamma_\alpha\in H^T_{[*]}(F_\alpha)$ we have
\begin{equation*}
    [Z_\cA]\ast \gamma_\alpha = j_{\alpha*} p_\alpha^*(\gamma_\alpha).
\end{equation*}
In particular, the following properties are clearly satisfied.
\begin{equation*}
    \begin{split}
        & i_\alpha^* ([Z_\cA]\ast \gamma_\alpha) =
        e(L_\alpha^-)\cup \gamma_\alpha,
\\
        & i_\beta^* ([Z_\cA]\ast \gamma_\alpha) = 0 \quad
        \text{for $\beta\neq\alpha$}.
    \end{split}
\end{equation*}

\subsection{Definition of the stable envelope}\label{sec:defin-stable-envel}

Now we return back to a general situation.

Let $\iota_{\beta,\alpha}\colon F_\beta\times F_\alpha\to M\times M^T$
denote the inclusion.

\begin{Theorem}\label{thm:stable-envel}
    Choose and fix a chamber $\cC$ and a polarization $\pm$. There
    exists a unique homology class $\cL\in H^T_{[0]}(Z_\cA)$ with the
    following three properties\textup:
 
 \textup{(1)} $\cL|_{M\times F_\alpha}$ is supported on $\bigcup_{\beta\le\alpha} \overline{\Leaf_\beta}\times_{X^T} F_\alpha$.
 
 \textup{(2)} Let $e(L^-_\alpha)$ denote the equivariant Euler class
 of the bundle $L^-_\alpha$ over $F_\alpha$.  We have
 \begin{equation*}
  \iota_{\alpha,\alpha}^*\cL = \pm e(L^-_\alpha)\cap [\Delta_{F_\alpha}].
\end{equation*}

\textup{(3)} For $\beta < \alpha$, we have $\iota_{\beta,\alpha}^*\cL = 0$.
\end{Theorem}

\begin{Remark}
    Since $Z_\cA$ is lagrangian in $M\times M^T$, $H^T_{[0]}(Z_\cA)$
    has a base by irreducible components of $Z_\cA$. The same is true
    for the $\TT$-equivariant Borel-Moore homology, and even for the
    non-equivariant homology. In particular, though $\cL$ is
    constructed in $H^T_{[0]}(Z_\cA)$, it automatically gives a
    well-defined class in $H^\TT_{[0]}(Z_\cA)$. This is important as
    we eventually want to understand $H^\TT_{[*]}(M)$.

    Note that the property (2) still holds in $\TT$-equivariant
    cohomology group by the construction. However
    $\iota_{\beta,\alpha}^*\cL \neq 0$ in general.
\end{Remark}

\begin{proof}[Proof of the existence]
    We first prove the existence. 

    We apply the construction in \subsecref{subsec:case-when-pi} after
    a deformation as follows.

    We use the following fact due to Namikawa \cite{MR2746388}: $M$
    and $X$ have a $1$-parameter deformation $\scr M$, $\scr X$
    together with $\Pi\colon \scr M\to \scr X$ over $\CC$ such that
    the original $M$, $X$ are fibers over $0\in\CC$, and other fibers
    $M^t$, $X^t$ ($t\neq 0$) are isomorphic, both are smooth and
    affine. Moreover the $T$-action extends to $M^t$, $X^t$.

For $\Gi{d}\to\Uh[r]{d}$, such a deformation can be defined by using quiver description~: we perturb the moment map equation as $[B_1,B_2]+IJ = t\id$. The $\TT$-action does not preserve this equation, as it scales $t\id$. But the $T$-action preserves the equation. For $T^*(\text{flag variety})\to(\text{nilpotent variety})$, we use deformation to semisimple adjoint orbits. (See \cite[\S3.4]{CG}.)

We define $Z_{\cA}^t$ for $M^t$ in the same way. It decomposes as
$Z_{\cA}^t = \bigsqcup \Leaf_\alpha^t$ as in
\subsecref{subsec:case-when-pi}. We consider the fundamental class
$[Z_{\cA}^t] = \sum [\Leaf_{\alpha}^t]$ . We now define
\begin{equation*}
   \cL = \sum_\alpha \cL_\alpha; \qquad
   	\cL_\alpha \defeq \pm \lim_{t\to 0} [\Leaf_\alpha^t],
\end{equation*}
where $\lim_{t\to 0}$ is the specialization in Borel-Moore homology
groups (see \cite[\S2.6.30]{CG}).

The conditions (1),(2),(3) are satisfied for $\sum \pm
[\Leaf_\alpha^t]$.
Taking the limit $t\to 0$, we see that three conditions (1),(2),(3)
are satisfied also for $\cL$.
\end{proof}

The uniqueness can be proved under a weaker condition. Let us explain
it before starting the proof.

Note that $T$ acts on $M^T$ trivially. Therefore the equivariant
cohomology $H^{[*]}_T(M^T)$ is isomorphic to $H^{[*]}(M^T)\otimes
H^*_T(\mathrm{pt})$. The second factor $H^*_T(\mathrm{pt})$ is the
polynomial ring $\CC[\operatorname{Lie}T]$.

We define the {\it degree\/} $\deg_T$ on $H^{[*]}_T(M^T)$ as the
degree of the component of $H^*_T(\mathrm{pt}) =
\CC[\operatorname{Lie}T]$. Now (3) is replaced by
\begin{equation*}
   \deg_T \iota_{\beta,\alpha}^*\cL < \frac 12 \codim_{M} F_\beta.
   \leqno{(3)'}
\end{equation*}

\begin{proof}[Proof of the uniqueness]
    We decompose $\cL = \sum \cL_\alpha$ according to $M^T = \bigsqcup
    F_\alpha$ as above. Let $Z_1$, $Z_2$, \dots be irreducible
    components of $Z$ as in \propref{prop:irr}.
Each $Z_k$ is coming from an irreducible component of
$F_\beta\times_{X^T} F_\alpha$. So let us indicate it as
$Z_k^{(\beta,\alpha)}$.
Since $\cL$ is the top degree class, we have $\cL_\alpha = \sum a_k
[Z_k^{(\beta,\alpha)}]$ with some $a_k\in\CC$. Moreover by (1),
$Z_k^{(\beta,\alpha)}$ with nonzero $a_k$ satisfies $\beta\le\alpha$.

Consider $Z_k^{(\alpha,\alpha)}$. By (2), we must have
$Z_k^{(\alpha,\alpha)} = \overline{\Leaf_\alpha}$, $a_k = \pm 1$,
where $\overline{\Leaf_\alpha}$ is mapped to $Z_\cA = M\times_{X^T} M^T$
by $(\text{inclusion})\times p$.

Suppose that $\cL^1_\alpha$, $\cL^2_\alpha$ satisfy conditions
(1),(2),(3)'. Then the above discussion says $\cL^1_\alpha -
\cL^2_\alpha = \sum a'_k [Z_k^{(\beta,\alpha)}]$, where
$Z_k^{(\beta,\alpha)}$ with nonzero $a'_k$ satisfies $\beta <
\alpha$. Suppose that $\cL^1_\alpha\neq \cL^2_\alpha$ and take a
maximum $\beta_0$ among $\beta$ such that there exists
$Z_k^{(\beta,\alpha)}$ with $a'_k\neq 0$.

Consider the restriction
$\iota_{\beta_0,\alpha}^*(\cL^1_\alpha-\cL^2_\alpha)$. By the maximality,
only those $Z_k^{(\beta,\alpha)}$'s with $\beta = \beta_0$ contribute
and hence
\begin{equation*}
    \iota_{\beta_0,\alpha}^*(\cL^1_\alpha-\cL^2_\alpha)
    = \sum a'_k e(L^-_{\beta_0})\cap [Z_k^F],
\end{equation*}
where $Z_k^F$ is an irreducible component of
$F_{\beta_0}\times_{X^T} F_\alpha$ corresponding to
$Z_k^{(\beta_0,\alpha)}$.

Since $e(L^-_{\beta_0})|_{H^*_T(\mathrm{pt})}$ is the product of
weights $\lambda$ such that $\langle \rho, \lambda\rangle < 0$
(cf. \eqref{eq:3}), it has degree exactly equal to $\codim
F_\beta/2$. But it contradicts with the condition (3)'. Therefore
$\cL^1_\alpha = \cL^2_\alpha$.
\end{proof}

\begin{Definition}
    The operator defined by the formula \eqref{eq:5} given by the
    class $\cL$ constructed in \thmref{thm:stable-envel} is called the
    {\it stable envelope}. It is denoted by $\Stab_{\cC}$. (The
    dependence on the polarization is usually suppressed.)
\end{Definition}

\begin{Example}
    Consider $\pi\colon M = T^*\proj^1\to X = \CC^2/\pm$ as in
    \ref{ex:T*P}. Since $\infty\le 0$, $\cL_\infty = \cL|_{M\times
      \{\infty\}} = \pm e(L^-_\infty)\cap[\Delta_{\infty}]$. This is
    equal to $\pm [\Leaf_\infty]$ if we identify $M\times\{\infty\}$
    with $M$. Next consider $\cL_0 = \cL|_{M\times \{0\}}$. We
    consider $\Leaf_0^t$ in a deformation $M^t$. It decomposes into
    the union $\Leaf_0\cup \Leaf_\infty$ as $t\to 0$. Therefore $\cL_0
    = \pm ([\overline{\Leaf_0}] + [\Leaf_\infty])$. Alternatively we
    consider the restriction of $[\overline{\Leaf_0}]$ and
    $[\Leaf_\infty]$ to $\infty$. The tangent spaces of
    $\overline{\Leaf_0}$ and $\Leaf_\infty$ at $\infty$ are dual to
    each other as $T$-modules. Therefore the restriction of
    $[\overline{\Leaf_0}]$ and $[\Leaf_\infty]$ cancel. Thus the
    property (3) holds.
\end{Example}

\begin{Exercise}(See \cite[Ex.~5.10.5]{2014arXiv1406.2381B}.)
    Compute the stable envelope for $M = T^*\CP^{r-1}$, $T =
    (\CC^\times)^{r-1}$. Here $T$ is a maximal torus of $\SL(r)$
    acting on $M$ in the natural way.
\end{Exercise}

Let us list several properties of stable envelopes.

\subsection{Adjoint}\label{subsub:adjoint}

By a general property of the convolution product, the adjoint
operator
\begin{equation*}
    \Stab_\cC^*\colon H^{[*]}_{\TT,c}(M) \to H^{[*]}_{\TT,c}(M^T)
\end{equation*}
is given by changing the first and second factors in \eqref{eq:5}:
\begin{equation*}
    H^{[*]}_{\TT,c}(M)\ni c \mapsto 
    (-1)^{\codim X^T} p_{2*}(p_1^* c\cap \cL) \in H^{[*]}_{\TT,c}(M^T),
\end{equation*}
where the sign comes from our convention on the intersection form.

\subsection{Image of the stable envelope}

From the definition, the stable envelope defines a homomorphism
\begin{equation}\label{eq:4}
    H^{[*]}_\TT(M^T) \to H_{[-*]}^\TT(\cA_M),
\end{equation}
so that the original $\Stab_\cC$ is the composition of the above
together with the pushforward $H_{[-*]}^\TT(\cA_M)\to H_{[-*]}^\TT(M)$
of the inclusion $\cA_M\to M$ and the Poincar\'e duality
$H_{[-*]}^\TT(M)\cong H^{[*]}_\TT(M)$.

\begin{Proposition}
    \eqref{eq:4} is an isomorphism.
\end{Proposition}

By \propref{prop:filtration}, $H^\TT_{[*]}(\cA_M)$ has a natural
filtration $\bigcup_{\alpha} H^\TT_{[*]}(\cA_{M,\le\alpha})$ whose
associated graded is $\bigoplus_\alpha
H^\TT_{[*]}(\Leaf_\alpha)$. From the construction, the stable envelope
is compatible with the filtration, where the filtration on
$H^{[*]}_\TT(M^T)$ is one induced from the decomposition to $\bigoplus
H^{[*]}_\TT(F_\alpha)$. Moreover the induced homomorphism
$H^{[*]}_\TT(F_\alpha) \to H_{[*]}^\TT(\Leaf_\alpha)$ is the
pull-back, and hence it is an isomorphism. Therefore the original stable envelope is also an isomorphism.

\subsection{Subtorus}

Let $T'$ be a subtorus of $T$ such that $M^{T'}$ is strictly larger
than $M^T$. We have a natural inclusion
\(
   \Hom_{\mathrm{grp}}(\CC^\times,T')\otimes_\ZZ \RR
   \subset
   \Hom_{\mathrm{grp}}(\CC^\times,T)\otimes_\ZZ \RR.
\)
Then a one parameter subgroup in $T'$ is contained in the union of
hyperplanes, which we removed when we have considered chambers. We take a chamber $\cC'$ for $T'$ and a chamber $\cC$ for $T$ such that
\begin{equation*}
    \cC'\subset \cC.
\end{equation*}
A typical example will appear in Figure~\ref{fig:YB}, where $T$ is
$2$-dimensional torus $\{ (t_1,t_2,t_3) \mid t_1t_2t_3 = 1\}$, and
$T'$ is $1$-dimensional given by, say, $t_1 = t_2$. For example, we
take $\cC$ as in the figure, and $\cC'$ as a half line in $a_1 = a_2$
from the origin to the right.

We consider $\Stab_{\cC'}\colon H^{[*]}_{\TT}(M^{T'})\to
H^{[*]}_{\TT}(M)$ defined by $\cC'$. We can also consider $M^T$ as the
fixed point locus in $M^{T'}$ with respect to the action of the
quotient group $T/T'$. Then the chamber $\cC$ projects a chamber for
$T/T'\curvearrowright M^{T'}$. Let us denote it by $\cC/\cC'$. It also
induces $\Stab_{\cC/\cC'}\colon H^{[*]}_{\TT}(M^{T})\to
H^{[*]}_{\TT}(M^{T'})$.
Here note that a polarization for $\cC$ induces ones for $\cC'$ and
$\cC/\cC'$ by considering of restriction of weights from
$\operatorname{Lie}T$ to $\operatorname{Lie}T'$. The uniqueness
implies

\begin{Proposition}\label{prop:factorStab}
    \begin{equation*}
        \Stab_{\cC'} \circ \Stab_{\cC/\cC'} = \Stab_{\cC}
        \colon H^{[*]}_{\TT}(M^{T})\to H^{[*]}_{\TT}(M).
    \end{equation*}
\end{Proposition}


\section{Sheaf theoretic analysis of stable envelopes}

In this lecture, we will study stable envelopes in sheaf theoretic terms. In particular, we will connect them to hyperbolic restriction introduced in \cite{Braden}.

\subsection{Nearby cycle functor}\label{subsec:nearby-cycle-functor}

This subsection is a short detour, or a warm up, proving
$\pi_*(\cC_M)$ is perverse for a symplectic resolution $\pi\colon M\to
X$ without quoting Kaledin's semi-smallness result, mentioned in a
paragraph after \propref{prop:semismall}. It uses the nearby cycle
functor,\footnote{I learned usage of the nearby cycle functor for
  symplectic resolutions from Victor Ginzburg many years ago. He
  attributed it to Gaitsgory. See \cite{MR1826370}.} which is
important itself, and useful for sheaf theoretic understanding of
stable envelopes in \secref{sec:exactn-nearby-cycle}. Moreover we will
use the same idea to show that hyperbolic restriction functors
appearing in symplectic resolution preserve perversity. See
\subsecref{sec:exactn-nearby-cycle}.

Let us recall the definition of the nearby cycle functor in
\cite[\S8.6]{KaSha}.

Let $\scr X$ be a complex manifold and $f\colon \scr X\to \CC$ a
holomorphic function. Let $X = f^{-1}(0)$ and $r\colon X\to \scr X$ be
the inclusion. We take the universal covering $\tilde\CC^*\to\CC^*$ of
$\CC^* = \CC\setminus\{0\}$. Let $c\colon\tilde\CC^*\to\CC$ be the
composition of the projection and inclusion $\CC^*\to\CC$. We take the
fiber product $\tilde{\scr X}^*$ of $\scr X$ and $\tilde\CC^*$ over $\CC$. We
consider the diagram
\begin{equation*}
    \begin{CD}
        @. \tilde{\scr X}^* @>>> \tilde{\CC}^*\\
        @. @VV{\tilde c}V @VV{c}V \\
        X @>>r> \scr X @>>f> \CC.
    \end{CD}
\end{equation*}
Then the nearby cycle functor from $D^b(\scr X)$ to $D^b(X)$ is
defined by
\begin{equation}\label{eq:15}
    \psi_f(\bullet) = r^* \tilde c_* \tilde c^*(\bullet).
\end{equation}
Note that it depends only on the restriction of objects in $D^b(\scr X)$ to $\scr X\setminus X$.


We now take deformations $\scr X$, $\scr M$ used in the proof of
\thmref{thm:stable-envel}. They are defined over $\CC$, hence we have
projections $\scr X$, $\scr M\to \CC$. Let us denote them by $f_{X}$,
$f_{M}$ respectively.
(Singularities of $\scr X$ causes no trouble in the following
discussion: we replace $\scr X$ by an affine space to which $\scr X$
is embedded.)
We consider the nearby cycle functors $\psi_{f_{X}}$,
$\psi_{f_{M}}$. Since $\Pi\colon\scr M\to\scr X$ is proper, one
can check
\begin{equation}\label{eq:16}
    \psi_{f_{X}} \Pi_* = \pi_* \psi_{f_{M}}
\end{equation}
from the base change. 

Now we apply the both sides to the shifted constant sheaf $\cC_{\scr
  M}$. Since $\scr M\to\CC$ is a smooth fibration, i.e., there is a
homeomorphism $\scr M\to M\times \CC$ such that $f_{\scr M}$ is sent
to the projection $M\times\CC\to \CC$. Then we have
\begin{equation}
    \label{eq:12}
    \psi_{f_{M}}\cC_{\scr M} = \cC_M[1].
\end{equation}
On the other hand, since we only need the restriction of
$\Pi_*\cC_{\scr M}$ to $\scr X\setminus X$, we may replace
$\pi_*\cC_{\scr M}$ by $\cC_{\scr X}$ as $\scr M\setminus M\to \scr
X\setminus X$ is an isomorphism.  (In fact, $\Pi_*\cC_{\scr M} =
\IC(\scr X)$ as it is known that $\Pi$ is small.) We thus have
\begin{equation}\label{eq:14}
    \psi_{f_{X}}\cC_{\scr X} = \pi_*(\cC_M)[1].
\end{equation}
We have a fundamental property of the nearby cycle functor:
$\psi_f[-1]$ sends perverse sheaves on $\scr X$ to perverse sheaves on
$X$. (See \cite[Cor.~10.3.13]{KaSha}.) Hence $\pi_*(\cC_M)$ is perverse.

\begin{Exercise}
    (1) Check \eqref{eq:16}.

    (2) Check \eqref{eq:12}.
\end{Exercise}

\subsection{Homology group of the Steinberg type variety}

Recall that the variety $Z_\cA$ is defined as a fiber product
$\cA_M\times_{X^T}M^T$, analog of Steinberg variety. The homology
group of Steinberg variety, or more generally the fiber product
$M\times_{X} M$ nicely fits with framework of perverse sheaves by
Ginzburg's theory \cite[\S8.6]{CG}. A starting point of the
relationship is an algebra isomorphism (not necessarily grading
preserving)
\begin{equation}\label{eq:8}
   H_*(Z) \cong \Ext^*_{D^b(X)}(\pi_*(\cC_M), \pi_*(\cC_M)).
\end{equation}
Moreover, if $M\to X$ is semi-small, $\pi_*(\cC_M)$ is a semisimple
perverse sheaf, and 
\begin{equation}\label{eq:9}
   H_{[0]}(Z) \cong \Hom_{D^b(X)}(\pi_*(\cC_M), \pi_*(\cC_M)).
\end{equation}
See \cite[Prop.~8.9.6]{CG}.

We have the following analog for $Z_\cA$:
\begin{Proposition}[cf.\ \protect{\cite[Lemma~4]{tensor2}}]\label{prop:ZA}
    We have a natural isomorphism
    \begin{equation}\label{eq:10}
        H_*(Z_\cA) \cong \Ext^*_{D^b(X^T)}(\pi^T_*(\cC_{M^T}),
        p_* j^! \pi_*(\cC_M)).
    \end{equation}
\end{Proposition}

Recall that $H_*(Z_\cA)$ is an $(H^\TT_{[*]}(Z),
H^\TT_{[*]}(Z^T))$-bimodule. Similarly the right hand side of the above isomorphism is a bimodule over
\[
  (\Ext^*_{D^b(X^T)}(\pi^T_*(\cC_{M^T}),\pi^T_*(\cC_{M^T})),
  \Ext^*_{D^b(X)}(\pi_*(\cC_{M}),\pi_*(\cC_{M}))).
\]
Under \eqref{eq:8} and its analog for $M^T$, the above isomorphism
respects the bimodule structure.

In view of (\ref{eq:8},\ref{eq:9},\ref{eq:10}), it is natural to ask
what happens if we replace the left hand side of \eqref{eq:10} by the
shifted degree $0$ part $H_{[0]}(Z_\cA)$, where the cycle $\Stab_\cC =
\cL$ in \thmref{thm:stable-envel} lives. Since $\pi^T\colon M^T\to
X^T$ is semi-small, $\pi^T_*(\cC_{M^T})$ is a semisimple perverse
sheaf. As we shall see in the next two subsections, the same is true
for $p_* j^!\pi_*(\cC_M)$. Then we have
\begin{equation*}
    H_{[0]}(Z_\cA) \cong \Hom_{D^b(X^T)}(\pi_*^T(\cC_{M^T}),
    p_* j^! \pi_*(\cC_M)).
\end{equation*}

Now $\Stab_\cC$ can be regarded as the {\it canonical\/} homomorphism
from $\pi_*^T(\cC_{M^T})$ to $p_* j^! \pi_*(\cC_M))$. Since it has
upper triangular and invertible diagonal entries, it is an
isomorphism. (In fact, we will see it directly in
\subsecref{sec:exactn-nearby-cycle}.) Therefore $\Stab_\cC$ is the
{\it canonical\/} isomorphism. In particular, we have an algebra
homomorphism
\begin{multline}
    \label{eq:20}
    H^\TT_{[*]}(Z) = \Ext^*(\pi_*(\cC_M), \pi_*(\cC_M))
    \xrightarrow{p_*j^!} \Ext^*(p_*j^!\pi_*(\cC_M), p_*j^!\pi_*(\cC_M))
\\
    \xrightarrow{\cong} \Ext^*(\pi^T_*(\cC_{M^T}), \pi^T_*(\cC_{M^T}))
    = H^\TT_{[*]}(Z^T).
\end{multline}
This homomorphism can be regarded as a `coproduct', as it will be
clear in \secref{sec:r-matrix}.

\begin{Exercise}
    Give a proof of \propref{prop:ZA}.
\end{Exercise}

\subsection{Hyperbolic
  restriction}\label{subsec:hyperb-restr}

The claim that $p_* j^! \pi_*(\cC_M)$ is a semisimple perverse sheaf
is a consequence of two results:
\begin{aenume}
      \item Braden's result \cite{Braden} on preservation of purity.
    \item Dimension estimate of fibers, following an idea of
  Mirkovic-Vilonen \cite{MV2}.
\end{aenume}
One may compare these results with similar statements for proper
pushforward homomorphisms:
\begin{aenume}
      \item the decomposition theorem, which is a consequence of
    preservation of purity.
    \item Semi-smallness implies the preservation of perversity.
\end{aenume}
In fact, the actual dimension estimate (b) required for $p_*
j^!\pi_*(\cC_M)$ is rather easy to check, once we use the nearby cycle
functor. The argument can be compared with one in
\subsecref{subsec:nearby-cycle-functor}.

We first treat (a) in this subsection.

Let $\cR_{X}$ be the {\it repelling set\/}, i.e., the attracting set for
the opposite chamber $-\cC$. We have the diagram
\begin{equation*}
    X^T
    \overset{p_-}{\underset{i_-}{\leftrightarrows}} \cR_{X}
    \overset{j_-}{\rightarrow} {X},
\end{equation*}
as for the attracting set.

\begin{Theorem}[\cite{Braden,2013arXiv1308.3786D}]\label{thm:Braden}
    We have a natural isomorphism $p_* j^! \cong p_{-!} j_-^*$ on
    $T$-equivariant complexes $D^b_T({X})$.
\end{Theorem}

This theorem implies the preservation of the purity for $p_* j^! =
p_{-!} j_-^*$ as $p_*$, $j^!$ increase weights while $p_{-!}$, $j^*$
decrease weights. In particular, semisimple complexes are sent to
semisimple ones. Thus $p_* j^! \pi_*(\cC_M)$ is semisimple.

\begin{Definition}
    The functor $p_* j^! = p_{-!}j_-^*$ is called the {\it hyperbolic
      restriction}.
\end{Definition}

Note that we have homomorphisms
\begin{equation}
    \label{eq:25}
    \begin{split}
        & H^*_T(X^T, p_* j^! \scF) \to H^*_T(X,\scF),
        \\
        & H^*_{T,c}(X,\scF) \to H^*_{T,c}(X^T, p_* j^!\scF)
    \end{split}
\end{equation}
for $\scF\in D^b_T(X)$ by adjunction. These become isomorphisms when
we take $\otimes_{H^*_T(\mathrm{pt})}
\operatorname{Frac}(H^*_T(\mathrm{pt}))$ by the localization theorem
in equivariant cohomology.

\subsection{Exactness by nearby cycle functors}\label{sec:exactn-nearby-cycle}

Consider one parameter deformation $f_{M}$, $f_{X}\colon \scr M$,
$\scr X\to \CC$ as in \subsecref{subsec:nearby-cycle-functor}. We have
the diagram
\begin{equation*}
    \scr X^T
    \overset{p_{\scr X}}{\underset{i_{\scr X}}{\leftrightarrows}} \cA_{\scr X}
    \overset{j_{\scr X}}{\rightarrow} \scr X,
\end{equation*}
as in \eqref{eq:13}. We have the hyperbolic restriction functor
$p_{\scr X*} j_{\scr X}^!$. 

We also have a family $f_{X^T}\colon \scr X^T\to \CC$.

The purpose of this subsection is to prove the following.

\begin{Proposition}\label{prop:nearby-hyperbolic}
    \textup{(1)} The restriction of $p_{\scr X*} j_{\scr X}^!\cC_{\scr
      X}$ to $\scr X^T\setminus X^T$ is canonically isomorphic to the
    constant sheaf $\cC_{\scr X^T\setminus X^T}$.

\textup{(2)} The nearby cycle functors commute with the hyperbolic
restriction:
\[
    p_* j^!\psi_{f_{X}} = \psi_{f_{X^T}}
    p_{\scr X*} j_{\scr X}^!.
\]
\end{Proposition}

\begin{Corollary}\label{cor:hyp-perverse}
    \textup{(1)} $p_* j^! \pi_* (\cC_M)$ is perverse.

    \textup{(2)} The isomorphism $\cC_{\scr X^T}|_{\scr X^T\setminus X^T}
    \xrightarrow{\cong}
    p_{\scr X*} j_{\scr X}^!\cC_{\scr
      X}|_{\scr X^T\setminus X^T}$ induces an isomorphism
    $\pi^T_*(\cC_{M^T})\xrightarrow{\cong} p_* j^! \pi_*(\cC_M)$.
\end{Corollary}

In fact, $\pi_*(\cC_M) = \psi_{f_{X}}\cC_{\scr X}[-1]$ by
\eqref{eq:14}. Therefore \propref{prop:nearby-hyperbolic}(2) implies
$p_* j^!\pi_*(\cC_M) = \psi_{f_{X^T}} p_{\scr X*} j_{\scr X}^!
\cC_{\scr X}[-1]$. Now we can replace $p_{\scr X*} j_{\scr X}^!
\cC_{\scr X}[-1]$ by $\cC_{\scr X^T}[-1]$ by (1) over $\scr
X^T\setminus X^T$. Since $\psi_{f_{X^T}}[-1]$ sends a perverse sheaf
to a perverse sheaf, we get the assertion (1).

Applying $\psi_{f_{X^T}}$ to $\cC_{\scr X^T}|_{\scr X^T\setminus X^T}
    \xrightarrow{\cong}
    p_{\scr X*} j_{\scr X}^!\cC_{\scr
      X}|_{\scr X^T\setminus X^T}$, we get an isomorphism
\[
   \pi^T_*(\cC_{M^T}) \cong \psi_{f_{X^T}}\cC_{\scr X^T}
   \xrightarrow{\cong}
   \psi_{f_{X^T}} p_{\scr X*} j_{\scr X}^!\cC_{\scr X}
   \cong 
   p_* j^! \pi_*(\cC_M).
\]
This is (2).

Moreover, the isomorphism (2) coincides with one given by the class
$\cL\in H^T_{[0]}(Z_\cA) \cong \Hom_{D^b(X^T)}(\pi_*^T(\cC_{M^T}),
p_* j^! \pi_*(\cC_M))$. This is clear from the definition, and the
fact that the nearby cycle functor coincides with the specialization
of Borel-Moore homology groups.

We have canonical homomorphisms $\IC(X^T)\to \pi^T_*(\cC_{M^T})$ and
$\pi_*(\cC_M)\to\IC(X)$ as the inclusion and projection of direct
summands. Therefore we have a canonical homomorphism
\begin{equation*}
    \IC(X^T) \to p_* j^! \IC(X).
\end{equation*}
Note that $\IC(X^T)$, $p_* j^! \IC(X)$ make sense even when we do not
have a symplectic resolution. Therefore we can ask whether there is a
canonical homomorphism $\IC(X^T) \to p_* j^! \IC(X)$, though it is
not precise unless we clarify in what sense it is canonical.
In \subsecref{subsec:hyperb-restr-affine} and
\secref{lec:perv-sheav-uhlenb} below, we will construct canonical
homomorphisms.

\begin{History}
    A different but similar proof of \corref{cor:hyp-perverse}(1) was
    given by Varagnolo-Vasserot \cite{VV2} for hyperbolic restrictions
    for quiver varieties, and it works for symplectic resolutions. It
    uses hyperbolic semi-smallness for the case of symplectic
    resolution on one hand, and arguments in \cite[3.7]{MR792706} and
    \cite[4.7]{Lu-can2} on the other hand. The latter two originally gave
    decomposition of restrictions of character sheaves, restrictions
    of perverse sheaves corresponding to canonical bases respectively.
    (Semisimplicity follows from a general theorem
    \ref{thm:Braden}. But decomposition must be studied in a different
    way.)
\end{History}

\begin{proof}[Proof of \propref{prop:nearby-hyperbolic}]
    (1) The key is that $\scr M\setminus M$ is isomorphic to $\scr
    X\setminus X$. We have the decomposition $\scr M^T = \bigsqcup
    \scr F_\alpha$ corresponding to $M^T = \bigsqcup F_\alpha$ to
    connected components. Then we have the induced decomposition
    \begin{equation*}
        \cA_{\scr X} \setminus \cA_X 
        = \bigsqcup p_{\scr X}^{-1}(\scr F_\alpha\setminus F_\alpha)
    \end{equation*}
    to connected components, such that each $p_{\scr X}^{-1}(\scr
    F_\alpha\setminus F_\alpha)$ is an affine bundle over $\scr
    F_\alpha\setminus F_\alpha$. The same holds for $\cR_{\scr
      X}\setminus \cR_X = \bigsqcup p_{\scr X-}^{-1}(\scr F_\alpha\setminus
    F_\alpha)$.

    Moreover $p_{\scr X}^{-1}(\scr F_\alpha\setminus F_\alpha)$ is a
    smooth closed subvariety of $\scr X\setminus X$. Its codimension
    is equal to the dimension of fibers of $p_{\scr X-}^{-1}(\scr
    F_\alpha\setminus F_\alpha)$.

    Thus the hyperbolic restriction $p_{\scr X*}j_{\scr X}^!\cC_{\scr X}$ is the
    constant sheaf $\cC_{\scr X^T\setminus X^T}$ up to shifts. (Shifts
    could possibly different on components.)

    Now normal bundles of $\scr F_\alpha\setminus F_\alpha$ to
    $p_{\scr X}^{-1}(\scr F_\alpha\setminus F_\alpha)$ and $p_{\scr
      X-}^{-1}(\scr F_\alpha\setminus F_\alpha)$ are dual vector
    bundles with respect to the symplectic structure. In particular
    the dimension of fibers are equal. This observation implies that
    shifts are unnecessary, $p_{\scr X*}j_{\scr X}^!\cC_{\scr X}$ is
    equal to $\cC_{\scr X^T\setminus X^T}$. Moreover the isomorphism
    is given by the restriction of the constant sheaves and the Thom
    isomorphism for the cohomology group of an affine
    bundle. Therefore it is canonical.

    (2) Recall that the nearby cycle functor $\psi_{f_{X}}$ is
    given by $r_{X}^* \tilde c_{X*} \tilde c_{X}^*$ as in
    \eqref{eq:15}, where we put subscripts $X$ to names of maps to
    indicate we are considering the family $\scr X\to \CC$.

    We replace $p_* j^!$ by $p_{-!} j_-^*$ by \thmref{thm:Braden}. We
    have the commutative diagram
    \begin{equation*}
        \begin{CD}
            X^T @<p_-<< \cR_X @>j_->> X\\
            @V{r_{X^T}}VV @VV{r_{\cR_X}}V @VV{r_X}V \\
            \scr X^T @<<p_{\scr X-}< \cR_{\scr X} @>>j_{\scr X-}> \scr X,
        \end{CD}
    \end{equation*}
    where both squares are cartesian. Thus $p_{-!} j_-^* r^*_X =
    p_{-!} r_{\cR_X}^* j_{\scr X-}^* = r_{X^T}^* p_{\scr X-!} j_{\scr
      X-}^*$. Namely the hyperbolic restriction commutes with the
    restrictions $r^*_X$, $r^*_{X^T}$ to $0$-fibers.

    Next we replace back $p_{\scr X-!} j_{\scr X-}^*$ to $p_{\scr X*}
    j_{\scr X}^!$ and consider the diagram
    \begin{equation*}
        \begin{CD}
            \scr X^T @<p_{\scr X}<< \cA_{\scr X} @>j_{\scr X}>> \scr X \\
            @A{\tilde c_{X^T}}AA @AA{\tilde c_{\cA_X}}A @AA{\tilde c_X}A \\
            \tilde{\scr X}^{T*} @<<p_{\tilde{\scr X}^*}< \tilde{\cA}^*_{\scr X}
            @>>j_{\tilde{\scr X}^*}> 
            \tilde{\scr X}^*.
        \end{CD}
    \end{equation*}
    The bottom row is pull back to the universal cover $\tilde{\CC}^*$
    of the upper row. This is commutative and both squares are
    cartesian. Thus we have
\(
   p_{\scr X*} j_{\scr X}^! \tilde c_{X*} =
   p_{\scr X*} \tilde c_{\cA_X*}j_{\tilde{\scr X}^*}^! 
   = \tilde c_{X^T*} (p_{\tilde{\scr X}^*})_* j_{\tilde{\scr X}^*}^!.
\)
Thus the hyperbolic restriction commutes with the pushforward for
the coverings $\tilde c_X$, $\tilde c_{X^T}$.

Finally we commute the hyperbolic restriction with pullbacks for
$\tilde c_X$, $\tilde c_{X^T}$. We do not need to use
\thmref{thm:Braden} as we saw $j_{\tilde{\scr X}^*}$, $p_{\tilde{\scr
    X}^*}$ are (union of) embedding of smooth closed subvarieties and
projections of vector bundles. Therefore $*$ and $!$ are the same up
to shift.
(Also \thmref{thm:Braden} is proved for algebraic varieties, and is
not clear whether the proof works for $\tilde{\scr X}^*$ in
general. This problem disappears if we consider the nearby cycle
functor in algebraic context.) This finishes the proof of (2).
\end{proof}

\subsection{Hyperbolic semi-smallness}

Looking back the proof of \propref{prop:nearby-hyperbolic}, we see
that a key observation is the equalities $\rank \cA_X = \rank \cR_X =
\codim_X X^T/2$. (More precisely restriction to each component of $X^T$.)

In order to prove the exactness of the hyperbolic restriction functor
in more general situation, in particular, when we do not have
symplectic resolution, we will introduce the notion of hyperbolic
semi-smallness in this subsection.

The terminology is introduced in \cite{2014arXiv1406.2381B}, but the
concept itself has appeared in \cite{MV2} in the context of the
geometric Satake correspondence.

Let $X = \bigsqcup X_\alpha$ be a stratification of $X$ such that
$i_\alpha^! \IC(X)$, $i_\alpha^* \IC(X)$ are shifts of locally
constant sheaves.
Here $i_\alpha$ denotes the inclusion $X_\alpha\to X$.
We suppose that $X_0$ is the smooth locus of $X$ as a convention.
\begin{NB}
    Therefore $i_{0}^! \IC(X) = i_0^* \IC(X) = \cC_{X_0}
    = \CC_{X_0}[\dim X]$.
\end{NB}%

We also suppose that the fixed point set $X^T$ has a stratification
$X^T = \bigsqcup Y_\beta$ such that the restriction of $p$ to
$p^{-1}(Y_\beta)\cap X_\alpha$ is a topologically locally trivial
fibration over $Y_\beta$ for any $\alpha$, $\beta$ (if it is
nonempty). We assume the same is true for $p_-$. We take a point
$y_\beta\in Y_\beta$.

\begin{Definition}\label{def:hypsemismall}
    We say $(p,j)$ is {\it hyperbolic semi-small\/} if the following
    two estimates hold
    \begin{equation}\label{eq:est}
        \begin{split}
            & \dim p^{-1}(y_\beta)\cap X_\alpha\le \frac12 
            (\dim X_\alpha - \dim Y_\beta),
\\
            & \dim p_-^{-1}(y_\beta)\cap X_\alpha\le \frac12 
            (\dim X_\alpha - \dim Y_\beta).
        \end{split}
    \end{equation}
\end{Definition}

These conditions are for $p$ and $j$. Nevertheless we will often say
the functor $p_*j^!$ is hyperbolic semi-small for brevity. There
should be no confusion if we use $p_* j^!$ only for $\IC(X)$.

Suppose $X$ is smooth (and $X = X_0$). Then $X^T$ is also smooth. We
decompose $X^T = \bigsqcup Y_\beta$ into connected components as
usual. Then the above inequalities must be equalities, i.e., $\rank
\cA_X|_{Y_\beta} = \rank \cR_X|_{Y_\beta} = \codim_X Y_\beta/2$. They
are the condition which we have mentioned in the beginning of this
subsection.

\begin{NB}
    Let $i_{y_\beta}$ denote the inclusion of $y_\beta$ to
    $Y_\beta$. We consider components of $\Phi(\IC(X))$ involving
    local systems on $Y_\beta$. To compute the former, we study
    $H^{\dim Y_\beta}(i_{y_\beta}^!)$. We use the base change that the
    fiber is
    \begin{equation*}
        \begin{split}
            & H_{\dim X - \dim Y_\beta}(p^{-1}(y_\beta)\cap X_0) 
    = H^{-\dim X + \dim Y_\beta}(p^{-1}(y_\beta)\cap X_0, \mathbb D_{p^{-1}(y_\beta)\cap X_0})
\\
    =\; & H^{-\dim X + \dim Y_\beta}(
    p^{-1}(y_\beta)\cap X_0,
    \tilde j^! \mathbb D_{X_0})
\\
    =\; & H^{\dim Y_\beta}(
    p^{-1}(y_\beta)\cap X_0, \tilde j^! \cC_{X_0})
    = H^{\dim Y_\beta}(i_{y_\beta}^!, p_* j^! \cC_{X_0}),
        \end{split}
    \end{equation*}
where $\tilde j$ is the inclusion of $p^{-1}(y_\beta)\cap X_0$.
For the latter, we study $H^{-\dim Y_\beta}(i_{y_\beta}^*)$. We have
\begin{equation*}
    \begin{split}
    & H^{\dim X - \dim Y_\beta}_c(p_-^{-1}(y_\beta)\cap X_0)
    = H^{- \dim Y_\beta}_c(p_-^{-1}(y_\beta)\cap X_0, \tilde j_-^* \cC_{X_0})
\\
   =\; & H^{- \dim Y_\beta}(i^*_{y_\beta} (p_-)_! j_-^* \cC_{X_0}).
    \end{split}
\end{equation*}
\end{NB}%
Note that $p^{-1}(y_\beta)\cap X_0$ and $p_-^{-1}(y_\beta)\cap X_0$
are at most $(\dim X - \dim Y_\beta)/2$-dimensional if $\Phi$ is
hyperbolic semi-small. 
We consider homology and cohomology groups $H_{\dim X - \dim
  Y_\beta}(p^{-1}(y_\beta)\cap X_0)$ and $H^{\dim X - \dim
  Y_\beta}_c(p_-^{-1}(y_\beta)\cap X_0)$. They have bases given by
$(\dim X - \dim Y_\beta)/2$-dimensional irreducible components of
$p^{-1}(y_\beta)\cap X_0$ and $p_-^{-1}(y_\beta)\cap X_0$
respectively.
When $y_\beta$ moves in $Y_\beta$, they form local systems where
$\pi_1(Y_\beta)$ acts by permuting irreducible components.
Let us decompose them as $\bigoplus H_{\dim X - \dim
  Y_\beta}(p^{-1}(y_\beta)\cap X_0)_\chi\otimes\chi$ and $\bigoplus
H^{\dim X - \dim Y_\beta}_c(p_-^{-1}(y_\beta)\cap
X_0)_\chi\otimes\chi$, where $\chi$ is a simple local system $\chi$ on
$Y_\beta$.

\begin{Theorem}\label{thm:hypsemismall}
    Suppose $(p,j)$ is hyperbolic semi-small. Then $p_*j^!(\IC(X))$ is
    perverse and it is isomorphic to
    \begin{equation*}
        \bigoplus_{\beta,\chi} \IC(\overline{Y_\beta},\chi)\otimes
        H_{\dim X-\dim Y_\beta}(p^{-1}(y_\beta)\cap X_0)_\chi.
    \end{equation*}
Moreover, we have an isomorphism
\begin{equation*}
    H_{\dim X-\dim Y_\beta}(p^{-1}(y_\beta)\cap X_0)_\chi
    \cong
    H^{\dim X-\dim Y_\beta}_c(p_-^{-1}(y_\beta)\cap X_0)_\chi.
\end{equation*}
\end{Theorem}

The proof is similar to one in \cite[Theorem~3.5]{MV2}, hence the
detail is left as an exercise for the reader. In fact, we only use the
case when $X^T$ is a point, and the argument in detail
for that case was given in \cite[Th.~A.7.1]{2014arXiv1406.2381B}.

\begin{Exercise}\label{ex:hyperb-example}
    Consider $X = \CC^2$ with the hyperbolic action $(x,y) \mapsto
    (tx, t^{-1}y)$ for $t\in\CC^\times_{\rm hyp}$. We choose a chamber
    $\RR_{>0}$. Check that $(S^d X, \CC^\times_{\rm hyp})$ is
    hyperbolic semi-small. For a partition $\lambda$ of $d$, we take
    an irreducible representation $\rho$ of $\Stab(\lambda)$.
    ($\Stab(\lambda)$ was introduced in
    \subsecref{sec:inters-cohom-group}.) We consider it as a local
    system on $S_\lambda(\CC^2)$, and take the associated $IC$ sheaf
    $\IC(\overline{S_\lambda(\CC^2)},\rho)$. Compute the hyperbolic
    restriction of $\IC(\overline{S_\lambda(\CC^2)},\rho)$.
\end{Exercise}

\begin{Exercise}
    Give a proof of \thmref{thm:hypsemismall}. Assume $X^T$ is a
    single point for brevity.
\end{Exercise}

\subsection{Hyperbolic restriction for affine Grassmannian}
\label{subsec:hyperb-restr-affine}

Recall the affine Grassmannian $\Gr_G = G((z))/G[[z]]$ is defined for
a complex reductive group $G$. Perverse sheaves on $\Gr_G$ are related
to finite dimensional representations of $G^\vee$, the Langlands dual
of $G$ by the geometric Satake correspondence.
Let us review how the hyperbolic restriction functor appears in
Mirkovic-Vilonen's work \cite{MV2} on the geometric Satake
correspondence. This topic as well as further topics on the geometric
Satake can be found in Zhu's lectures in the same volume. So we will
omit details.

Let $\Perv_G(\Gr_G)$ be the abelian category of $G[[z]]$-equivariant
perverse sheaves on $\Gr_G$. Let us consider the natural $T$-action on
$\Gr_G$, where $T$ is a maximal torus of $G$. Then it is well-known
that the fixed point set $\Gr_G^T$ is the affine Grassmannian of $T$,
which is nothing but the coweight lattice
$\Hom_{\mathrm{grp}}(\CC^\times,T)$.
Let us choose the coweight $2\rho^\vee\colon\CC^\times\to T$, the sum
of positive coroots. (This is a conventional choice, and is not
essential.)
Let $\nu$ be a coweight and let $z^\nu$ be the corresponding
$T$-fixed point in $\Gr_G$. Let $S_\nu$ be the corresponding
attracting set
\begin{equation*}
    S_\nu = \{ x\in\Gr_G \mid \lim_{t\to 0} 2\rho^\vee(t)x = z^\nu \}.
\end{equation*}
We consider the diagram
\begin{equation}\label{eq:21}
    \{ z^\nu\} 
    \overset{p_\nu}{\underset{i_\nu}{\leftrightarrows}} S_\nu
    \overset{j_\nu}{\rightarrow} \Gr_G
\end{equation}
as in \eqref{eq:13}, and the corresponding hyperbolic restriction $p_*
j^!$. Here we need to treat fixed points separately. In fact, we do
not have the projection $p$ on the union of $S_\nu$, considering
it as a closed subscheme of $\Gr_G$.
This problem has not occur, as we have only considered hyperbolic
restriction functors for affine varieties so far.

Let $\lambda$ be a dominant weight and $\Gr_G^\lambda$ be the
$G[[z]]$-orbit through $z^\lambda$. Let $\overline{\Gr}_G^\lambda$ be
its closure. It is the union of $\Gr_G^\mu$ over dominant weights
$\mu\le\lambda$. Then it was proved in \cite[Th.~3.2]{MV2} that the
restriction of \eqref{eq:21} to $\overline{\Gr}_G^\lambda$ is
hyperbolic semi-small. Therefore by \thmref{thm:hypsemismall},
$p_{\nu*}j_\nu^!  \IC(\overline{\Gr}_G^\lambda)$ is perverse, in other
words it is concentrated in degree $0$, as $z^\nu$ is a point. In
Mirkovic-Vilonen's formulation of geometric Satake correspondence, the
equivalence $\Perv_G(\Gr_G)\cong \operatorname{Rep} G^\vee$ is
designed so that we have an isomorphism
\begin{equation*}
    H^0(p_{\nu*}j_\nu^! \IC(\overline{\Gr}_G^\lambda))
    \cong V_\nu(\lambda),
\end{equation*}
where $V(\lambda)$ is the irreducible representation of $G^\vee$ with
highest weight $\lambda$, and $V_\nu(\lambda)$ is its weight $\nu$
subspace.  (See \cite[Cor.~7.4]{MV2}.) In other words, the
decomposition in \thmref{thm:hypsemismall} gives the weight space
decomposition.
Moreover the construction also gives a basis of $H^0(p_{\nu*}j_\nu^!
\IC(\overline{\Gr}_G^\lambda))$, and hence of $V_\nu(\lambda)$ given
by irreducible components of $S_\nu\cap \Gr_G^\lambda$.

\begin{Remark}\label{rem:duality}
    A resemblance between pushforward of (partial) resolution
    morphisms and hyperbolic restrictions is formal at this stage. But
    the following observation, given in \cite{Na-branching} in the
    context of geometric Satake correspondence for loop groups
    \cite{braverman-2007}, suggests a deeper relation is hidden. Also
    it is natural to conjecture that a similar relation exists in more
    general framework of the \emph{symplectic duality}
    \cite{2014arXiv1407.0964B} and Higgs/Coulomb branches of gauge
    theories \cite{2015arXiv150303676N}. In particular, the following
    observation is discussed in the framework of the symplectic
    duality in \cite[\S10.5]{2014arXiv1407.0964B}.

    It is known that the representation theory of $G$ is related to
    geometry in two ways. One is through $\Gr_G$ as we have just
    reviewed. Another is via quiver varieties of finite types. In
    these connections, partial resolution and hyperbolic restrictions
    appear in different places:

\begin{table*}
	\centering
    \begin{tabular}[h]{c|c|c}
        & hyperbolic restrictions & partial resolution \\
        \hline
        affine Grassmannian & weight spaces & tensor products \\
        quiver varieties & tensor products & weight spaces 
    \end{tabular}
\end{table*}

For $\Gr_G$, a hyperbolic restriction is used to realize weight spaces
as we have reviewed just above. For quiver varieties, weight spaces are
top homology groups of central fibers of affinization morphisms
$\pi\colon M\to X$ (\cite{Na-alg}). It means that weights spaces are
isotypical components of $\cC_{0}$ in $\pi_*(\cC)$. Here $\cC_0$ is
the constant sheaf at the point $0$.

On the other hand, a hyperbolic restriction appears for realization of
tensor products for quiver varieties. (We will see it for Gieseker
spaces in \subsecref{subsubsec:heisenberg-operators}.) For $\Gr_G$,
tensor products are realized by the convolution diagram (see
\cite[\S4]{MV2}). We will not review the construction, but we just
mention that tensor product multiplicities are understood as dimension
of isotypical components of $\IC(\overline{\Gr}^\lambda_G)$ in the
pushforward of a certain perverse sheaf under a semi-small morphism.

Affine Grassmannian (more precisely slices of $\Gr^\lambda_G$ in the
closure of another $G[[z]]$-orbit) and quiver varieties of finite
types are examples of Coulomb and Higgs branches of common gauge
theories (see \cite{2015arXiv150303676N}). It is natural to expect
that similar exchanges of two functors appear for pairs of \emph{dual}
symplectic varieties and Coulomb/Higgs branches. Toric hyper-K\"ahler manifolds provide us other examples \cite{BraMau}.

Let us consider a $3$-dimensional SUSY gauge theory associated with
$(H_c, \mathbf M)$ in general.
A maximal torus of the normalizer of $H_c$ in $\grpSp(\mathbf M)$
(flavor symmetry group) gives a torus action on the Higgs branch, and
(partial) resolution on the Coulomb branch
\cite[\S5]{2015arXiv150303676N}. On the other hand, the Pontryagin
dual of the fundamental group of the gauge group $H_c$ gives a torus
action on the Coulomb branch and (partial) resolution on the Higgs
branch \cite[\S4(iv)]{2015arXiv150303676N}. Therefore the exchange of
partial resolution and hyperbolic restriction are natural.
\end{Remark}

\section{$R$-matrix for Gieseker partial compactification}\label{sec:r-matrix}

In this lecture we introduce $R$-matrices of stable envelopes for symplectic
resolutions. Then we study them for Gieseker partial compactification, in particular relate them to the Virasoro algebra.

\subsection{Definition of $R$-matrix}

We consider the setting in \subsecref{subsec:setting}.

Let $\iota\colon M^T\to M$ be the inclusion. Then
$\iota^*\circ\Stab_\cC\in\End (H^{[*]}_\TT(M^T))$ is upper triangular,
and the diagonal entries are multiplication by
$e(L_\alpha^-)$. Since $e(L_\alpha^-)|_{H^*_{T}(\mathrm{pt})}$
is nonzero, $\iota^*\circ\Stab_\cC$ is invertible over
$\bF_T = \CC(\operatorname{Lie}\TT)$. By the localization theorem for equivariant cohomology groups, $\iota^*$ is also invertible. Therefore $\Stab_\cC$ is invertible over $\bF_T$.

\begin{Definition}
    Let $\cC_1$, $\cC_2$ be two chambers. We define the {\it $R$-matrix\/} by
    \begin{equation*}
        R_{\cC_1,\cC_2} = \Stab_{\cC_1}^{-1} \circ \Stab_{\cC_2}
        \in \End(H^{[*]}_\TT(M^T))\otimes\bF_T.
    \end{equation*}
\end{Definition}

\begin{Example}
    Consider $M = T^*\proj^1$ with $T = \CC^\times$, $\TT =
    \CC^\times\times\CC^\times$ as before. We denote the corresponding
    equivariant variables by $u$, $\hbar$ respectively. (So
    $H^*_T(\mathrm{pt}) = \CC[\operatorname{Lie}T] = \CC[u]$,
    $H^*_\TT(\mathrm{pt}) = \CC[\operatorname{Lie}\TT] =
    \CC[u,\hbar]$.) Choose a chamber $\cC = \{ u > 0\}$. Then
    $R_{-\cC,\cC}$ is the middle block of Yang's $R$-matrix
    \begin{equation*}
        R = 1 - \frac{\hbar P}u, \qquad
        P = \sum_{i,j=1}^2 e_{ij}\otimes e_{ji},
    \end{equation*}
    where $e_{ij}$ is the matrix element acting on $\CC^2$ up to
    normalization.
\end{Example}

\begin{Exercise}
    Check this example.
\end{Exercise}

It is customary to write the $R$-matrix as $R(u)$ to emphasize its
dependence on $u$.
This variable $u$ is called a {\it spectral parameter\/} in the
context of representation theory of Yangian.

The reason why it should be called the $R$-matrix is clear if we
consider the `coproduct' in \eqref{eq:20}. Let us denote it
$\Delta_\cC$ as it depends on a choice of a chamber $\cC$. If we take
two chambers $\cC_1$, $\cC_2$ as above, we have {\it two\/} coproducts
$\Delta_{\cC_1}$, $\Delta_{\cC_2}$. Then we can regard
$H^{[*]}_\TT(M^T)$ as an $H^\TT_{[*]}(Z)$-module in two ways, through $\Delta_{\cC_1}$ or $\Delta_{\cC_2}$. Let us distinguish two modules as
$H^{[*]}_\TT(M^T)_{\cC_1}$, $H^{[*]}_\TT(M^T)_{\cC_2}$. Then the $R$-matrix
\begin{equation*}
    R_{\cC_1,\cC_2}\colon H^{[*]}_\TT(M^T)_{\cC_2}\otimes \bF_T
    \to H^{[*]}_\TT(M^T)_{\cC_1}\otimes \bF_T
\end{equation*}
is an intertwiner.

Yang's $R$-matrix above originally appeared in a quantum many-body
problem. Subsequently it is understood that $R(u)$ is an intertwiner
between $V(a_1)\otimes V(a_2)$ and $V(a_2)\otimes V(a_1)$, where
$V(a_i)$ is a $2$-dimensional evaluation representation of the Yangian
$Y(\algsl_2)$ associated with $\algsl_2$. Conversely the Yangian
$Y(\algsl_2)$ can be constructed from the $R$-matrix by the so-called
RTT relation. See \cite{MR2355506}. Maulik-Okounkov \cite{MO} apply
this construction to the $R$-matrix for $\Gi{d}$, and more generally
for quiver varieties. For quiver varieties of type $ADE$, this
construction essentially recovers the usual Yangian $Y(\g)$ associated
with $\g$ \cite{MR3192991}. But it is a new Hopf algebra for quiver
varieties of other types.

We will not review the RTT construction, as the $\scr W$-algebra
associated with $\g$ (not of type $A$) is something different.

For the Yangian, we have two coproducts $\Delta$, $\Delta^{\text{op}}$
where the latter is given by exchanging two factors in tensor
products. The $R$-matrix is an intertwiner of two coproducts.

\subsection{Yang-Baxter equation}

Suppose that $T$ is two dimensional such that $\operatorname{Lie}T =
\{ a_1 + a_2 + a_3 = 0\}$. We suppose that there are six chambers
given by hyperplanes $a_1 = a_2$, $a_2 = a_3$, $a_3 = a_1$ as for Weyl
chambers for $\algsl(3)$. The cotangent bundle of the flag variety of
$\SL(3)$ and $\Gi[3]{d}$ are such examples by
Exercise~\ref{ex:chamber}.
We factorize the $R$-matrix from $\cC$ to $-\cC$ in two ways to get
the Yang-Baxter equation
\begin{multline*}
    R_{12}(a_1 - a_2) R_{13}(a_1 - a_3) R_{23}(a_2 - a_3)
   = R_{23}(a_2 - a_3) R_{13}(a_1 - a_3) R_{12}(a_1 - a_2).
\end{multline*}
See Figure~\ref{fig:YB}.

\begin{figure}[htbp]
    \centering
\begin{tikzpicture}
    \draw[thick] (-2,0) -- (2,0) node [right] {$a_1=a_2$};
    \draw[<-,dotted] (1.5,0) -- (2.5,-0.4) node [below right] {$\cC'$};
    \draw[thick] (240:2) -- (60:2) node [right] {$a_2=a_3$};
    \draw[thick] (120:2) -- (300:2) node [right] {$a_1=a_3$};
    \path (30:2) node {$\cC$};
    \path (330:2) node {$\cC_2$};
    \path (210:2.3) node {$-\cC$};
    \draw[thick,->] (45:1) arc (45:75:1) node [above] {$R_{23}$};
    \draw[thick,->] (105:1) arc (105:135:1) node [anchor=south east] {$R_{13}$};
    \draw[thick,->] (165:1) arc (165:195:1) node [anchor=east] {$R_{12}$};
    \draw[thick,->] (15:1) arc (15:-15:1) node [anchor=west] {$R_{12}$};
    \draw[thick,->] (315:1) arc (315:285:1) node [anchor=north] {$R_{13}$};
    \draw[thick,->] (255:1) arc (255:225:1) node [anchor=north east] {$R_{23}$};
\end{tikzpicture}
    \caption{Yang-Baxter equation}
    \label{fig:YB}
\end{figure}
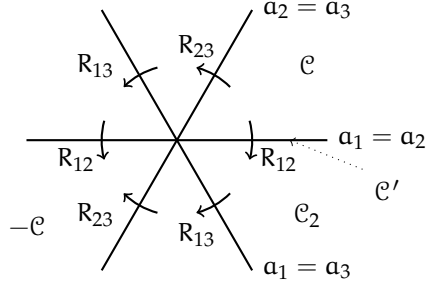

Let $\cC$, $\cC_2$ be chambers as in Figure~\ref{fig:YB}. Let $\cC'$
be a half line separating $\cC$, $\cC_2$, considered as a chamber for
a subtorus $T'$ corresponding to $a_1 = a_2$. Then by
\propref{prop:factorStab} we have
\begin{equation}
    \label{eq:22}
    R_{12} = R_{\cC_2,\cC} = \Stab^{-1}_{\cC_2/\cC'} \circ \Stab_{\cC/\cC'}.
\end{equation}
Here $\Stab_{\cC_2/\cC'}$, $\Stab_{\cC/\cC'}$ are maps from
$H^{[*]}_\TT(M^T)$ to $H^{[*]}_\TT(M^{T'})$. Therefore the right
hand side is the $R$-matrix for the action of $T/T'$ on $M^{T'}$. For
the cotangent bundle of the flag variety of $\SL(3)$ and $\Gi[3]{d}$,
it is the $R$-matrix for $\SL(2)$ and $\Gi[2]{d}$ respectively.


\subsection{Heisenberg operators}\label{subsubsec:heisenberg-operators}

In the remainder of this lecture, we study the $R$-matrix for the case
$X = \Gi{d}$. But we first need to consider behavior of Heisenberg
operators under the stable envelope.

Let us observe that the Heisenberg operators $P^\Delta_{-m}(\alpha)$
sends $H^\TT_{[*]}(\cA_{\Gi{d}})$ to
$H^\TT_{[*-\deg\alpha]}(\cA_{\Gi{d+m}})$. This is because
$q_1(q_2^{-1}(\cA_{\Gi{d}}\times\CC^2))\subset
\cA_{\Gi{d+m}}$. Therefore the direct sum
\(
    \bigoplus_d H^\TT_{[*]}(\cA_{\Gi{d}}) 
\) 
is a module over the Heisenberg algebra. Recall we have isomorphisms
\begin{equation*}
    \bigoplus_d H^\TT_{[*]}(\cA_{\Gi{d}}) 
    \underset{\eqref{eq:4}}{\cong}
    \bigoplus_d H^{[*]}_\TT((\Gi{d})^T)
    \underset{\eqref{eq:31}}{\cong}
    \bigotimes_{i=1}^r \bigoplus_{d_i=0}^\infty H^{[*]}_\TT((\CC^2)^{[d_i]}).
\end{equation*}
Note that the rightmost space is the tensor product of $r$ copies of
the Fock space. Therefore it is a representation of the product of $r$
copies of the Heisenberg algebra.

We compare two Heisenberg algebra representations.

\begin{Proposition}[\protect{\cite[Th.~12.2.1]{MO}}]\label{prop:12.2.1}
    The operator $P^\Delta_{-m}(\alpha)$ is the diagonal Heisenberg
    operator
    \begin{equation*}
        \sum_{i=1}^r \id\otimes\dots\otimes\id\otimes
        \underbrace{P_{-m}(\alpha)}_{\text{\rm $i^{\mathrm{th}}$ factor}}
        \otimes\id\otimes\dots\otimes
        \id.
    \end{equation*}
\end{Proposition}

Thus $P^\Delta_{-m}(\alpha)$ is a \emph{primitive} element with
respect to the coproduct $\Delta$.

\begin{Exercise}
    Give a proof of the above proposition.
\end{Exercise}

\subsection{\texorpdfstring{$R$}{R}-matrix as a Virasoro intertwiner}
\label{subsec:R-matrix-as-a-Virasoro}

Consider the $R$-matrix for $\Gi{d}$.  By \eqref{eq:22}, it is enough to consider the $r=2$ case. 

We mostly consider the localization
\[
H^{[*]}_\TT((\Gi[2]{d})^T)\otimes_{\bA_T}\bF_T,
\]
from now. It means that we consider the localization at the
\emph{generic point} in $\Spec\bA_T = \Spec \CC[\vec{a}, \ve_1, \ve_2]$.
By \eqref{eq:31} 
$\bigoplus H^{[*]}_\TT((\Gi[2]{d})^T)$ is isomorphic to the tensor product of two copies of Fock space.
Let us denote the Heisenberg generator for the first and second
factors by $P_m^{(1)}$, $P_m^{(2)}$ respectively. Here we take $1\in
H^{[*]}_\TT(\CC^2)$ for the cohomology class $\alpha$ in
\subsecref{subsec:heis-algebra-via}, and omit $(1)$ from the notation.
Note $P_n^{(1)}$, $P_n^{(2)}$ are well-defined operators on
$H^{[*]}_\TT((\Gi[2]{d})^T)$ only for $n < 0$ (a creation
operator). To make sense it also for $n > 0$, we need to consider the
localized equivariant cohomology group as above.
Note also that we have $\langle 1,1\rangle = - 1/\ve_1\ve_2$ in the
commutation relation. This does not make sense unless $\ve_1$, $\ve_2$
are invertible, i.e., in $\CC(\ve_1,\ve_2)$. We will discuss the
integral form in \secref{sec:W}.

We have $P^\Delta_m = P^{(1)}_m + P^{(2)}_m$ by \propref{prop:12.2.1}.
Since $P^\Delta_m$ is defined by a correspondence which makes sense
without going to fixed points $(\Gi[2]{d})^T$, it commutes with the
$R$-matrix.
Therefore the $R$-matrix should be described by anti-diagonal Heisenberg generators $P^{(1)}_m - P^{(2)}_m$. Let us denote them by $P^-_m$.

Let us denote the corresponding Fock spaces by $F^\Delta$ and
$F^-$. Therefore we have $\bigoplus H^{[*]}_\TT((\Gi[2]{d})^T) \cong
F^\Delta\otimes F^-$. The above observation means that the $R$-matrix
is of a form $\id_{F^\Delta}\otimes R'$ for some operator $R'$ on
$F^-$.

We now characterize $R'$ in terms of the Virasoro algebra, acting on $F^-$ by the well-known Feigin-Fuchs construction. 

Let us recall the Feigin-Fuchs construction.
We put
\begin{equation*}
   P^-_0 \defeq \frac1{\ve_1\ve_2} \left( a_1 - a_2 - (\ve_1+\ve_2) \right).
\end{equation*}
This element is central, i.e., it commutes with all other $P^-_m$. In particular, the Heisenberg relation \eqref{eq:Heis} remains true.

We then define
\begin{equation}\label{eq:18}
  L_n \defeq - \frac{\ve_1\ve_2}4 \sum_m \normal{P^-_m P^-_{n-m}}
   - \frac{n+1}2 ({\ve_1+\ve_2}) 
   P^-_n.
\end{equation}
These $L_n$ satisfy the Virasoro relation
\begin{equation}\label{eq:28}
   [L_m, L_n] = (m-n) L_{m+n} + \left(1 + \frac{6(\ve_1+\ve_2)^2}{\ve_1\ve_2}\right)
   	\delta_{m,-n}\frac{m^3-m}{12}
\end{equation}
with the central charge $1 + {6(\ve_1+\ve_2)^2}/{\ve_1\ve_2}$.

The vacuum vector $|\mathrm{vac}\rangle = 1_{\Gi[2]{0}}\in
H^{[*]}_\TT(\Gi[2]{d})$ is a {\it highest weight vector\/}, it is
killed by $L_n$ ($n>0$) and satisfies
\begin{equation}\label{eq:29}
   L_0 |\mathrm{vac}\rangle = -\frac14 \left(\frac{(a_1-a_2)^2}{\ve_1\ve_2} - 
   	\frac{(\ve_1+\ve_2)^2}{\ve_1\ve_2} \right)|\mathrm{vac}\rangle.
\end{equation}

Here we have used the normal ordering $\normal{\ }$, which is defined
by moving all annihilation operators to the right. See
\cite[Def.~9.34]{Lecture} for more detail.

It is known that the Fock space, as a representation of the Virasoro algebra, irreducible if its highest weight is generic. Moreover its isomorphism class is determined by its highest weight (and the central charge).

Looking at the above formula for the highest weight, we see that it is
unchanged under the exchange $a_1\leftrightarrow a_2$. Therefore there
exists the unique automorphism on $F^-$ (over $\bF_T$) sending
$|\mathrm{vac}\rangle$ to itself, and intertwining $L_n$ defined for
$(a_1,a_2)$ with $L_n$ for $(a_2,a_1)$. It is called the {\it
  reflection operator}.

Now a fundamental observation due to Maulik-Okounkov is the following result:

\begin{Theorem}[\protect{\cite[Th.~14.3.1]{MO}}]\label{thm:reflection}
 The $R$-matrix is $\id_{F^\Delta}\otimes(\text{\rm reflection operator})$.
\end{Theorem}

We will give a sketch of proof of \thmref{thm:reflection} in
\subsecref{subsec:R_minimal} after recalling the relation between
Virasoro algebra and cohomology groups of Hilbert schemes in the next
subsection.

\subsection{Virasoro algebra and Hilbert schemes}

Historically the first link between the Virasoro algebra and
cohomology groups of instanton moduli spaces was found for the rank
$1$ case by Lehn \cite{Lehn}. Lehn's result holds for an arbitrary
nonsingular complex quasiprojective surface, but let us specialize to
the case $\CC^2$.

Recall the tautological bundle $\cV$ over the Hilbert scheme
$(\CC^2)^{[d]}$ (\subsecref{subsec:taut}). We consider its first Chern
class $c_1(\cV)$.

Let $P_n$ denote the Heisenberg operator $P^\Delta_n(1)$ for the $r=1$
case.

\begin{Theorem}[\cite{Lehn}]\label{thm:Lehn}
    We have
    \begin{multline*}
        c_1(\cV) \cup \bullet 
        =  -\frac{(\ve_1\ve_2)^2}{3!} \sum_{m_1+m_2+m_3=0}
        \normal{P_{m_1} P_{m_2} P_{m_3}}
        \\
        - \frac{\ve_1\ve_2(\ve_1+\ve_2)}4 \sum_m (|m|-1)
        \normal{P_{-m} P_m}.
    \end{multline*}
    \begin{equation*}
    \end{equation*}
\end{Theorem}

Taking the commutator with $P_n$, we have
\begin{equation*}
    [c_1(\cV)\cup\bullet, P_n]
    = \frac{n\ve_1\ve_2}2 \sum_{l+m=n} \normal{P_l P_m}
    - \frac{n(|n|-1)}2 (\ve_1+\ve_2) P_n.
\end{equation*}
If we compare this with \eqref{eq:18}, we find that this looks very
similar to $n L_n$, except a mysterious expression $|n|$.

A different proof, which works only for $\CC^2$, was given in
\cite{more}. Relation with Jack polynomials is explained there.

\subsection{$R$-matrix at the minimal element}\label{subsec:R_minimal}

Let us give a sketch of the proof of \thmref{thm:reflection}. 

The first step is to determine the classical $r$-matrix, which is the
second coefficient of $R$:
\[
    R = 1 + \frac{(\ve_1+\ve_2)}{a_1 - a_2} r + O((a_1-a_2)^{-2}).
\]

Recall we have $\bigoplus H^{[*]}_\TT((\Gi[2]{d})^T) \cong
F^\Delta\otimes F^-$ such that the $R$-matrix is of a form
$\id_{F^\Delta}\otimes R'$ for some operator $R'$ on
$F^-\otimes_{\bA_T}\bF_T$. In order to determine an operator of this
form, it is enough to determine the matrix element for the component
$(\CC^2)^{[0]}\times (\CC^2)^{[d]}$ in $(\Gi[2]{d})^T$. (See
\cite[Lem.~12.4.2]{MO}.) This is the minimum component with respect to
the partial order in \defref{def:order}. By the triangularity of the
stable envelope, the matrix element is explicitly written in terms of
$e(L_\alpha^\pm)$. The vector bundles $L_\alpha^+$, $L_\alpha^-$ are
the tautological vector bundle $\cV$, tensored with appropriate
characters. Here $(\CC^2)^{[0]}$ is a point, so $(\CC^2)^{[0]}\times
(\CC^2)^{[d]} = (\CC^2)^{[d]}$.
Since the rank of $\cV$ is $d$, which is $-\ve_1\ve_2 \sum_{n>0}
P_{-n} P_n$ (the leading part of $L_0$), we determine the classical
$r$-matrix as
\begin{equation*}
    r = - \ve_1\ve_2 \sum_{n>0} P^-_{-n} P_n^-.
\end{equation*}
(See \cite[Th.~12.4.4]{MO}.)

Next we consider the first Chern class $c_1(\cV)$ of the tautological
bundle over $\Gi[2]{d}$. Recall the formula of the coproduct $\Delta$
on Heisenberg operator $P^\Delta_{-m}(\alpha)$ in
\propref{prop:12.2.1}. We have $\Delta P^\Delta_{-m}(\alpha) =
P_{-m}(\alpha)\otimes 1 + 1\otimes P_{-m}(\alpha)$ for rank $2$.
The underlying geometric reason of this formula
is that $P^\Delta_{-m}(\alpha)$ is given by the lagrangian
correspondence.
Note $c_1(\cV)$ is the fundamental class of the diagonal cut out by
$c_1(\cV)$ as a correspondence. It is not lagrangian, as it is cut.
But one can show that $\Delta c_1(\cV) - c_1(\cV)\otimes 1 - 1\otimes
c_1(\cV)$ involves \emph{only} the classical $r$-matrix. (See
\cite[\S10.1.2]{MO}.)
From this together with the above formula of the classical $r$-matrix,
one calculate $\Delta c_1(\cV) - c_1(\cV)\otimes 1 - 1\otimes
c_1(\cV)$. Then we combine it with the formula of $c_1(\cV)\otimes 1 -
1\otimes c_1(\cV)$ in \thmref{thm:Lehn}. The formula is given in
\cite[Th.~14.2.3]{MO}, but only its restriction on $F^-$ is necessary
for us.

Let us consider
the commutator $[\Delta c_1(\cV), P_n^\Delta]$ acting on $\bigoplus
H^{[*]}_\TT((\Gi[2]{d})^T)\otimes_{\bA_T}\bF_T \cong
F^\Delta_\loc\otimes F^-_\loc$, where the subscript `$\loc$' means
$\otimes_{\bA_T}\bF_T$. We take its restriction to $1\otimes
F^-_\loc\cong F^-_\loc$, composed with the projection
$F^\Delta_\loc\otimes F^-_\loc\to 1\otimes F^-_\loc$. It is an operator
on $F^-_\loc$. Let us denote it by $[\Delta c_1(\cV),
P_n^\Delta]|_{F^-_\loc}$. Then the formula \cite[Th.~14.2.3]{MO} gives us

\begin{Theorem}\label{thm:14.2.3}
    We have $[\Delta c_1(\cV), P_n^\Delta]|_{F^-_\loc} = n L_n$, where
    $L_n$ is given by \eqref{eq:18}.
\end{Theorem}

Since $R$ intertwines $\Delta c_1(\cV)$ and $\Delta^{\text{op}}
c_1(\cV)$ by its definition, the above implies that 
it intertwines Virasoro operators.

Note that $F^-_\loc$ is characterized in $F^\Delta_\loc\otimes
F^-_\loc$ as the intersection of kernels of $P^\Delta_m$ for $m > 0$.
This subspace is nothing but $\bigoplus_d
IH^{[*]}_\TT(\Uh[\SL(2)]{d})\otimes_{\bA_T}\bF_T$ by
Exercise~\ref{ex:trivial}. Thus $\bigoplus_d
IH^{[*]}_\TT(\Uh[\SL(2)]{d})\otimes_{\bA_T}\bF_T$ is a module of the
Virasoro algebra, which is the $\scr W$-algebra associated with $\g =
\algsl(2)$. This is the first case of the AGT correspondence mentioned
in Introduction.

Moreover the operator $[\Delta c_1(\cV), P_n^\Delta]$ is well-defined
on \emph{non}-localized equivariant cohomology
$IH^{[*]}_\TT(\Uh[\SL(2)]{d})$ if $P_n^\Delta = P_n^\Delta(1)$ is
replaced by $\ve_1\ve_2 P_n^\Delta = P_n^\Delta([0])$ for $n > 0$. This is the starting point of our discussion on integral forms. See \subsecref{subsec:integr-form-viras}.

\begin{Remark}
    For quiver varieties, we have tautological vector bundles $\cV_i$
    associated with vertexes $i$. The formula of $\Delta c_1(\cV_i)$
    is given in the same way (see \cite[Th.~10.1.1]{MO}). On the other
    hand, the coproduct of the Yangian $Y(\g)$ associated with a
    finite dimensional simple Lie algebra $\g$ has an explicit formula
    for the first Fourier mode of fields corresponding to Chevalley
    generators of $\g$.
    It is given in terms of the classical $r$-matrix as in the
    geometric construction, where $r$ is the invariant bilinear form.
    The constant Fourier modes are primitive, and there are no known
    explicit formula for second or higher Fourier modes.

    When $c_1(\cV_i)$ is identified with $h_{i,1}$ for type $ADE$
    quiver varieties, the first Fourier mode of the field
    corresponding to a Cartan element $h_i$, geometric and algebraic
    coproducts coincide. Since $h_{i,1}$ together with constant modes
    generates $Y(\g)$, two coproducts are equal \cite{MR3192991}.
\end{Remark}

\section{Perverse sheaves on instanton moduli
  spaces}\label{lec:perv-sheav-uhlenb}

We now turn to $\Uh{d}$ for general $G$.


\subsection{Hyperbolic restriction on instanton moduli spaces}

Let $\rho\colon\CC^\times\to T$ be a one parameter subgroup. We have associated Levi and parabolic subgroups
\begin{equation*}
        L 
        = G^{\rho(\CC^\times)},
\qquad
        P 
        = \left\{ g\in G \middle| \exists \lim_{t\to 0}\rho(t) g \rho(t)^{-1}
        \right\}.
\end{equation*}
Unlike before, here we allow nongeneric $\rho$ so that
$G^{\rho(\CC^\times)}$ could be different from $T$. This is not an
actual generalization. We can replace $T$ by $Z(L)^0$, the connected
center of $L$. Then $\rho$ above can be considered as a generic one
parameter subgroup in $Z(L)^0$.

We consider the induced $\CC^\times$-action on $\Uh{d}$. Let us
introduce the following notation for the diagram \eqref{eq:13}:
\begin{equation}\label{eq:17}
    \Uh[L]{d} \defeq (\Uh{d})^{\rho(\CC^\times)}
    \overset{p}{\underset{i}{\leftrightarrows}} 
    \Uh[P]{d}\defeq \cA_{\Uh{d}}
    \overset{j}{\rightarrow} \Uh{d}.
\end{equation}

Let us explain how these notation can be justified. 
If we restrict our concern to the open subscheme $\Bun{d}$, a framed
$G$-bundle $(\mathcal F,\varphi)$ is fixed by $\rho(\CC^\times)$ if
and only if we have an $L$-reduction $\mathcal F_L$ of $\mathcal F$
(i.e., $\mathcal F = \mathcal F_L\times_L G$) so that $\mathcal
F_L|_{\linf}$ is sent to $\linf\times L$ by the trivialization
$\rho$. Thus $(\Bun{d})^{\rho(\CC^\times)}$ is the moduli space of
framed $L$-bundles, which we could write
$\Bun[L]{d}$.\footnote{\label{fnt:restriction}Here we use the
  assumption $G$ is of type $ADE$. The instanton number is defined via
  an invariant bilinear form on $\g$. For almost simple groups, we
  normalize it so that the square length of the highest root $\theta$
  is $2$. If $G$ is of type $ADE$, instanton numbers are preserved for
  fixed point sets, but not in general. See
  \cite[\S2.1]{2014arXiv1406.2381B}.}
The definition of Uhlenbeck partial compactification is a little
delicate, and is defined for almost simple groups.
Nevertheless it is still known that $\Uh[L]{d}$ is homeomorphic to the
Uhlenbeck partial compactification for $[L,L]$ when it has only one
simple factor (\cite[Prop.~4.2.5]{2014arXiv1406.2381B}), though we do
not know they are the same as schemes or not.
We will actually use this fact later, therefore the same notation for
fixed point subschemes and genuine Uhlenbeck partial compactifications
are natural for us.

On the notation $\Uh[P]{d}$: If we have a framed $P$-bundle $(\mathcal
F_P,\varphi)$, the associated framed $G$-bundle $(\mathcal F_P\times_P
G, \varphi\times_P G)$ is actually a point in the attracting set
$\cA_{\Uh{d}}$. Thus the moduli space of framed $P$-bundle
$\Bun[P]{d}$ is an open subset in $\cA_{\Uh{d}}$.
This is the reason why we use the notation $\Uh[P]{d}$. However a
point in $\Uh[P]{d}\cap\Bun{d}$ is not necessarily coming from a
framed $P$-bundle like this. See Exercise~\ref{ex:ext} below.
Nevertheless we believe that it is safe to use the notation
$\Uh[P]{d}$, as we never consider genuine Uhlenbeck partial
compactifications for the parabolic subgroup $P$.

\begin{Example}
    If $G = \SL(r)$ and $L = S(\GL(r_1)\times \GL(r_2))$, $\Bun[L]{d}$
    is the moduli space of pairs of framed vector bundles
    $(E_1,\varphi_1)$, $(E_2,\varphi_2)$. On the other hand,
    $\Bun[P]{d}$ is the moduli space of vector bundles $E$ which arise
    as extension $0\to E_1 \to E\to E_2\to 0$. (In this situation, one
    can show that $E$ determines $E_1$, $E_2$.)
\end{Example}

\begin{Exercise}[cf.\ an example in \protect{\cite[\S4.4]{2014arXiv1406.2381B}}]\label{ex:ext}
    Consider the case $G=\SL(r)$. Suppose $(E,\varphi)$ is a framed
    vector bundle which fits in an exact sequence $0\to E_1\to E\to
    E_2\to 0$ compatible with the framing. Here we merely assume
    $E_1$, $E_2$ are torsion-free sheaves. Are $E_1$, $E_2$ locally
    free ? Give a counter-example.
\end{Exercise}

\begin{Definition}
Now we consider the hyperbolic restriction functor $p_* j^!\colon
D^b_{\TT}(\Uh{d})\to D^b_{\TT}(\Uh[L]{d})$ and denote it by
$\Phi^P_{L,G}$. If groups are clear from the context, we simply denote
it by $\Phi$.
\end{Definition}

We have the following associativity of hyperbolic restrictions.

\begin{Proposition}\label{prop:trans}
  Let $Q$ be another parabolic subgroup of $G$, contained in $P$ and
  let $M$ denote its Levi subgroup. Let $Q_L$ be the image of $Q$ in
  $L$ and we identify $M$ with the corresponding Levi group.
  Then we have a natural isomorphism of functors
    \begin{equation*}
       \Phi_{M, L}^{Q_L} \circ \Phi_{L,G}^P \cong \Phi_{M,G}^Q.
    \end{equation*}
\end{Proposition}

\begin{proof}
    It is enough to show that
    \begin{equation*}
        \label{eq:2a}
            \UhP{d} \times_{\UhL{d}} \Uh[{Q_L}]{d} = \Uh[Q]{d}.
    \end{equation*}
    This is easy to check. See \cite[Prop.~4.5.1]{2014arXiv1406.2381B}.
\end{proof}

\subsection{Exactness}\label{subsec:exactness}

For a partition $\lambda$, let $S_\lambda\CC^2$ be a stratum of a
symmetric product as before. Let $\Stab(\lambda) = S_{\alpha_1}\times
S_{\alpha_2}\times\dots$ if we write $\lambda =
(1^{\alpha_1}2^{\alpha_2}\dots)$. We consider an associated covering
\begin{equation*}
    (\CC^2)^{\alpha_1}\times (\CC^2)^{\alpha_2}\times
  \cdots \setminus\text{diagonal}
  \to
  S_{\lambda}(\CC^2).
\end{equation*}
Let $\rho$ be a simple local system over $S_\lambda\CC^2$
corresponding to an irreducible representation of $\Stab(\lambda)$.

We consider the following class of perverse sheaves:
\begin{Definition}
    Let $\Perv(\Uh[L]{d})$ be the additive subcategory of the abelian
    category of semisimple perverse sheaves on $\Uh[L]{d}$, consisting of
    finite direct sum of $\IC(\overline{\Bun[L]{d'}\times
    S_\lambda(\CC^2)},1\boxtimes\rho)$ for various $d'$, $\lambda$,
    $\rho$.
\end{Definition}

Here we consider the stratification of $\Uh[L]{d}$ as in \eqref{eq:7}:
\begin{equation*}
    \Uh[L]{d} = \bigsqcup_{d=|\lambda|+d'} \Bun[L]{d'} \times S_{\lambda}(\CC^2).
\end{equation*}
It is the restriction of the stratification \eqref{eq:7} to $\Uh[L]{d}$.

Let us explain why we need to consider nontrivial local systems, even
though our primary interest will be on $\IC(\Uh{d})$: When we analyze
$\IC(\Uh{d})$ through hyperbolic restriction functor, $\IC$ sheaves
for nontrivial local systems occur. This phenomenon can be seen for
type $A$ as follows.

Let us take the Gieseker space $\Gi{d}$ and consider the hyperbolic
restriction for a chamber $\cC$ for the $T$-action. By
\corref{cor:hyp-perverse}(2), we have $p_* j^! \pi_*(\cC_{\Gi{d}})
\cong \pi^T_*(\cC_{(\Gi{d})^T})$, where $(\Gi{d})^T$ is the fixed
point set and $\pi^T\colon (\Gi{d})^T\to (\Uh[r]{d})^T$ is the
restriction of $\pi$. The fixed point sets are given by Hilbert
schemes and symmetric products, and $\pi^T$ factors as
\begin{multline*}
    (\Gi{d})^T = 
    \displaystyle{\bigsqcup_{d_1+\dots+d_r = d} 
        (\CC^2)^{[d_1]}\times\cdots\times (\CC^2)^{[d_r]}}
\\
      \xrightarrow{\pi\times\cdots\times\pi}
      \bigsqcup S^{d_1}\CC^2\times\cdots\times S^{d_r}\CC^2
      \xrightarrow{\kappa}
        S^d\CC^2 = (\Uh[r]{d})^T,
\end{multline*}
where $\kappa$ is the `sum map', defined by $\kappa(C_1,\cdots,C_r) =
C_1+\dots+C_r$ if we use the `sum notation' for points in symmetric
products, like $x_1 + x_2 + \dots + x_d$. The pushforward for the
first factor $\pi\times\cdots\times\pi$ is
\begin{equation*}
    \bigoplus_{|\lambda^1|+\dots+|\lambda^r|=d}
    \cC_{\overline{S_{\lambda^1}(\CC^2)}}\boxtimes\cdots\boxtimes
    \cC_{\overline{S_{\lambda^r}(\CC^2)}}
\end{equation*}
by discussion in \secref{sec:inters-cohom-group}, where
$\lambda^1$, \dots, $\lambda^r$ are partitions. Therefore we do not
have nontrivial local systems. But $\kappa$ produces nontrivial local
systems. Since $\kappa$ is a finite morphism, in order to calculate
its pushforward, we need to study how $\kappa$ restricts to covering
on strata. For example, for $\lambda = (1^r)$, $d_1=d_2=\cdots=d_r =
1$, it is a standard $S_r$-covering
$(\CC^2)^r\setminus\text{diagonal}\to S_{(1^r)}(\CC^2)$.

We have the following

\begin{Theorem}
    $\Phi^P_{L,G}$ sends $\Perv(\Uh{d})$ to $\Perv(\Uh[L]{d})$.
\end{Theorem}

By the remark after \thmref{thm:Braden} we know that $\Phi^P_{L,G}$
send $\Perv(\Uh{d})$ to semisimple complexes. From the factorization,
it is more or less clear that they must be direct sum of shifts of
simple perverse sheaves in $\Perv(\Uh[L]{d})$. Therefore the actual
content of this theorem is the $t$-exactness, that is {\it shifts are
  unnecessary.} For type $A$, it is a consequence of
\corref{cor:hyp-perverse}.

For general $G$, we use hyperbolic semi-smallness
(\thmref{thm:hypsemismall}).
The detail is given in \cite[Appendix~A]{2014arXiv1406.2381B}. We
sketch another shorter proof added at a revision, which was suggested
by a referee of \cite{2014arXiv1406.2381B}.
By a recursive nature and the factorization property of instanton
moduli spaces, it is enough to estimate dimension of the extreme
fibers, i.e., $p^{-1}(d\cdot 0)$, $p_-^{-1}(d\cdot 0)$ in
\eqref{eq:est}. Considering the stratification \eqref{eq:6}, we see
that enough to estimate dimensions of their intersections with
$\Bun{d}$.
Furthermore using the associativity of hyperbolic
restriction (\propref{prop:trans}), it is enough to prove the case
$L=T$. In fact, we can further reduce to the case of the hyperbolic
restriction for the larger torus $T\times\CC^\times_{\rm hyp}$. The
associativity (\propref{prop:trans}) remains true for the larger
torus. Moreover the fixed point set is the single point $d\cdot 0$ as
$(\CC^2)^{\CC^\times_{\rm hyp}} = \{0\}$.
We then consider the case $G=\GL(r)$. The attracting and repelling
sets in the symplectic resolution $\Gi{d}$ are lagrangian, as we
explained in \subsecref{subsec:leaves}.
For general $G$, we take a faithful embedding $G\to \GL(r)$. We have
the induced embedding $\Bun[G]{d}\to \Bun[\GL(r)]{d'}$ which respects
symplectic structures. Therefore the intersections of $p^{-1}(d\cdot
0)$, $p_-^{-1}(d\cdot 0)$ with $\Bun[G]{d}$ are isotropic, hence at
most half dimensional.

\subsection{Calculation of the hyperbolic restriction}
\label{subsec:calc-hyperb-restr}


Our next task is to compute $\Phi_{L,G}^P(\IC(\Uh{d}))$. We have two
most extreme simple direct summands in it:
\begin{aenume}
      \item $\IC(\Uh[L]{d})$,
      \item $\cC_{S_{(d)}(\CC^2)}$.
\end{aenume}
Other direct summands are basically products of type (a) and (b). Let
us first consider (a). Let us restrict the diagram \eqref{eq:17} to
$\Bun[L]{d}$.  Note that a point in $p^{-1}(\Bun[L]{d})$ is a genuine
bundle, cannot have singularities, as singularities are equal or
increased under $p$. Thus the diagram sit in moduli spaces of genuine
bundles as
\begin{equation*}
    \Bun[L]{d} \overset{p}{\underset{i}{\leftrightarrows}} 
    \Bun[P]{d}
    \overset{j}{\rightarrow} \Bun{d}.
\end{equation*}
Then $\Bun[P]{d}$ is a vector bundle over $\Bun[L]{d}$. This can be
seen as follows. For $\mathcal F\in\Bun{d}$, the tangent space
$T_{\mathcal F}\Bun{d}$ is $H^1(\proj^2,\g_{\mathcal F}(-\linf))$, where
$\g_{\mathcal F}$ is the associated bundle $\mathcal F\times_G\g$. If $\mathcal F\in\Bun[L]{d}$, we have the decomposition
\begin{multline*}
    H^1(\proj^2,\g_{\mathcal F}(-\linf))
    \cong 
    H^1(\proj^2,{\mathfrak l}_{\mathcal F}(-\linf))
    \oplus
    H^1(\proj^2,{\mathfrak n}^+_{\mathcal F}(-\linf))
    \oplus
    H^1(\proj^2,{\mathfrak n}^-_{\mathcal F}(-\linf)),
\end{multline*}
according to the decomposition $\g = \mathfrak l\oplus \mathfrak
n^+\oplus \mathfrak n^-$. Here $\mathfrak n^+$ is the nilradical of
$\mathfrak p$, and $\mathfrak n^-$ is the nilradical of the opposite
parabolic.
The first summand gives the tangent bundle of $\Bun[L]{d}$. Thus the
normal bundle is given by sum of the second and third
summands. Moreover, the second and third summands are $\Leaf$ and
$\Leaf^-$ respectively.

Therefore we have the Thom isomorphism
\(
    p_* j^! \cC_{\Bun{d}} \cong \cC_{\Bun[L]{d}}.
\)
It extends to a canonical isomorphism
\begin{equation*}
    \Hom(\IC(\overline{\Bun[L]{d}}), p_* j^! \IC(\Uh{d})) \cong \CC.
\end{equation*}
In other words, we have $p^{-1}(y_\beta)\cap X_0$ in
\thmref{thm:hypsemismall} is the fiber of the vector bundle
$\Bun[P]{d}\to\Bun[L]{d}$, hence we have the canonical isomorphism
$H_{\dim X-\dim Y_\beta}(p^{-1}(y_\beta\cap X_0)_\chi\cong \CC$, given
by its fundamental class.

On the other hand, the multiplicity of the direct summand
$\cC_{S_{(d)}(\CC^2)}$ of type (b) has no simple description. 
It is the top degree homology group of the space $p^{-1}(d\cdot
0)\cap\Bun[G]{d}$ by \thmref{thm:hypsemismall}, hence has a base given
by irreducible components. But it is hard to calculate the base.
Therefore let us introduce the following space:
\begin{equation*}
    U^d \equiv U^{d}_{L,G} \defeq
    \Hom(\cC_{S_{(d)}(\CC^2)}, \Phi^P_{L,G}(\IC(\Uh{d}))).
\end{equation*}
Thus $U^d\otimes \cC_{S_{(d)}(\CC^2)}$ is the isotropic component of
$\Phi^P_{L,G}(\IC(\Uh{d}))$ for $\cC_{S_{(d)}(\CC^2)}$.

From the factorization, we have the canonical isomorphism
\begin{equation*}
    \Phi^P_{L,G}(\IC(\Uh{d}))
    \cong \bigoplus_{d=|\lambda|+d'}
    \IC(\overline{\Bun[L]{d'}\times S_\lambda(\CC^2)},
    (U^1)^{\otimes \alpha_1}\otimes
    (U^2)^{\otimes \alpha_2}\otimes\cdots),
\end{equation*}
where $\lambda = (1^{\alpha_1}2^{\alpha_2}\dots)$ and $(U^1)^{\otimes
  \alpha_1}\otimes (U^2)^{\otimes \alpha_2}\otimes\cdots$ is a
representation of $\Stab(\lambda) = S_{\alpha_1}\times
S_{\alpha_2}\times\cdots$.  Though $(U^1)^{\otimes \alpha_1}\otimes
(U^2)^{\otimes \alpha_2}\otimes\cdots$ may not be irreducible, it is
always semisimple. We understand the associated $\IC$-sheaf as the
direct sum of $\IC$-sheaves of simple constitutes. 

If we take the global cohomology group, contributions of nontrivial
local systems vanish. (See \cite[Lemma~4.8.7]{2014arXiv1406.2381B}.)
Therefore (after taking an isomorphism
$H^*_\TT(\overline{S_\lambda\CC^2})\cong \bA_T$) we get
\begin{equation}\label{eq:24}
    \begin{split}
        & \bigoplus_d H^*_\TT(\Phi^P_{L,G}(\IC(\Uh{d}))
    \cong \bigoplus_d H^*_\TT(\IC(\overline{\Bun[L]{d}}))\otimes_\CC
    \operatorname{Sym}
    \left(U^1\oplus U^2\oplus\cdots\right),
    \\
    & \bigoplus_d H^*_\TT(\Phi^B_{T,G}(\IC(\Uh{d}))
    \cong 
    \bA_T\otimes_\CC 
    \operatorname{Sym}
    \left(U^1\oplus U^2\oplus\cdots\right),
    \end{split}
\end{equation}
where the lower isomorphism is the special case of the upper one as
$\Bun[T]{d} = \emptyset$ for $d\neq 0$ and is a point for $d=0$.
\begin{NB}
    The notation $IH^*_\TT(\Uh[L]{d})$ may be confusing, as it is not
    $IH^*_\TT(S^d(\CC^2)) = H^*_\TT(S^d(\CC^2))$ for $L=T$.
\end{NB}

The first result on $U^d$ is 
\begin{Lemma}[\protect{\cite[Lemma~4.8.11]{2014arXiv1406.2381B}}]\label{lem:calc-Ud}
\begin{equation*}
    \dim U^d = \rank G - \rank [L,L],
\end{equation*}
in particular $\dim U^d = \rank G$ if $L=T$.    
\end{Lemma}

This result was proved as
\begin{enumerate}
      \item a reduction to the case $L=T$ via the associativity
    (\propref{prop:trans}),
      \item a reduction to a computation of the ordinary restriction
    to $S_{(d)}\AA^2$ thanks to a theorem of Laumon \cite{Laumon-chi},
    and
      \item the computation of the ordinary restriction to
    $S_{(d)}\AA^2$ in \cite[Theorem~7.10]{BFG}.
\end{enumerate}

Instead of explaining the detail of this argument, let us give two
explanations of $\dim U^d = \rank G$ for $L=T$ in the next two
subsections. In fact, we will also see that there is an isomorphism
$U^d \cong \mathfrak h$, the Cartan subalgebra of $G$. (This will be
actually one of our goal.)

\subsection{Hyperbolic restriction for Gieseker partial compactification}
\label{subsec:hyperb-restr-gies}

Let us suppose $G = \SL(r)$ and consider the hyperbolic restriction of
$\pi_*(\cC_{\Gi{d}})$. Let us suppose $L=T$ for brevity. We introduce
\begin{equation*}
    V^d \defeq \Hom(\cC_{S_{(d)}(\CC^2)}, \Phi^B_{T,G}(\pi_*(\cC_{\Gi{d}})))
\end{equation*}
as an analog of $U^d$. As in \eqref{eq:24}, we have
\begin{equation*}
    \bigoplus_d H^*_\TT(\Phi^B_{T,G}(\pi_*(\cC_{\Gi{d}})))
    \cong \bA_T\otimes_\CC\operatorname{Sym}
    \left(V^1\oplus V^2\oplus\cdots\right).
\end{equation*}

By the stable envelope, we have an isomorphism
\(
   \Phi^B_{T,G}(\pi_*(\cC_{\Gi{d}})) \cong \pi^T_*(\cC_{(\Gi{d})^T})
\)
(see \corref{cor:hyp-perverse}).
We decompose $\pi^T_*(\cC_{(\Gi{d})^T})$ according to the the
description \eqref{eq:2} of $(\Gi{d})^T$. One can easily check that a
direct summand has nonzero contribution to $V^d$ only when $d_1$,
\dots, $d_r$ are zero except one. That is $(d_1,\dots,d_r) =
(1,0,\dots,0)$, $(0,1,0,\dots,0)$, or $(0,\dots,0,1)$. Moreover
$\Hom(\cC_{S_{(d)(\CC^2)}}, \pi_*(\cC_{(\CC^2)^{[d]}}))$ is naturally
isomorphic to the fundamental class $[\pi^{-1}(d\cdot 0)]$ of the
fiber of $\pi$ at $d\cdot 0\in S_{(d)}(\CC^2)$. In particular, we have
a base $\{ e_1,\dots, e_r\}$ corresponding to summands
$(d_1,\dots,d_r) = (1,0,\dots,0)$, $(0,1,0,\dots,0)$, or
$(0,\dots,0,1)$, and $V^d$ is $r$-dimensional.

Moreover in the decomposition of $\pi_*(\cC_{\Gi{d}})$ in
\eqref{eq:23}, a direct summand contributes to $V^d$ only if $d' = d$
(and $\lambda=\emptyset$) or $\lambda = (d)$ (and $d' = 0$). Therefore
we have a natural direct sum decomposition
\begin{equation*}
    V^d = U^d \oplus \CC[\pi^{-1}(d\cdot 0)],
\end{equation*}
where $d\cdot 0$ is considered as a point in the stratum
$\Bun[\SL(r)]{0}\times S_{(d)}(\CC^2)$. In particular, we obtain
$\dim U^d = r - 1 = \rank \SL(r)$ as expected.

Note also that the class $[\pi^{-1}(d\cdot 0)]$ is $P^\Delta_{-d}([0])
|\mathrm{vac}\rangle = \ve_1\ve_2
P^\Delta_{-d}|\mathrm{vac}\rangle$. By \propref{prop:12.2.1} it is
equal to $e_1 + \dots + e_r$. Since $U^d$ corresponds to the subspace
killed by the Heisenberg operators, we also see
\begin{equation*}
    U^d \cong \left\{ a_1 e_1 + \dots + a_r e_r
      \middle| \sum a_i = 0 \right\}.
\end{equation*}
Thus $U^d$ is isomorphic to the Cartan subalgebra of $\algsl(r)$.
Furthermore for $r=2$, we have $F^-_{\rm{loc}} =
\bF_T\otimes_\CC\operatorname{Sym}(U^1\oplus U^2\oplus \cdots)$, where
$F^-_{\rm{loc}}$ is the Fock space of the Heisenberg operator $P_n^- =
P_n^{(1)} - P_n^{(2)}$ considered in the previous lecture. In fact, it
will be naturally considered as the Fock space for $\ve_1\ve_2 P_n^-$,
as $[\pi^{-1}(d\cdot 0)]$ is a cohomology class with compact support,
when we will consider integral forms in \secref{sec:W}.

\subsection{Affine Grassmannian for an affine Kac-Moody group}
\label{subsec:affine-grassm-an-affine}

In this subsection, we drop the assumption that $G$ is of type $ADE$.

Recall that it is proposed in \cite{braverman-2007} that $G$-instanton
moduli spaces on $\RR^4/\ZZ_\ell$ plays a role of affine Grassmannian
for an affine Kac-Moody group, as we mentioned in Introduction. Let us
give an interpretation of $\dim U^d = \rank G$ in this framework.

We restrict ourselves to the case $\ell = 1$. The main conjecture
in \cite{braverman-2007} is proved in this case.
Moreover there are no differences on instanton moduli spaces for
various $G$ with the common Lie algebra over $\RR^4$. Therefore we
expect that corresponding representations are of the Langlands dual
affine Lie \emph{algebra}. Let us denote it by
$\g_{\text{aff}}^\vee$. If $G$ is of type $ADE$, then it is the
untwisted affine Lie algebra of $\operatorname{Lie}G$. If $G$ is of
type $BCFG$, it is the Langlands dual of the untwisted affine Lie
algebra of $\operatorname{Lie}G$, hence the notation is
reasonable. Here the Langlands dual is just given by reversing arrows
in the Dynkin diagram. Concretely it is a twisted affine Lie algebra
given by the following table:

\begin{table*}
	\centering
    \begin{tabular}[h]{c|c}
        $G$ & the affine Lie algebra $\g_{\text{aff}}^\vee$ \\
        \hline
        $X_r$ ($X = ADE$) & $X^{(1)}_r$ \\
        $B_r$ & $A^{(2)}_{2r}$ \\
        $C_r$ & $D^{(2)}_{r+1}$ \\
        $F_4$ & $E_6^{(2)}$ \\
        $G_2$ & $D_4^{(3)}$
    \end{tabular}
\end{table*}
Here we follow \cite{Kac} for the notation of Lie algebras.

Since we are considering $\ell = 1$, our $\Uh{d}$ should correspond to
a level $1$ representation of $\g_{\text{aff}}^\vee$. It is known that all level
$1$ representations are related by Dynkin diagram automorphisms for
our choice of $\g_{\text{aff}}^\vee$ (e.g., rotations for $A^{(1)}_r$)
\cite[(12.4.5)]{Kac}. Therefore we could just take the $0$th
fundamental weight $\Lambda_0$ without loss of generality. In fact, we
do not have a choice of a highest weight in $\Uh{d}$, so the geometric
Satake could not make sense otherwise.

Let us consider $L(\Lambda_0)$, the irreducible representation of
$\g_{\text{aff}}^\vee$ with highest weight $\Lambda_0$.
The weight corresponding to $\Uh{d}$ is $\Lambda_0 - d\delta$. We only
have discrete parameter $d$ in the geometric side. Fortunately all
dominant weights of $L(\Lambda_0)$ are of this form, and the geometric
Satake for affine Kac-Moody group is formulated only for dominant
weights when \cite{braverman-2007} was written.

By \cite[Prop.~12.13]{Kac} we have
\begin{equation*}
    \operatorname{mult}_{L(\Lambda_0)}(\Lambda_0 - d\delta)
    = p^A(d),
\end{equation*}
where
\[
\sum_{d\ge 0} p^A(d)q^d = \prod_{n\ge 1} (1 -
q^n)^{-\operatorname{mult}n\delta}.
\]
Moreover $\operatorname{mult}n\delta = r$ if $\g_{\text{aff}}^\vee =
X^{(1)}_{r}$ for type $ADE$.
Next suppose $\g_{\text{aff}}^\vee$ is the Langlands dual of $X^{(1)}_{r}$
where $X$ is of type $BCFG$. Let $r^\vee$ be the lacing number of
$\g_{\text{aff}}^\vee$, i.e., the maximum number of edges connecting two
vertexes of the Dynkin diagram of $\g_{\text{aff}}^\vee$ (and
$X^{(1)}_{r}$). Then $\operatorname{mult}n\delta = r$ if $n$ is
a multiple of $r^\vee$, and equals to the number of long simple roots in
the finite dimensional Lie algebra $X_r$ otherwise. Explicitly
$r-1$ for $B_r$, $1$ for $C_r$ and $G_2$, and $2$ for
$F_4$. See \cite[Cor.~8.3]{Kac}.

Let us turn to the geometric side.
As we have explained in \subsecref{subsec:hyperb-restr-affine}, the
hyperbolic restriction with respect to a maximal torus corresponds to
a weight space. Our $T$ is a maximal torus of $G$, but not of the
affine Kac-Moody group. So let us consider a larger torus $\widetilde
T = T\times \CC^\times_{\rm hyp}$, where the second $\CC^\times_{\rm
  hyp}$ is a subgroup $(t,t^{-1})$ of $\CC^\times\times\CC^\times$
acting on $\CC^2$ preserving the symplectic form. There is only one
fixed point (for each $\Uh{d}$) with respect to $\widetilde T$, as
we have remarked at the end of \subsecref{subsec:exactness}.
Let us denote the corresponding hyperbolic restriction by
$\widetilde\Phi$. It is hyperbolic semi-small. (In fact, we proved the
property first for $\widetilde\Phi$, as we mentioned before.)
Therefore we expect
\begin{equation*}
    \widetilde\Phi(\IC(\Uh{d}))
\end{equation*}
is naturally isomorphic to a weight space of a level $1$
representation of $\g_{\text{aff}}^\vee$.
In our case, we have already computed this space implicitly in
\subsecref{subsec:calc-hyperb-restr} for type $ADE$. It is equal to
the degree $d$ part of the right hand side of \eqref{eq:24}. Otherwise
solve Exercise~\ref{ex:hyperb-example}.
When $G$ is of type $BCFG$, we have a different behavior because of
the mismatch of instanton numbers under the restriction, explained in
footnote~\ref{fnt:instanton_number}. One can check that it matches
with the above character formula.

For type $ADE$, the representation $L(\Lambda_0)$ has an explicit
realization by vertex operators (Frenkel-Kac construction)
\cite[\S14.8]{Kac}. The underlying vector space is given by
\begin{equation*}
    \operatorname{Sym}(\bigoplus_{d>0} z^{-d}\otimes\mathfrak h)
    \otimes \CC[Q],
\end{equation*}
where $\mathfrak h$ (resp.\ $Q$) is the Cartan subalgebra (resp.\ the
root lattice) of $\operatorname{Lie}G$. Therefore it is natural to
identify $U^d$ with $z^{-d}\otimes\mathfrak h$.


Let us remark that a conjecture (which is a theorem for $\ell=1$) in
\cite{braverman-2007} involves $IH^*(\Uh{d})$ instead of
$\widetilde\Phi(\IC(\Uh{d}))$. Two spaces have the same dimension
thanks to \cite{Laumon-chi} mentioned above, but the information of
the grading in $IH^*(\Uh{d})$ is lost in
$\widetilde\Phi(\IC(\Uh{d}))$. Therefore our formulation here is
simpler than \cite{braverman-2007}.

\subsection{Heisenberg algebra representation on localized equivariant
  cohomology}\label{subsec:localized}

Let us fix a Borel $B$, and consider the parabolic subgroups
$P_i\supset B$ for each vertex $i$ of the Dynkin diagram of $G$. Let
$L_i$ be the Levi factor.
By the associativity of the hyperbolic restriction
(\propref{prop:trans}), we have a natural decomposition
\begin{equation*}
   U^d_{T,G} \cong U^d_{T,L_i} \oplus U^d_{L_i,G}. 
\end{equation*}
Since $[L_i,L_i] \cong \SL(2)$, we apply the construction in
\subsecref{subsec:hyperb-restr-gies} to get a Fock space
representation on
$\bF_T\otimes_\CC\operatorname{Sym}(U^1_{T,L_i}\oplus
U^2_{T,L_i}\oplus\dots)$. Let us denote the Heisenberg operator by
$P^i_m$. We extend it to
$\bF_T\otimes_\CC\operatorname{Sym}(U_{T,G}^1\oplus
U_{T,G}^2\oplus\dots)$ by letting act trivially on $U^d_{L_i,G}$. So
we have $\rank G$ copies of Heisenberg generators action on
$\bF_T\otimes_\CC\operatorname{Sym}(U_{T,G}^1\oplus U_{T,G}^2\oplus\dots)$.

\begin{Proposition}
    We have
    \begin{equation}\label{eq:27}
        [P^i_m, P^j_n] = -m\delta_{m,-n}(\alpha_i,\alpha_j)\frac1{\ve_1\ve_2},
    \end{equation}
    where $\alpha_i$ is the simple root of $G$ corresponding to
    $P_i$. Therefore
    \begin{equation}
        \label{eq:32}
        \bigoplus_d H^*_\TT(\Phi^B_{T,G}(\IC(\Uh{d})))\otimes_{\bA_T}\bF_T
        \cong 
        \bF_T\otimes_\CC\operatorname{Sym}
        \left(U_{T,G}^1\oplus U_{T,G}^2\oplus\cdots\right)
    \end{equation}
    is naturally the Fock representation of the Heisenberg algebra
    associated with the Cartan subalgebra $\mathfrak h$ of $\g =
    \operatorname{Lie}G$.
\end{Proposition}

\begin{proof}
Let us consider the commutator $[P^i_m, P^j_n]$. If $i=j$, it is the
$\SL(2)$-case. Hence it is just given by $[P^{(1)}_m - P^{(2)}_m,
P^{(1)}_n - P^{(2)}_n]$ in the notation in
\subsecref{subsec:R-matrix-as-a-Virasoro}.

For two distinct vertexes $i$, $j$, we consider the corresponding
parabolic subgroup $P_{i,j}$ with Levi factor $L_{i,j}$. We have
\(
   U^d_{T,G} \cong U^d_{T,L_{i,j}} \oplus U^d_{L_{i,j},G}.
\)
Since $[L_{i,j},L_{i,j}]$ is either $\SL(3)$ or $\SL(2)\times \SL(2)$.
For $\SL(2)\times\SL(2)$, the commutator is clearly $0$.
For $\SL(3)$, we consider $\bigoplus H^{[*]}_\TT((\Gi[3]{d})^T)$ as
the tensor product of three copies of Fock space as in
\subsecref{subsec:R-matrix-as-a-Virasoro}. Then we have three
Heisenberg operators $P_m^{(1)}$, $P_m^{(2)}$, $P_m^{(3)}$ and
$P^\Delta_m = P^{(1)}_m + P^{(2)}_m + P^{(3)}_m$. It is also clear
that $P^i_m = P^{(1)}_m - P^{(2)}_m$, $P^j_m = P^{(2)}_m - P^{(3)}_m$
by a consideration of the associativity of the hyperbolic restriction.

Hence we get \eqref{eq:27}. In fact, there is a delicate point
here. We need to choose a polarization for instanton moduli spaces so
that it is compatible with a polarization for $\Gi[3]{d}$ for each
$L_{i,j}$. See \cite[\S6.2]{2014arXiv1406.2381B} for detail.

Finally the size of \eqref{eq:32} is the same as that of the Fock
representation by \lemref{lem:calc-Ud}. Since the Fock representation
is irreducible, we conclude \eqref{eq:32} is the Fock representation.
\end{proof}

\section{\texorpdfstring{${\scW}$}{W}-algebra representation on equivariant intersection cohomology groups}
\label{sec:W}

The goal of this final lecture is to explain a construction of
$\scW$-algebra representation on $\bigoplus_d
IH^{*}_\TT(\Uh{d})$. In fact, if we assume level is \emph{generic},
or take $\otimes_{\bA_T}\bF_T$ in equivariant cohomology, we have
already achieved it. This is because the $\mathscr W$-algebra is known
to be a (vertex) subalgebra of the Heisenberg (vertex) algebra at
generic level by Feigin-Frenkel (see \cite[Ch.~15]{F-BZ}). So we just
restrict the Heisenberg algebra representation on $\bigoplus_d
H^*_\TT(\Phi^B_{T,G}(\IC(\Uh{d})))$ to $\scW$, and use the
localization theorem
$H^*_\TT(\Phi^B_{T,G}(\IC(\Uh{d})))\otimes_{\bA_T}\bF_T \cong
IH^*_\TT(\Uh{d})\otimes_{\bA_T}\bF_T$.  But this is not an interesting
assertion, as it does not explain any geometric meaning of the
$\scW$-algebra.

One way to explain such a meaning is to prove the Whittaker conditions
are satisfied for fundamental classes $[\Uh{d}]\in IH^0_\TT(\Uh{d})$
as in \cite[\S8]{2014arXiv1406.2381B}. We will go a half way towards
this goal, namely we will explain a construction of a representation
of an integral form of $\scW$-algebra on $\bigoplus_d
IH^*_\TT(\Uh{d})$. The remaining half requires an introduction of
generating fields (denoted by $W_i$ in \cite{F-BZ} and by
$\widetilde{W}^{(\kappa)}$ in \cite{2014arXiv1406.2381B}), which is
purely algebraic, and hence is different from the theme of lectures.

We think that the integral form itself is an important object, as we
can recover arbitrary level (and highest weight) representation by a
specialization $\bA_T\to \CC$. Thus highest weights and level are
identified with equivariant variables.
This will be supposed to be discussed in future.

The $\scW$-algebra is defined by the quantum Drinfeld-Sokolov
reduction, that is the cohomology of the so-called BRST complex
associated with an affine Lie algebra \cite[Ch.~15]{F-BZ}. This
definition is algebraic, and will not be recalled here. The integral
form mentioned above is also defined by the BRST complex. Therefore we
will treat its characterization as a subalgebra of the Heisenberg
algebra (\thmref{thm:FF}) as a black box, and regard that it is
defined in this way.

\subsection{Four types of nonlocalized equivariant cohomology groups}

We consider cohomology groups with compact and arbitrary support with
coefficients in $\IC(\Uh{d})$ and its hyperbolic restriction
$\Phi_{T,G}(\IC(\Uh{d}))$. So we consider four types of nonlocalized
equivariant cohomology groups. They are related by natural adjunction
homomorphisms as
\begin{equation}\label{eq:26}
    \begin{multlined}[.8\textwidth]
    \bigoplus_d IH^*_{\TT,c}(\Uh{d}) \to
    \bigoplus_d H^*_{\TT,c}(\Phi^B_{T,G}(\IC({\Uh{d}})))
\\
    \to
    \bigoplus_d H^*_{\TT}(\Phi^B_{T,G}(\IC({\Uh{d}})))
    \to \bigoplus_d IH^*_{\TT}(\Uh{d}).
    \end{multlined}
\end{equation}
(See \eqref{eq:25}.) One can show that they are free $\bA_T$-modules,
and homomorphisms are inclusion, which become isomorphisms over $\bF_T$
(\cite[\S6]{2014arXiv1406.2381B}).

Moreover, we have a natural Poincar\'e pairing between the first and
fourth cohomology groups. When we compare it with Kac-Shapovalov form
so that it is a pairing between the Verma module with highest weight
$\boldsymbol a$ and the dual of the Verma module with highest weight
$-\boldsymbol a$, we need to make a certain twist of this pairing. In
particular, the pairing is sesquilinear in equivariant variables. See
\cite[\S6.8]{2014arXiv1406.2381B}. Let us ignore this point, as we
will not discuss details of highest weights. The pairing between the
second and third can be defined if we replace one of the hyperbolic
restriction by the opposite one $\Phi^{B_-}_{T,G}$ as $D p_* j^! = p_!
j^* D = p_{-*}j_-^!$ by \thmref{thm:Braden}. This is compensated by
the twist above, hence we have a pairing between the second and third.

Let $\widetilde{P}^i_n = \ve_1\ve_2 P^i_n = P^i_n([0])$ be the
Heisenberg operator associated with the fundamental class of the
origin $0\in\CC^2$. This operator is always well-defined on the middle
two cohomology groups in \eqref{eq:26}. It satisfies
\begin{equation*}
    [\widetilde{P}^i_m, \widetilde{P}^i_n] =
    -m\delta_{m,-n}(\alpha_i,\alpha_j)\ve_1\ve_2.
\end{equation*}
Note that the right hand side is a polynomial in $\ve_1$, $\ve_2$
contrary to \eqref{eq:27}.
The pairing is invariant.

Let $\Heis_\bA(\mathfrak h)$ denote the algebra over $\bA \defeq
\CC[\ve_1,\ve_2]$ generated by $\widetilde{P}^i_n$ with the above
relations. It is an $\bA$-form of the Heisenberg algebra.
Two middle cohomology groups in \eqref{eq:26} are $\Heis_\bA(\mathfrak
h)$-modules dual to each other.

Let us look at \eqref{eq:24}. The isomorphism uses the identification
$H^*_\TT(\overline{S_\lambda\CC^2})\cong\bA_T$, which is given by the
fundamental class $[\overline{S_\lambda\CC^2}]$. It is not compatible
with $\widetilde{P}^i_n$, but
$H^*_{\TT,c}(\overline{S_\lambda\CC^2})\cong\bA_T$ does as it is given
by the fundamental class $[0]$. Therefore we should consider the
cohomology group with compact support. We get

\begin{Proposition}
    $\bigoplus_d H^*_{\TT,c}(\Phi^B_{T,G}(\IC({\Uh{d}})))$ is a highest
    weight module of $\Heis_\bA(\mathfrak h)$ with the highest weight
    vector $|\mathrm{vac}\rangle$, i.e.,
    \begin{equation*}
        \bigoplus_d H^*_{\TT,c}(\Phi^B_{T,G}(\IC({\Uh{d}})))
        \cong \Heis_\bA(\mathfrak h) |\mathrm{vac}\rangle.
    \end{equation*}
\end{Proposition}

\subsection{Integral form of Virasoro algebra}\label{subsec:integr-form-viras}

Consider the case $G=\SL(2)$. Recall that we can normalize Virasoro
operators so that they act on non-localized equivariant cohomology as
we have discussed after \thmref{thm:14.2.3}. More precisely, as we set
$\widetilde{P}^i_n = \ve_1\ve_2 P^i_n$, it is natural to introduce
$\widetilde{L}_n = \ve_1\ve_2 L_n$. Then it is a well-defined operator
on $IH^{*}_{\TT}(\Uh[\SL(2)]{d})$ and $IH^{*}_{\TT,c}(\Uh[\SL(2)]{d})$
by \thmref{thm:14.2.3}.

The relation \eqref{eq:28} is modified to
\begin{equation*}
    [\widetilde{L}_m, \widetilde{L}_n] = 
    \ve_1\ve_2\left\{ (m-n) \widetilde{L}_{m+n} + 
      \left(\ve_1\ve_2 + 6(\ve_1+\ve_2)^2\right)
      \delta_{m,-n}\frac{m^3-m}{12}\right\}.
\end{equation*}
The relation is defined over $\CC[\ve_1,\ve_2]$. Therefore we have the
integral form of the Virasoro algebra defined over $\CC[\ve_1,\ve_2]$.
Let us denote it by $\Vir_\bA$. We have an embedding
$\Vir_\bA\subset\Heis_\bA(\mathfrak h_{\algsl_2})$, where $\mathfrak
h_{\algsl_2}$ is the Cartan subalgebra of $\algsl_2$.

\begin{Proposition}
    $IH^{*}_{\TT,c}(\Uh[\SL(2)]{d})$ is a highest weight
    module with the highest weight vector $|\mathrm{vac}\rangle$.
\end{Proposition}

We already know that $|\mathrm{vac}\rangle$ is killed by
$\widetilde{L}_n$. The highest weight is computed in \eqref{eq:29}.
Hence we need to check 
\(
    IH^{*}_{\TT,c}(\Uh[\SL(2)]{d}) = \Vir_\bA  |\mathrm{vac}\rangle.
\)
This is done by comparing graded dimensions of both sides
(\cite[Prop.~8.1.11]{2014arXiv1406.2381B}). (This argument works for general $G$.)

We call $IH^{*}_{\TT,c}(\Uh[\SL(2)]{d})$ the \emph{universal Verma
  module}, as it is specialized to the Verma module in the usual sense
by a specialization ${\bA_T}\to \CC$. Then
$IH^{*}_{\TT}(\Uh[\SL(2)]{d})$ is the dual universal Verma module.

\subsection{Integral form of \texorpdfstring{${\scW}$}{W}-algebra}
\label{subsec:integr-form}
Let us take the parabolic subgroup $P_i$ and its Levi factor $L_i$ as
in \subsecref{subsec:localized}. Then we have the corresponding
integral Virasoro algebra as a subalgebra in $\Heis_\bA(\mathfrak
h)$. Let us denote it by $\Vir_{i,\bA}$. Let
$\Heis_\bA(\alpha_i^\perp)$ denote the integral Heisenberg algebra for
the root hyperplane $\alpha_i = 0$ corresponding to $L_i$. Then
$\Vir_{i,\bA}\otimes_\bA \Heis_\bA(\alpha_i^\perp)$ is an
$\bA$-subalgebra of $\Heis_\bA(\mathfrak h)$. More precisely we need
to consider them as \emph{vertex algebras}, but we ignore this
point. Then an integral version of Feigin-Frenkel's result is

\begin{Theorem}[\protect{\cite[Th.~B.6.1]{2014arXiv1406.2381B}}]\label{thm:FF}
    Let $\scW_\bA(\g)$ be an $\bA$-form of the $\scW$-algebra defined
    by an $\bA$-form of the BRST complex. Then
    \begin{equation*}
        \scW_\bA(\g) \cong \bigcap_i \Vir_{i,\bA}\otimes_\bA
        \Heis_\bA(\alpha_i^\perp),
    \end{equation*}
    where the intersection is taken in $\Heis_\bA(\mathfrak h)$.
\end{Theorem}

The original version states the same result is true for \emph{generic}
level, and can be deduced from above by taking $\otimes_\bA
\CC(\ve_1,\ve_2)$. The level $k$ in the usual approach is related to
$\ve_1$, $\ve_2$ by the formula
\begin{equation*}
    k + h^\vee = -\frac{\ve_2}{\ve_1},
\end{equation*}
where $h^\vee$ is the dual Coxeter number of $\g = \operatorname{Lie}G$.

As mentioned above, we will not review the definition of the BRST
complex here. For our purpose, it is enough to consider the right hand side of the isomorphism in \thmref{thm:FF} as a definition of $\scW_\bA(\g)$.

\begin{Remark}
    Two equivariant variables $\ve_1$, $\ve_2$ are symmetric in the
    geometric context. On the other hand, $\ve_1$ and $\ve_2$ play
    very different role in the $\bA$-form on the BRST complex
    above. On the other hand, $\Vir_\bA$ and $\Heis_\bA$ have symmetry
    $\ve_1\leftrightarrow \ve_2$. Therefore the theorem implies the
    symmetry on $\scW_\bA(\g)$. This symmetry is nontrivial, and is
    called the Langlands duality of the $\scW$-algebra. (If $\g$ is
    not of type $ADE$, $\g$ must be replaced by its Langlands dual, as
    the Heisenberg commutation relation involves the inner product
    $(\alpha_i,\alpha_j)$.) See \cite[Ch.~15]{F-BZ}.

    Note also the symmetry is apparent in the geometric context, but
    it is so only for type $ADE$, and is not clear for other
    types. (To realize $\scW_\bA(\g)$, one should use instanton moduli
    spaces for \emph{twisted} affine Lie algebras. Then the roles of
    $\ve_1$, $\ve_2$ are asymmetric.)
\end{Remark}

\subsection{Intersection cohomology with compact support as the
  universal Verma module of the \texorpdfstring{${\scW}$}{W}-algebra}
\label{subsec:inters-cohom-with}

We use the associativity of the hyperbolic restriction as in
\subsecref{subsec:localized}. Combining \thmref{thm:FF} with the
result in \subsecref{subsec:integr-form-viras}, we find that
\begin{equation*}
    \bigoplus_d \bigcap H^*_{\TT,c}(\Phi^{P_i}_{L_i,G}(\IC({\Uh{d}})))
\end{equation*}
is a module of $\scW_\bA(\g)$. Here the intersection is taken in
$\bigoplus_d H^*_{\TT,c}(\Phi^{B}_{T,G}(\IC({\Uh{d}})))$. One can check
that this intersection is equal to $IH^*_{\TT,c}(\Uh{d})$. (See the proof of \cite[Prop.~8.1.7]{2014arXiv1406.2381B}.) Therefore

\begin{Theorem}[\protect{\cite[\S8.1]{2014arXiv1406.2381B}}]\label{thm:main}
    $\bigoplus_d IH^*_{\TT,c}(\Uh{d})$ is a $\scW_\bA(\g)$-module. It
    is a highest weight module with the highest weight vector
    $|\mathrm{vac}\rangle$.
\end{Theorem}

We do not review the highest weight, but it is given by an explicit
formula in equivariant variables as in \eqref{eq:29}. We call
$\bigoplus_d IH^*_{\TT,c}(\Uh{d})$ the universal Verma module as for
the Virasoro algebra.

\section{Concluding remarks}

\subsection{AGT correspondence}\label{subsec:AGT}

As mentioned in Introduction, the AGT correspondence is based on a
hypothetical $6$-dimensional quantum field theory. This theory is
rather difficult to justify its existence even in a physical level of
rigor, as its lagrangian description is unknown. There are surveys on
the AGT correspondences and other related subjects by a group of
physicists \cite{MR3410144}, and another by Tachikawa written for
mathematicians \cite{Tach-review}. Let us try to give a very short
summary of what the author has understood from \cite{Tach-review}. We
will also omit several important points, say topological twists, to
make it in a reasonable size, hence the reader should read it with
care, and consult \cite{Tach-review} for corrections.

The $6$-dimensional quantum field theory is associated with a Dynkin
diagram $\Gamma$ of type $ADE$. Let us denote it by $S_\Gamma$. It
gives a number $Z_{S_\Gamma}(X^6)$ for a compact Riemannian
$6$-dimensional manifold $X$, called the \emph{partition function}. If
$X^6$ has a boundary $\partial X^6$, $Z_{S_\Gamma}(X^6)$ depends on a
Riemannian metric on $\partial X^6$. Hence $Z_{S_\Gamma}(X^6)$ is a
function on the space of Riemannian metric on $\partial X^6$.
It is not an arbitrary function, and is an element of the \emph{space
  of states} or \emph{quantum Hilbert space} $Z_{S_\Gamma}(Y^5)$
associated with a compact $5$-manifold $Y^5 = \partial X^6$. The
reader should be familiar with Atiyah's axiomatic approach to
topological quantum field theories, and the theory $S_\Gamma$ should
rise a similar structure. A function may be multi-valued, or a section
of a line bundle, but we ignore this point, as we will consider it as
an element of $Z_{S_\Gamma}(Y^5)$.

For an application to our study of instanton moduli spaces, we take
$\RR^4\times C$ as a $6$-manifold $X$, where $C$ is a Riemann
surface. We endow $\RR^4$ with the $\CC^\times\times\CC^\times$-action
as in the main body. It can be regarded as a `skeleton' of a
Riemannian metric on $\RR^4$ appeared above. In particular, the
partition function $Z_{S_\Gamma}(\RR^4\times C)$ depends on the
equivariant variable $\ve_1$, $\ve_2$.
On the other hand, as a function on $C$, $Z_{S_\Gamma}(\RR^4\times C)$
depends only on a conformal structure on $C$, hence $C$ is a Riemann
surface instead of a Riemannian $2$-manifold.

We now view $Z_{S_\Gamma}(\RR^4\times C)$ as a partition function
$Z_{S_\Gamma[\RR^4]}(C)$ for a $2$-dimensional conformal field theory
$S_\Gamma[\RR^4]$. The latter $S_\Gamma[\RR^4]$ is called the
\emph{dimension reduction} of $S_\Gamma$ by $\RR^4$, and physicists
accept that this procedure from the $6d$ theory to a $2d$ theory is
possible.

A care is required as $\RR^4$ is noncompact. We need to specify a
point in the \emph{moduli space of vacua}, or rather the partition
function is a function on the moduli space of vacua. It is a finite
dimensional manifold associated with a $4$-dimensional quantum field
theory $S_\Gamma[C]$ obtained as the dimension reduction of $S_\Gamma$
by $C$. Unfortunately there is no mathematically rigorous definition
of the moduli space of vacua for an arbitrary quantum field theory,
but for $S_\Gamma[C]$, it is believed to be an affine space, something
like the base of Hitchin integrable system on $C$ associated with the
group $G$\footnote{Tachikawa told me that we need to specify a group
  upon a reduction.} whose Dynkin diagram is $\Gamma$, or its
modification. It should be also emphasized that the partition function
may have poles, and is defined only \emph{locally} on the moduli space
of vacua.

When $C$ has punctures, fields are allowed to have singularities
there.
We need to specify a type of singularities, which is explained in
\cite[\S3.5]{Tach-review}. It is very roughly corresponds to one for
the moduli space of Higgs bundles. Since Hitchin integrable system
appears as the moduli space of vacua, this sounds reasonable. But
Lusztig-Spaltenstein duality is involved, and the actual relation is
quite far from simple. We will not go further detail.

As a $2$-dimensional quantum field theory, $S_\Gamma[\RR^4]$ has the
space of states associated with $S^1$. A fundamental claim is
\begin{equation*}
    Z_{S_\Gamma[\RR^4]}(S^1) = Z_{S_\Gamma}(\RR^4\times S^1)
    = \prod_d IH^*_{T\times\CC^\times\times\CC^\times}(\Uh{d}) \otimes 
    \bF_T.
\end{equation*}
Namely the direct product of localized equivariant intersection
cohomology groups of instanton moduli spaces is the space of states of
$S_\Gamma[\RR^4]$. Its element is a function in equivariant variables
$\ve_1$, $\ve_2$ as well as $\vec{a}$. The latter is related with the
moduli space of vacua above.

Let us take an elliptic curve $E_\tau$ of period $\tau$ as an example
of $C$. Since $E_\tau$ is obtained by a gluing of a cylinder, we have
\begin{equation*}
    Z_{S_\Gamma[\RR^4]}(E_\tau) = 
    \tr_{Z_{S_\Gamma[\RR^4]}(S^1)}(A),
\end{equation*}
where $A$ is an operator on $Z_{S_\Gamma[\RR^4]}(S^1)$ corresponding
to the cylinder. Note that the base of Hitchin integrable system on
$E_\tau$ is regarded as a single point, hence
$Z_{S_\Gamma[\RR^4]}(E_\tau)$ does not have an extra dependence on
points in the moduli space of vacua.
It is believed that the operator $A$ is given by $e^{2\pi\sqrt{-1}\tau
  d}$ on the summand
$IH^*_{T\times\CC^\times\times\CC^\times}(\Uh{d})$. Thus the trace is
equal to $\prod_{n\ge 1} (1-q^n)^{-\rank G}$ with $q=
e^{2\pi\sqrt{-1}\tau}$ by \subsecref{subsec:calc-hyperb-restr}. This
is essentially a power of the Dedekind eta function $\eta$, which
satisfies $\eta(-1/\tau) = (-\sqrt{-1}\tau)^{1/2}\eta(\tau)$. This is
compatible with the above consideration, as $E_\tau$ and $E_{-1/\tau}$
is isomorphic. In fact, the dimension reduction of $S_\Gamma$ by
$E_\tau$ is known to be the $4$-dimensional $N=4$ supersymmetric gauge
theory, whose topological (Vafa-Witten) twist roughly gives generating
functions of Euler numbers of instanton moduli spaces
\cite{Vafa-Witten}. The conjecture that the generating function is a
modular form is explained as a consequence of $E_\tau\cong E_{-1/\tau}$.

The sum of fundamental cycles $\sum_d [\Uh{d}]$ is a vector in
$Z_{S_\Gamma[\RR^4]}(S^1)$. It is believed that it is
$Z_{S_\Gamma[\RR^4]}(C)$ where $C$ is the unit disk with an
\emph{irregular} puncture at the origin.

For type $A_{r-1}$, we can consider the $r$th power $e(\cV)^r$ of the
equivariant Euler class of the tautological vector bundle $\cV$ on
$\Gi{d}$.
It is believed that the corresponding $C$ is the unit disk with two
regular punctures.
Since $\cV$ is of rank $d$, $e(\cV)^r$ has degree $2dr$, which is the
dimension of $\Gi{d}$. Hence $e(\cV)^r$ lives in
$H^{[0]}_{T\times\CC^\times\times\CC^\times}(\Gi{d})$ in our
convention, hence the $r$th power is expected to be natural. We also
observe that $\cV^{\oplus r}$ is $\Ext^1(\mathcal O_{\proj^2}^{\oplus
  r},\mathcal E_d(-\linf))$, where $\mathcal E_d$ is the universal
bundle over $\proj^2\times\Gi{d}$. Considering $\mathcal
O_{\proj^2}^{\oplus r}$ is the special case of $\mathcal E_d$ with
$d=0$,we can more generally consider the class $\Ext^\bullet(\mathcal
E_{d'},\mathcal E_d(-\linf))$ on $\Gi{d'}\times\Gi{d}$. This class for
$r=1$ is studied in \cite{MR2942794}. See also
\cite{2016LMaPh.tmp...58N}.
The corresponding $C$ is a cylinder with
a regular puncture.
We do not know how to generalize these constructions outside type $A$.

Now it comes to a punch line of this story. The AGT correspondence
predicts that 
\begin{quote}
    $Z_{S_\Gamma[\RR^4]}$ is the conformal field theory
    associated with the $\scW$-algebra for $\Gamma$.
\end{quote}
This is an almost mathematically rigorous statement, except that only
the chiral part of a conformal field theory is usually studied as a
vertex algebra in the mathematical community. Nevertheless we can
still derive rigorous statements, for example our result that
$Z_{S_\Gamma[\RR^4]}[S^1]$ is the Verma module of $\scW$ is the very
first consequence. The vector $\sum_d [\Uh{d}] =
Z_{S_\Gamma[\RR^4]}(C)$ is a characterization in terms of $\scW$,
i.e., it is the Whittaker vector. For further results, besides papers
\cite{MR2942794,2016LMaPh.tmp...58N} mentioned above, there are many
papers in physics literature for type $A$, and also for classical
groups. See Tachikawa's survey in \cite{MR3410144}.

\subsection{Post-requisite}

Further questions and open problems are listed in
\cite[\S1.11]{2014arXiv1406.2381B}.

In order to do research in those
directions, the followings would be necessary besides what are explained in this lecture series.

\begin{itemize}
      \item As explained above, the AGT correspondence predicts that
    $Z_{S_\Gamma[\RR^4]}$ is the conformal field theory associated
    with the $\scW$-algebra. In order to understand it in
    mathematically rigorous way, one certainly needs to know the
    theory of vertex algebras (e.g., \cite{F-BZ}). In fact, we still
    lack a fundamental understanding why equivariant intersection
    cohomology groups of instanton moduli spaces have structures of
    vertex algebras. We do want to have an \emph{intrinsic}
    explanation without any computation, like checking Heisenberg
    commutation relations.

      \item The AGT correspondence was originally formulated in terms
    of Nekrasov partition functions. Their mathematical background is
    given for example in \cite{MR2095899}.

      \item In view of the geometric Satake correspondence for affine
    Kac-Moody groups \cite{braverman-2007}, the equivariant
    intersection cohomology group
    $IH^*_{G\times\CC^\times\times\CC^\times}$ of instanton moduli
    spaces of $\RR^4/\ZZ_{\ell}$ should be understood in terms of
    representations of the affine Lie algebra of $\g$ and the
    corresponding generalized $\scr W$-algebra. We believe that
    necessary technical tools are more or less established in
    \cite{2014arXiv1406.2381B}, but it still needed to be worked out
    in detail. Anyhow, one certainly needs knowledge of $\scr
    W$-algebras in order to study their generalization.
\end{itemize}
	
%
%
%
%
%
%
%
\bibspread





\bibliography{nakajima,mybib,tensor2,MO,coulomb}

\def\indexname{List of notations}
\printindex
\end{document}